\documentclass[10pt,a4paper]{amsart}

\usepackage[left=3.3cm,right=3.3cm,top=4cm,bottom=4cm]{geometry}
\usepackage[utf8]{inputenc}
\usepackage{amsmath}
\usepackage{amsfonts}
\usepackage{amssymb}
\usepackage{amsthm}
\usepackage[all]{xy}
\usepackage{color}
\usepackage{hyperref}
\usepackage{xypic}
\usepackage{mathrsfs}
\usepackage{centernot}
\usepackage{tikz-cd}

\theoremstyle{plain}
\newtheorem{theorem}{Theorem}

\newtheorem{lemma}[theorem]{Lemma}

\newtheorem{proposition}[theorem]{Proposition}

\DeclareMathOperator{\Div}{div}

\DeclareMathOperator{\Ext}{\mathrm{Ext}}

\DeclareMathOperator{\Gal}{Gal}

\DeclareMathOperator{\Hom}{Hom}

\DeclareMathOperator{\Pic}{Pic}

\DeclareMathOperator{\Spec}{Spec}

\newcommand{\del}{\partial}
\newcommand{\delbar}{\overline{\partial}}

\newcommand{\IN}{\mathbb{N}}

\newcommand{\IA}{\mathbb{A}}

\newcommand{\IZ}{\mathbb{Z}}

\newcommand{\IQ}{\mathbb{Q}}

\newcommand{\IQbar}{\overline{\mathbb{Q}}}
\newcommand{\IR}{\mathbb{R}}
\newcommand{\IC}{\mathbb{C}}
\newcommand{\Gm}{\mathbb{G}_m}
\newcommand{\IP}{\mathbb{P}}

\newcommand{\vol}{\mathrm{vol}}

\newcommand{\dprime}{{\prime\prime}}

\newcommand{\pr}{\mathrm{pr}}

\newcommand{\caO}{\mathcal{O}}

\newcommand{\id}{\mathrm{id}}

\newcommand{\hhat}{\widehat{h}}

\newcommand{\cOt}{\widetilde{\mathcal{O}}}

\newcommand{\chisup}{\chi_{\mathrm{sup}}}

\newcommand{\pio}{\overline{\pi}}

\newcommand{\Ltil}{\widetilde{L}}

\newcommand{\Mtil}{\widetilde{M}}

\newcommand{\Ntil}{\widetilde{N}}
\newcommand{\Otil}{\widetilde{\mathcal{O}}}
\newcommand{\Ptil}{\widetilde{P}}

\newcommand{\Lo}{\overline{L}}
\newcommand{\Mo}{\overline{M}}
\newcommand{\No}{\overline{N}}

\newcommand{\volh}{\widehat{\mathrm{vol}}}

\begin{document}

\title{Points of Small Height on Semiabelian Varieties}

\begin{abstract} The Equidistribution Conjecture is proved for general semiabelian varieties over number fields. Previously, this conjecture was only known in the special case of almost split semiabelian varieties through work of Chambert-Loir. The general case has remained intractable so far because the height of a semiabelian variety is negative unless it is almost split. In fact, this places the conjecture outside the scope of Yuan's equidistribution theorem on algebraic dynamical systems. To overcome this, an asymptotic adaption of the equidistribution technique invented by Szpiro, Ullmo, and Zhang is used here. It also allows a new proof of the Bogomolov Conjecture and hence a self-contained proof of the Strong Equidistribution Conjecture in the same general setting.
\end{abstract}

\author{Lars K\"uhne}
\thanks{This work was supported by an Ambizione Grant of the Swiss National Science Foundation.}
\email{lars.kuehne@unibas.ch}
\address{Departement Mathematik und Informatik \\
Spiegelgasse 1 \c
4051 Basel \\
Switzerland}


\maketitle

Throughout this article, $G$ denotes an arbitrary semiabelian variety over a number field $K \subset \IQbar$ with maximal subtorus $T$ of dimension $t$ and maximal abelian quotient $\pi: G \rightarrow A$ of dimension $g$. For a place $\nu$ of $K$, we denote by $K_\nu$ the associated completion of $K$, by $\overline{K}_\nu$ the algebraic closure of $K_\nu$, and by $\IC_\nu$ the completion of the algebraic closure $\overline{K}_\nu$. For a quasi-projective algebraic variety $X$ over a number field $K$ and a place $\nu$ of $K$, we denote by $X_{\IC_\nu}^{\mathrm{an}}$ the $\IC_\nu$-analytic space associated with $X_{\IC_\nu}$. If $\nu$ is archimedean, this means that $X_{\IC_\nu}^{\mathrm{an}}$ is a complex (analytic) space (see \cite{Grauert1994} for this notion). If $\nu$ is non-archimedean, $X_{\IC_\nu}^{\mathrm{an}}$ is a Berkovich $\IC_\nu$-analytic space (see \cite[Section 3.4]{Berkovich1990}).

In order to state our main results, we need a canonical height $\hhat$ on $G$. For details, the reader is referred to Section \ref{section::semiabelian} and \cite[Sections 2 and 3]{Kuehne2017a}. To simplify our exposition, we enlarge $K$ if necessary so that we can assume that $T=\Gm^t$. Then $(\IP^1)^t$ is naturally a $\Gm^t$-equivariant compactification of $\Gm^t$, and each multiplication-by-$n$ map $\Gm^t \rightarrow \Gm^t$ extends to a map $(\IP^1)^t \rightarrow (\IP^1)^t$. This yields a compactification $\overline{G}$ of $G$, and $\pi$ extends to a map $\overline{\pi}: \overline{G} \rightarrow A$, whose fibers are isomorphic to $(\IP^1)^t$. Furthermore, the multiplication-by-$n$ map $[n]: G \rightarrow G$ extends to a map $\overline{[n]}: \overline{G} \rightarrow \overline{G}$. The boundary $(\IP^1)^t \setminus \Gm^t$ gives rise to a Weil divisor on $\overline{G}$. Letting $M_{\overline{G}}$ denote the line bundle associated with this divisor, we have $\overline{[n]}^\ast M_{\overline{G}} = M_{\overline{G}}^{\otimes n}$. In addition, we fix an ample symmetric line bundle $N$ on $A$ and set $L= M_{\overline{G}} \otimes \pio^\ast N$. Tate's limit argument allows us to define a unique canonical height $\hhat_{L}(x)$ for each closed point $x \in G$, starting from the Weil heights of $M_{\overline{G}}$ and $\overline{\pi}^\ast N$.

This already suffices to define the main object of our study. Recall that a sequence $(x_i) \in G^\IN$ of \textit{closed} points is said to be \textit{generic} (resp.\ \textit{strict}) if none of its subsequences is contained in a proper algebraic subvariety (resp.\ a proper algebraic subgroup) of $G$. Furthermore, we say that a sequence $(x_i) \in G^{\IN}$ of \textit{closed} points is a sequence of small points if $\hhat_L(x_i) \rightarrow 0$.%\footnote{Formally, this is an abuse of notation because the condition $\hhat_L(x_i) \rightarrow 0$ may depend on the compactification $\overline{T}$ and the line bundles $M$ and $N$. However, any such choice amounts to an equivalent definition by Lemma \ref{lemma::heightcomparison}.}

As in the case of abelian varieties, the following two conjectures convey significant information about the diophantine geometry of semiabelian varieties: For each place $\nu \in \Sigma(K)$, a closed point $x \in G$ yields a $0$-cycle $\mathbf{O}_\nu(x) = (x \otimes_K \IC_\nu)^{\mathrm{an}}$ on the $\IC_\nu$-analytic group $G_{\IC_\nu}^{\mathrm{an}}$ associated with $G$. We write $\delta_y$ for the Dirac measure associated with a point $y \in G_{\IC_\nu}^{\mathrm{an}}$.

\textbf{Equidistribution Conjecture (EC).} \textit{For every generic sequence $(x_i)\in G^\IN$ of small points, the measures $\frac{1}{\# \mathbf{O}_\nu(x_i)}\sum_{y \in \mathbf{O}_\nu(x_i)} \delta_{y}$ converge weakly to the measure $c_1(\Lo_\nu)^{\wedge g+t}/\deg_L(\overline{G})$.}

More explicitly, $\mathrm{(EC)}$ asserts that 
\begin{equation}
\label{equation::equidistribution_explicit}
\frac{1}{\# \mathbf{O}_\nu(x_i)}\sum_{y \in \mathbf{O}_\nu(x_i)} f(y) \longrightarrow \frac{1}{\deg_L(\overline{G})}\int_{G_{\IC_\nu}^{\mathrm{an}}}fc_1(\Lo_\nu)^{\wedge g+t} \quad (i \rightarrow \infty)
\end{equation}
for every compactly supported $f\in \mathscr{C}^0(G_{\IC_\nu}^{\mathrm{an}})$. 

The measure $c_1(\Lo_\nu)^{\wedge g+t}$ arises naturally in a refined approach to the canonical height $\hhat_L$ introduced above (see Sections \ref{section::arithmeticintersectiontheory} and \ref{section::semiabelian} for details). In fact, the line bundle $L$ can be endowed with a canonical $\nu$-metric. This yields a $\nu$-metrized line bundle $\Lo_\nu=(L,\Vert \cdot \Vert_\nu)$. Letting $\nu$ vary over all places of $K$, these canonical metrics combine to an (adelically) metrized line bundle $\Ltil = (L, \{ \Vert \cdot \Vert_\nu \})$. The (adelic) height function $h_{\Ltil}$ associated with $\Ltil$ coincides with the Néron-Tate height $\hhat_L$ from above. Additionally, it enables us to assign a height $h_{\Ltil}(X)$ with every algebraic subvariety $X \subseteq \overline{G}$. Most importantly, each $\nu$-metrized line bundle $\Lo_\nu$ supplies us with a regular Borel measure $c_1(\Lo_\nu)^{\wedge g+t}$ on the analytic space $G^{\mathrm{an}}_{\IC_\nu}$ (see Subsection \ref{section::borelmeasure}).

For an archimedean place $\nu$, it is well-known that $c_1(\Lo_\nu)^{\wedge g+t}$ is a Haar measure on the maximal compact subgroup of $G^{\mathrm{an}}_{\IC_\nu}$. In fact, this is a special case of our Lemma \ref{lemma::measure_realanalytic} below. The maximal compact subgroup can be easily described by ignoring the complex structure on $G^{\mathrm{an}}_{\IC_\nu}$. As a real Lie group, every semiabelian variety is isomorphic to $\IR^{2(g+t)}/\Lambda$ for an arbitrary discrete subgroup $\Lambda \subseteq \IR^{2(g+t)}$ of rank $2g+t$. The maximal compact subgroup of $G^{\mathrm{an}}_{\IC_\nu}$ corresponds then evidently to the $\IR$-linear subspace $\IR \cdot \Lambda \subseteq \IR^{2(g+t)}$. For a non-archimedean place $\nu$, the determination of $c_1(\Lo_\nu)^{\wedge g+t}$ is more intricate since even for abelian varieties the reduction of $G$ with respect to $\nu$ plays a role. A complete description of the abelian case is given by Gubler in \cite[Example 7.2]{Gubler2010}. It seems very likely that his techniques can be also used for general semiabelian varieties. Since this seems, unfortunately, a lengthy distraction from our main investigation, we leave it nevertheless to the interested reader.

%Denote by $t_\nu$ the toric rank of the reduction at $\nu$ of the Néron model of $A$. As in \cite[Section 4]{Gubler2010}, we can find an analytic group $E$ and a discrete subgroup $M \subset E(\IC_\nu)$ such that $E/M = G^{\mathrm{an}}_{\IC_\nu}$ and there is a surjective continuous homomorphism $\val_\nu: G^{\mathrm{an}}_{\IC_\nu} \rightarrow \IR^{t+t_\nu}$. The restriction $\val_\nu \!|_M$ induces an isomorphism between $M$ and $\Lambda=\val_\nu(M)$. In addition, there is a natural homeomorphism $\IR^{t+t_\nu}/\Lambda \approx S(A)$ (the skeleton), deformation retract
%
%\cite[Example 7.2]{Gubler2010}
%
%The argument of \cite[Section 6]{Gubler2010} shows that $c_1(\Lo_\nu)^{\wedge g+t}/L^{g+t}$ is a Haar measure with total volume $1$ on the compact analytic group $\val_\nu^{-1}(0,\dots,0) \subseteq G_{\IC_\nu}^{\mathrm{an}}$. The proof, which is a simple but lengthy adaption of \cite{Gubler2010}, is left to the reader.

\textbf{Bogomolov Conjecture (BC).} \textit{Let $X$ be a geometrically irreducible algebraic subvariety of $G$. Then, either $X$ is the translate of a connected subgroup by a torsion point or there exists some $\varepsilon=\varepsilon(X)>0$ such that
\begin{equation*}
\{ x \in X(\IQbar) \ | \ \hhat_L(x) < \varepsilon \}
\end{equation*}
is not Zariski-dense in $X$.}

Closely related to these two conjectures is a formal strengthening of the first one.

\textbf{Strong Equidistribution Conjecture (SEC).} \textit{For every strict sequence $(x_i)\in G(\IQbar)^\IN$ of small points, the measures $\frac{1}{\# \mathbf{O}_\nu(x_i)}\sum_{y \in \mathbf{O}_\nu(x_i)} \delta_{y}$ converge weakly to $c_1(\Lo_\nu)^{\wedge g+t}/\deg_L(\overline{G})$.}

In fact, it is easy to prove the equivalence $\mathrm{(EC)} \wedge \mathrm{(BC)} \Leftrightarrow \mathrm{(SEC)}$. In the nineties of the last century, considerable efforts were dedicated to prove the above conjectures in various settings. An important special case of $\mathrm{(BC)}$ is the embedding $X=C \hookrightarrow G=\mathrm{Jac}(C)$ of a geometrically irreducible curve $C$ over $K$ into its Jacobian variety $\mathrm{Jac}(C)$. Before being completely settled by Ullmo \cite{Ullmo1998}, this case of $(\mathrm{BC})$ was proven by Szpiro \cite{Szpiro1990} and Zhang \cite{Zhang1993, Zhang1995a} in numerous cases. In the meantime, Zhang proved $(\mathrm{BC})$ for algebraic tori in \cite{Zhang1995}. A complete proof of $(\mathrm{BC})$ for abelian varieties was given by Zhang \cite{Zhang1998}. (The reader may also consult the surveys \cite{Abbes1997, Zhang1998a}.)

For abelian varieties, $(\mathrm{EC})$ was proven in a joint work of Szpiro, Ullmo, and Zhang \cite{Szpiro1997}. Bilu \cite{Bilu1997} proved directly $(\mathrm{SEC})$ for algebraic tori and gave a deduction $(\mathrm{SEC}) \Rightarrow (\mathrm{BC})$. A further advancement was made by Chambert-Loir  \cite{Chambert-Loir2000}, who gave a proof of $(\mathrm{SEC})$ in the case where $G$ is almost split (i.e., $G$ is isogenous to the product of an abelian variety and a torus). Up to the present work, his work has contained the best result in the direction of $(\mathrm{SEC})$. Indeed, the canonical height $h_{\Ltil}(\overline{G})$ is strictly negative unless it is almost split (\cite[Corollaire 4.3]{Chambert-Loir2000}). In addition, all points of negative height lie on the boundary $\overline{G} \setminus G$ (\cite[Lemma 3.9]{Chambert-Loir2000}). This means that \textit{no} generic sequence $(x_i) \in G^\IN$ of closed points can satisfy
\begin{equation}
\label{equation::nonconvergence}
h_{\Ltil}(x_i) \longrightarrow h_{\Ltil}(\overline{G}), \ i \rightarrow \infty,
\end{equation} 
unless $G$ is almost split. Yuan's general equidistribution theorem for algebraic dynamical systems (\cite[Theorem 10.2]{Yuan2012}) is hence empty in this situation as the equidistribution method developed by Szpiro, Ullmo, and Zhang \cite{Szpiro1997} does generally only apply to sequences satisfying (\ref{equation::nonconvergence}). The reader is referred to \cite{Chambert-Loir2000} for details.

Using different methods, David and Philippon \cite{David2000} %\footnote{Up to now (July 2018), the details of their proof have not appeared.} 
proved $(\mathrm{BC})$ for general semiabelian varieties. However, their method seems completely incapable to approach $(\mathrm{EC})$. In this article, we tackle both $(\mathrm{EC})$ and $(\mathrm{BC})$ for general semiabelian varieties with arithmetic intersection theory. We proceed in a way that is surprisingly close to the method of Szpiro, Ullmo, and Zhang in spite of the above-mentioned obstacle. Our main result is as follows.

\begin{theorem} \label{theorem::main} $\mathrm{(SEC)}$ is true for every semiabelian variety $G$ over $\IQbar$ and every strict sequence $(x_i) \in G^\IN$ of small points.
\end{theorem}

As already mentioned, $(\mathrm{SEC})$ is a direct consequence of $(\mathrm{EC})$ and $(\mathrm{BC})$; we refer to Section \ref{section::mainthm} for the deduction of Theorem \ref{theorem::main} from $(\mathrm{EC})$ and $(\mathrm{BC})$. The fastest way to prove Theorem \ref{theorem::main} is hence to prove merely $(\mathrm{EC})$ and to rely on \cite{David2000} for $(\mathrm{BC})$. However, our technique to prove $(\mathrm{EC})$ can be also used to give a new proof of $(\mathrm{BC})$, which is remarkably close to Zhang's proof of $(\mathrm{BC})$ in the case of abelian varieties \cite{Zhang1998}. Consequently, we can give a self-contained and genuinely Arakelov-theoretic proof of Theorem \ref{theorem::main}. This seems worth to afford the detour of proving $(\mathrm{BC})$ anew, and we do so in Section \ref{section::bogomolov} after some preparation in Section \ref{section::equilibrium}.

The centerpiece of our argument is Proposition \ref{proposition1}, which includes $(\mathrm{EC})$ as a special case. Section \ref{section::toricrank1} is completely devoted to its proof. The main idea is rather simple and we describe it next. For this, we exclusively \textit{restrict} ourselves to the case where $G$ is the extension of an abelian variety $A$ \textit{by $T=\Gm$} (i.e., $t=1$ with the above notation). As already mentioned, a semiabelian variety $G$ has canonical height zero if and only if it is almost split, which means here that the associated extension class $\eta \in \Ext^1_{\IQbar}(A,\Gm) = A^\vee(\IQbar)$ is a torsion point. For an ample symmetric line bundle $N$ on $A^\vee$, this is equivalent to the Néron-Tate height $\hhat_N(\eta)$ being zero. One can hence suspect that $\hhat_N(\eta)$ quantifies the obstruction to proving (EC) by means of the standard equidistribution arguments.

The $\IQbar$-isogeny class of $G$ contains semiabelian varieties associated with extension classes $\eta^\prime \in \Ext^1_{\IQbar}(A,\Gm) = A^\vee(\IQbar)$ such that $ \hhat_N(\eta^\prime)$ is arbitrary small. In fact, if $\eta^\prime$ is such that $n \eta^\prime = \eta$ for some positive integer $n$, then $ \hhat_N(\eta^\prime)= n^{-2} \hhat_N(\eta)$.
Writing $G_n$ for the semiabelian variety described by $\eta^\prime$, there is an isogeny $\varphi_n: G_n \rightarrow G$ of degree $n$ (see Section \ref{section::toricrank1}). Additionally, it is not hard to see that we only need to prove $\mathrm{(EC)}$ for a single element in the isogeny class of $G$. It is hence reasonable to replace our original $G$ with some $G_n$, $n \gg 1$, and to hope that this facilitates the proof of $\mathrm{(EC)}$ with the traditional procedure.

If $G$ is not almost split, we deduce from \cite[Th\'eor\`eme 4.2]{Chambert-Loir2000} that
\begin{equation}
\label{equation::introheightdecrease}
-n^{-2}\ll h_{\Ltil_n}(\overline{G}_n) < 0
\end{equation}
for a certain compactification $\overline{G}_n$ of $G_n$ and a certain metrized line bundle $\Ltil_n$ on $\overline{G}_n$. Thus, merely replacing $G$ with a fixed $G_n$, $n \gg 1$, is not sufficient, but working asympotically as $n \rightarrow \infty$ seems prospective. In other words, one should try to carry out the argument of Szpiro, Ullmo, and Zhang \cite{Szpiro1997} for each level $G_n$ and observe what happens as $n \rightarrow \infty$. The integer $n$ is not the only parameter that appears here. As in previous proofs of $(\mathrm{EC})$, a rational\footnote{In the archimedean case (e.g., in \cite{Zhang1998}) this scaling factor can in fact be a real number, but the non-archimedean variant uses a rational number instead (see the proof of \cite[Lemme 3.5]{Chambert-Loir2006}; our $\lambda$ corresponds to the variable $t$ therein).} 
scaling factor $\lambda \rightarrow 0$ comes up. It turns out that there is some interplay between these both parameters: Up to suppressing some easily controllable terms, written $(\dots)$, we obtain an upper bound
\begin{equation}
\label{equation::equidistribution4}
\limsup_{i \rightarrow \infty} \left\vert \frac{1}{\# \mathbf{O}_\nu(x_i)}\sum_{y \in \mathbf{O}_\nu(x_i)} f(y) - \int_{G_{\IC_\nu}^{\mathrm{an}}}fc_1(\Lo_\nu)^{\wedge g+t} \right\vert \ll_{G,f} n|\lambda| + n^{-2}|\lambda|^{-1}+(\dots)
\end{equation}
(compare (\ref{equation::equidistribution3}) below) for every integer $n$ and every real $\lambda \in (0,n^{-1}]$. Choosing $n$ a square number and $\lambda= n^{-3/2}$ yields an upper bound $\ll_{G,f} n^{-1/2}$ so that $(\mathrm{EC})$ follows with $n \rightarrow \infty$. 

To conclude this sketch of our argument, let us briefly discuss the provenience of the two terms on the right-hand side of (\ref{equation::equidistribution4}). The first term stems from the error term in the expansion (compare to \cite[p.\ 162]{Zhang1998} and Lemma \ref{lemma::lambdanestimate} below)
\begin{equation*}
h_{\Ltil_n(\lambda f)}(\overline{G}_n) =
h_{\Ltil_n}(\overline{G}_n) 
+ \lambda \int_{\overline{G}_n(\IC_\nu)} f_n c_1(\Lo_n)^{\wedge g+t} + O_{\overline{G},f, n}(|\lambda|^2), \ f_n = f \circ \varphi_n,
\end{equation*}
for some integer $n \gg 1$. The integer $n$ appears in the implicit constant of $O_{\overline{G},f, n}(|\lambda|^2)$, and we have to render the dependency more precisely. This is done by applying the projection formula to (a compactification of) the isogeny $\varphi_n: G_n \rightarrow G$ of degree $n$. The second term is more or less $|\lambda^{-1}h_{\Ltil}(\overline{G}_n)|$, which is majorized by (\ref{equation::introheightdecrease}). That there is a suitable choice of $n$ and $\lambda$ relies utimately on the fact that $|\lambda^{-1}h_{\Ltil}(\overline{G}_n)|$ decreases faster than $\deg(\varphi_n)$ increases. To make a long story short, the quadraticity of the N\'eron-Tate height on the dual abelian variety $A^\vee$ is played off against the linearity in the toric part of $G$, and quadratic decay prevails over linear growth.

As is well-known (see \cite[Corollary 4]{Zhang1998}), $(\mathrm{SEC})$ implies directly the Manin-Mumford conjecture for semiabelian varieties, which was proven by Raynaud \cite{Raynaud1983,Raynaud1983a} for abelian varieties, by Laurent \cite{Laurent1984} for algebraic tori, and by Hindry \cite{Hindry1988} in our setting. Furthermore, Poonen \cite{Poonen1999} and Zhang \cite[Remark (3) on p.\ 41]{Zhang2000} pointed out that $(\mathrm{SEC})$ and the Mordell-Lang Conjecture $(\mathrm{MLC})$ imply a common generalization. For a finitely generated subgroup $\Gamma$ of $G(\IQbar)$, we set
\begin{equation*}
\Gamma^\prime = \{ x \in G(\IQbar) \ | \ \exists n \in \IZ \setminus \{ 0 \} : n x \in \Gamma \}
\end{equation*}
as well as
\begin{equation*}
\Gamma^\prime(\varepsilon) = \{ x \in G(\IQbar) \ | \ \exists y \in \Gamma^\prime, z \in G(\IQbar): x = y + z, \ \widehat{h}_L(z)\leq \varepsilon \}
\end{equation*}
for any real $\varepsilon>0$.

\textbf{Mordell-Lang plus Bogomolov Conjecture (MLBC).} \textit{For any subvariety $X \subset G$ that is not a translate of a semiabelian subvariety of $G$ by a point in $\Gamma^\prime$, there exists some $\varepsilon>0$ such that $\Gamma^\prime(\varepsilon) \cap X(\IQbar)$ is not Zariski-dense in $X$.}

By the time \cite{Poonen1999} and \cite{Zhang2000} were written, McQuillan \cite{McQuillan1995} had already proven $\mathrm{(MLC)}$ for general semiabelian varieties so that their arguments were only conditional on the then-missing $\mathrm{(SEC)}$ for an archimedean place $\nu$. In \cite{Remond2003}, R\'emond gave a proof of $(\mathrm{MLBC})$ that avoids $(\mathrm{SEC})$ and uses instead his version \cite{Remond2005} of Vojta's inequality for semiabelian varieties as well as $(\mathrm{BC})$ for semiabelian varieties \cite{David2000}. Our Theorem \ref{theorem::main} renders the original proofs of Poonen \cite{Poonen1999} and Zhang \cite{Zhang2000} unconditional, giving an alternative to R\'emond's approach.

Finally, let us remark that the availability of $(\mathrm{EC})$ for general semiabelian varieties also allows to extend Zhang's equidistribution result on almost division points \cite[Theorem 1.1]{Zhang2000} (see Remark (3) on p.\ 41 in \textit{loc.cit.}).

\textbf{Notations and conventions.} \textit{General.}
For two terms $a$ and $b$, we write $a \ll b$ if there exists a positive real number $c$ such that $a \leq c \cdot b$. If $c$ depends on some data, say an algebraic variety $X$, we write $a \ll_X b$ etc. If there is no subscript, the implied constant $c$ is absolute. We use $\gg$ similarly.

\textit{Number fields.} Throughout this article, we let $K \subset \IQbar$ denote a number field with integer ring $\caO_K$, and we set $S=\Spec(\caO_K)$. In addition, $\Sigma_f(K)$ (resp.\ $\Sigma_\infty(K)$) is the set of non-archimedean (resp.\ archimedean) places, and we set $\Sigma(K) = \Sigma_f(K) \cup \Sigma_\infty(K)$. For each $\nu \in \Sigma(K)$, we let $K_\nu$ denote the $\nu$-adic completion of $K$. By $\IC_\nu$ is denoted the completion of an algebraic closure $\overline{K}_\nu$ of $K_\nu$, by $\mathcal{O}_\nu \subset \IC_\nu$ its ring of integers, by $\mathfrak{p}_\nu$ the maximal prime ideal of $\mathcal{O}_\nu$, by $k_\nu$ the residue field $\mathcal{O}_\nu/\mathfrak{p}_\nu$, and by $p_\nu$ the residue characteristic of $\IC_\nu$. For all $\nu \in \Sigma_f(K)$, the absolute value $\vert \cdot \vert_\nu$ on $\IC_\nu$ is normalized such that $\vert p_\nu \vert_\nu = p_\nu^{-[K_\nu:\IQ_p]}$. We use the standard values of $\IR$ and $\IC$ for archimedean places. This normalization leads to an additional factor
\begin{equation*}
\delta_\nu =
\begin{cases}
2 & \text{if $\nu$ is complex archimedean,} \\
1 & \text{otherwise},
\end{cases}
\end{equation*}
in some identities. For a scheme $\mathcal{X}$ over $S$, we denote by $\mathcal{X}_\nu$ the formal completion of the special fiber $\mathcal{X} \times_S \Spec(k_\nu)$ in $\mathcal{X} \times_{S} \Spec(\caO_\nu)$. This is a formal scheme over $\mathrm{Spf}(\caO_\nu)$. 

\textit{Algebraic Geometry (General).} Denote by $k$ an arbitrary field. A \textit{$k$-variety} is a reduced scheme of finite type over $k$. By a \textit{subvariety} of a $k$-variety we mean a closed reduced subscheme. A subvariety is determined by its underlying topological space and we frequently identify both. The tangent bundle of a $k$-variety $X$ is written $TX$ and its fiber over a point $x \in X$ is denoted by $T_x X$. Furthermore, $X^{\mathrm{sm}}$ denotes the smooth locus of $X$. If $X$ is an irreducible $k$-variety, we write $\eta_X$ for its generic point. The unity of a $k$-algebraic group $G$ is written $e_G$.

For a non-negative integer $d$ and a $k$-variety $X$, a \textit{$d$-cycle} on $X$ is a finite formal sum $\sum_{i=1}^{r} n_i [Z_i]$ where each $n_i$ is an integer and each $Z_i$ is a $k$-irreducible subvariety of $X$ having dimension $d$. %For a (not necessarily irreducible) subvariety $Z \subseteq X$ of pure dimension $d$, we also write $[Z]$ for the cycle that is the sum over the irreducible components of $Z$.

For a line bundle $L$ over a general scheme, we denote by $\mathcal{F}(L)$ its sheaf of sections.

\textit{Generic sequences.} Let $X$ be an algebraic $k$-variety. We say that a sequence $(x_i) \in X^\IN$ of closed points is \textit{$X$-generic} if none of its subsequences is contained in a proper algebraic subvariety of $X$. If the variety $X$ can be inferred from context, we simply say \textit{generic} instead of $X$-generic.

\textit{Line bundles and intersection theory.} For line bundles $L_1, L_2, \dots, L_d$ on a proper algebraic variety $X$ of dimension $d$ over a field $k$, we use the intersection numbers
\begin{equation*}
L_1 \cdot L_2 \cdots L_d \in \IZ
\end{equation*}
defined by Kleiman \cite{Kleiman1966} and Snapper \cite{Snapper1960} (see \cite[Section VI.2]{Kollar1996} for a good introduction). These coincide with the numbers
\begin{equation*}
\deg(c_1(L_1) \cap c_1(L_2) \cdots \cap c_1(L_d) \cap [X]) \in \IZ
\end{equation*}
in the terminology of \cite{Fulton1998}. If $\{ M_1, M_2, \dots, M_r \} = \{ L_1, L_2, \dots, L_d \}$ and each $M_i$ occurs $n_i$-times among $L_1,L_2,\dots,L_d$, we set
\begin{equation*}
M_1^{n_1} \cdot M_2^{n_2} \cdots M_r^{n_r} := L_1 \cdot L_2 \cdots L_d;
\end{equation*}
a similar notation is used for the Borel measures defined in Subsection \ref{section::borelmeasure} and the arithmetic intersection numbers defined in Subsection \ref{section::arithmeticintersectionpairing}. Furthermore, we write $\deg_L(X)$ for $L^{d}$. We define the \textit{volume}
\begin{equation*}
\vol (L) = \limsup_{N \rightarrow \infty} \frac{h^0(X,L^{\otimes N})}{N^{\dim(X)}/\dim(X)!}.
\end{equation*}
The group law of Picard groups of line bundles, as well as of their various metrized versions introduced in Section \ref{section::arithmeticintersectiontheory}, is written additively. For an invertible rational section $\mathbf{s}: X \dashrightarrow L$ of a line bundle $L$, its divisor is denoted by $\Div(\mathbf{s})$. The support of a divisor $D$ (resp.\ a cycle $\mathfrak{Z}$) is written $\left\vert D\right\vert$ (resp.\ $\left\vert \mathfrak{Z} \right\vert$).

\textit{Admissible formal schemes and generic fibers.} For each $\nu \in \Sigma_f(K)$, we define \textit{admissible} formal schemes over $\mathrm{Spf}(\caO_\nu)$ as in \cite[2.6]{Gubler2007a}. As there, we assign with an admissible formal scheme $\mathcal{X}_\nu$ over $\mathrm{Spf}(\caO_\nu)$ a Berkovich $\IC_\nu$-analytic space $\mathcal{X}_{\nu,\eta}$, its \textit{generic fiber}.

\textit{Continuity and smoothness.} 
We use $\mathscr{C}^0$ (resp.\ $\mathscr{C}^\infty$) as an abbreviation for continuous (resp.\ smooth). For any topological space $X$, $\mathscr{C}^0(X)$ denotes the continuous functions on $X$ and  $\mathscr{C}^0_c(X)$ the continuous functions on $X$ having compact support.

\textit{Tangent spaces.}
For each differentiable or real-analytic manifold $M$ we denote by $TM$ its tangent bundle. The fiber of $TM$ over $x \in M$ is denoted $T_x M$.

Let $Y$ be a complex manifold (e.g., $(X^{\mathrm{sm}})^{\mathrm{an}}_{\IC_\nu}$ for an algebraic variety $X$ over $K$ and some $\nu \in \Sigma_\infty(K)$). To $Y$ is associated its real tangent bundle $T_{\IR} Y$ and its holomorphic tangent bundle $T^{1,0}_{\IC} Y$ (e.g., $(T X)^{\mathrm{an}}_{\IC_\nu}$ for a smooth complex algebraic variety $X$ and some $\nu \in \Sigma_\infty(K)$). The reader is referred to \cite[Section 0.2]{Griffiths1994} and \cite[Section 1.2]{Huybrechts2005} for details.

\textit{Riemann metrics.} A \textit{Riemannian metric} $g$ on a manifold $M$ is a smooth $\IR$-linear map
\begin{equation*}
g: T_\IR M \otimes T_\IR M \rightarrow \IR \times M
\end{equation*}
of $\IR$-bundles such that $g(t,t) \in \IR^{\geq 0} \times M$ for all $t \in T_{\IR,x} M$. (We usually drop the second factor and write e.g.\ $g(t,t)\geq 0$.) We say that it is non-degenerate if $g(t,t) = \{ 0\} \times M$ is equivalent to $t \in \{0\} \times M$. The volume element $\vol(g)$ associated with a Riemannian metric $g$ on an orientable manifold is defined as usual (see e.g.\ \cite[p.\ 362]{Helgason2001}).

A Riemannian metric on a complex manifold $Y$ is just a Riemannian metric on the underlying real-analytic manifold. With each hermitian metric on $Y$, we can associate a Riemannian metric (as e.g.\ in \cite[pp.\ 361-362]{Helgason2001}).

\textit{Measure Theory.} 
We adhere to the definitions used in \cite[Chapters 1 and 2]{Rudin1987}. The \textit{support} of a measure $\mu$ on $X$ is the set of all points $x \in X$ for which every neighborhood $N$ of $x$ satisfies $\mu(N)>0$.

\textit{Complex spaces.} Let $S$ be a reduced complex (analytic) space. Recall that this means that $S$ is locally biholomorphic to a closed analytic subvariety $V$ in a complex domain $U \subset \IC^n$. A function $f$ on $S$ is \textit{smooth} if, for each such sufficiently small local chart, it is the restriction of a smooth function on $U$. We write $\mathscr{C}^\infty(S)$ for the smooth real-valued functions on $S$. In the same way, we use local charts to define \textit{plurisubharmonic} functions on $S$ as restrictions. 

Similarly, a \textit{$\mathscr{C}^\infty$-form} $\omega$ on $S$ is a differential form on the smooth locus $S^\mathrm{sm}$ of $S$ with the following extension property: $S$ can be covered by local charts $V \subset U \subset \IC^n$ as above such that for each chart the differential form $\omega|_{V^{\mathrm{sm}}}$ is the restriction of a $\mathscr{C}^\infty$-differential form on $U$. There are also well-defined linear operators $d$ and $d^c = i/2\pi(\delbar - \del)$ on the $\mathscr{C}^\infty$-differential forms on $S$. For each local chart $V \subset U \subset \IC^n$, these are simply the restrictions of the operators of the same name on $\IC^n$. Having defined $\mathscr{C}^\infty$-differential forms on $S$, we can define \textit{currents} by duality as in \cite[Définition 1.1]{Demailly1985}.

For each $\mathscr{C}^0$-hermitian metric $\Vert \cdot \Vert$ on a holomorphic line bundle $L$ over $S$, we can define a \textit{Chern current} $c_1(L,\Vert \cdot \Vert)$ in the usual way; if $\mathbf{s}: U \rightarrow L$ is a non-zero section over some open subset $U \subset S$, we set $c_1(L,\Vert \cdot \Vert)|_U= dd^c (-\log \Vert \mathbf{s} \Vert )$. This is a $\mathscr{C}^\infty$-differential form on $S$ if $\Vert \cdot \Vert$ is $\mathscr{C}^\infty$.

A current on $S$ is called \textit{semipositive} here if it is ``(faiblement) positif'' according to \cite[Définition 1.2]{Demailly1985}. For two currents $T_1$ and $T_2$ on $S$, we use $T_1 \geq T_2$ as a shorthand for the statement that $T_1 - T_2$ is semipositive. 

A \textit{$\mathscr{C}^0$-hermitian metric} $\omega$ on $S$ is a $\mathscr{C}^0$-hermitian metric on $T^{1,0}_\IC S^{\mathrm{sm}}$ such that $S$ can be covered as above by local charts $V \subset U \subset \IC^n$ on each of which $\omega|_{V^{\mathrm{sm}}}$ is the restriction of a $\mathscr{C}^0$-hermitian metric on $T^{1,0}_\IC U$. Note that each $\mathscr{C}^0$-hermitian metric $\omega$ on $S$ yields a current $[\omega]$ of bidegree $(1,1)$ on $S$. We say that a current $T$ of bidegree $(1,1)$ is \textit{strictly positive} if there exists a $\mathscr{C}^0$-hermitian metric $\omega$ on $S$ such that $T \geq [\omega]$; we write $T > 0$ in this case.

%In particular, to a regular Borel measure $\mu$ on a compact Hausdorff space $X$ is attached a support, which is the \textit{complement} of the maximal open subset $U \subseteq X$ such that $\mu(U)=0$ (\cite[Exercise 2.11]{Rudin1987}).

%\textit{Metrized pseudo-divisors.}
%We say that $\widehat{D}$ is a $\nu$-metrized pseudo-divisor on a $K$-variety $X$ if it is a metrized pseudo-divisor on $X \times_K K_\nu$ as defined in \cite[3.4]{Gubler2007a}. In the terminology of \cite{Gubler2003}, a $\nu$-metrized pseudo-divisor on $X$ for some $\nu \in \Sigma_\infty(K)$ is the same as a hermitian pseudo-divisor on $X \times_K K_\nu$.
%By \cite[]{Berkovich1990}, each irreducible component $Y$ of $\mathcal{X}_{\nu,s}$ determines a unique point $\xi_Y \in X_{\IC_\nu}^{\mathrm{an}}$. Furthermore, there is a reduction map 

% In addition, there is a canonical reduction map $\mathrm{red} \!: \mathcal{X}_{\nu,\eta}\rightarrow \mathcal{X}_{\nu,s}$.

\section{Arithmetic Intersection Theory}
\label{section::arithmeticintersectiontheory}

In this section, we recall the basics of arithmetic intersection theory for (adelically) metrized line bundles. Our basic references are \cite{Chambert-Loir2000,Chambert-Loir2006,Chambert-Loir2011, Gubler1997, Gubler1998, Gubler2003, Gubler2007a, Maillot2000, Moriwaki2014,Yuan2012,Zhang1995a}. We only need to associate intersection numbers with integrable (adelically) metrized line bundles on projective varieties over a number field and can hence avoid arithmetic Chow rings \cite{Bost1994, Faltings1992, Gillet1990, Soule1992}. %, we can work rather economically, avoiding some intricacies of the original arithmetic intersection theory of Gillet and Soul\'e \cite{Bost1994, Faltings1992, Gillet1990, Soule1992} such as arithmetic Chow rings and so forth. 

%In contrast to other articles, we allow semipositive metrics at archimedean places to be hermitian $\mathscr{C}^0$-metrics without requiring a global uniform approximation by semipositive hermitian $\mathscr{C}^\infty$-metrics. In fact, it seems that such an approximation is not known to exist. Nevertheless, global approximations with small negative part are well-known to exist (see \cite{Demailly1992a, Richberg1968} and our Appendix \ref{appendixA}), and this can be used to extend the $\mathscr{C}^\infty$-theory to a $\mathscr{C}^0$-theory.

\subsection{Setup}

Throughout this section, we consider a number field $K$ and an irreducible, projective $K$-variety $X$ of pure dimension $d$. Let $\nu \in \Sigma_f(K)$ and $e$ be a positive integer. Since $X$ is projective, the analytic space $X_{\IC_\nu}^{\mathrm{an}}$ is Hausdorff and compact for every $\nu \in \Sigma(K)$ (\cite[Theorem 3.4.8 (ii)]{Berkovich1990}).

Let $L_1, L_2, \dots, L_{k}$ be line bundles on $X$. (The case $k=0$ is allowed here.) A formal $\caO_\nu$-model $(\mathcal{X}_\nu,\mathcal{L}_{1,\nu},\dots,\mathcal{L}_{k,\nu})$ of $(X,L_1^{\otimes e_1},\dots, L_k^{\otimes e_k})$ consists of an admissible formal scheme $\mathcal{X}_\nu$ over $\mathrm{Spf}(\mathcal{O}_\nu)$ and line bundles $\mathcal{L}_{i,\nu}$ on $\mathcal{X}_\nu$ such that $X_{\IC_\nu}^{\mathrm{an}} = \mathcal{X}_{\nu,\eta}$ and $(L^{\otimes e_i}_i)_{\IC_\nu}^{\mathrm{an}} = \mathcal{L}_{i,\nu,\eta}$. Similarly, an $S$-model $(\mathcal{X},\mathcal{L}_1,\dots,\mathcal{L}_k)$ of $(X,L_1^{\otimes e_1},\dots,L_k^{\otimes e_k})$ is a flat, integral, projective $S$-scheme $\mathcal{X}$ and a collection of line bundles $\mathcal{L}_i$ on $\mathcal{X}$ such that $X = \mathcal{X} \times_S K$ and $L^{\otimes e_i}_i = \mathcal{L}_i \times_S K$. For each $\nu\in \Sigma_f(K)$, a $S$-model $(\mathcal{X},\mathcal{L}_1,\dots,\mathcal{L}_k)$ of $(X,L_1^{\otimes e_1},\dots,L_k^{\otimes e_k})$ gives naturally rise to a formal $\caO_\nu$-model $(\mathcal{X}_\nu,\mathcal{L}_{1,\nu},\dots,\mathcal{L}_{k,\nu})$ of $(X,L_1^{\otimes e_1},\dots, L_k^{\otimes e_k})$ by taking formal completions.

\subsection{Metrics}
\label{section::metrics}

Let $L$ be a line bundle on $X$. A \textit{$\nu$-metric} on $L$ is a %$\Gal(\IC_\nu/K_\nu)$-invariant\footnote{Each $\sigma \in \Gal(\IC_\nu/K_\nu)$ induces an algebraic map $\sigma_L: L_{\IC_\nu} \rightarrow L_{\IC_\nu}$ over $\sigma: \IC_\nu \rightarrow \IC_\nu$, and invariance means that $\Vert \cdot \Vert \circ \sigma_L^{\mathrm{an}} = \Vert \cdot \Vert$.}
map $\Vert \cdot \Vert: L_{\IC_\nu}^{\mathrm{an}} \rightarrow \IR_{\geq 0}$ such that for each open subset $U \subseteq X$, each section $\mathbf{s}: U \rightarrow L$, and every $f \in \mathcal{O}_X(U)$, 
\begin{itemize}
\item[(a)] the function $\Vert \mathbf{s} \Vert = \Vert \cdot \Vert \circ \mathbf{s}^{\mathrm{an}}_{\IC_\nu}: U^{\mathrm{an}}_{\IC_\nu} \rightarrow \IR_{\geq 0}$ is continuous,
\item[(b)] if $\mathbf{s}$ vanishes nowhere on $U$, then $\Vert \mathbf{s} \Vert: U^{\mathrm{an}}_{\IC_\nu} \rightarrow \IR_{\geq 0}$ has no zeros,
\item[(c)] for every $f \in \mathcal{O}_X(U)$, we have $\Vert f \mathbf{s} \Vert = |f| \cdot \Vert \mathbf{s} \Vert$ on $U^{\mathrm{an}}_{\IC_\nu}$, and
\item[(d)] $\Vert \cdot \Vert$ is $\Gal(\IC_\nu/K_\nu)$-invariant\footnote{Each $\sigma \in \Gal(\IC_\nu/K_\nu)$ induces an algebraic map $\sigma_L: L_{\IC_\nu} \rightarrow L_{\IC_\nu}$ over $\sigma: \IC_\nu \rightarrow \IC_\nu$, and invariance means here that $\Vert \cdot \Vert \circ \sigma_L^{\mathrm{an}} = \Vert \cdot \Vert$. If $\nu \in \Sigma_f(K)$, the analytic space $L_{K_\nu}^{\mathrm{an}}$ is the quotient of $L_{\IC_\nu}^{\mathrm{an}}$ by $\Gal(\IC_\nu/K_\nu)$ (see \cite[Corollary 1.3.6]{Berkovich1990}) so that one can equivalently demand that $\Vert \cdot \Vert$ is the pullback of a map $L_{K_\nu}^{\mathrm{an}} \rightarrow \IR_{\geq 0}$ satisfying (a),(b), and  (c) with $\IC_\nu$ replaced by $K_\nu$.}.
\end{itemize}
For a sequence of $\nu$-metrics $\Vert \cdot \Vert^{(n)}$ on $L$, we say that $\Vert \cdot \Vert^{(n)}$ \textit{converges uniformly} to a $\nu$-metric $\Vert \cdot \Vert$ if the $\mathscr{C}^0$-functions $\Vert \cdot \Vert^{(n)}/\Vert \cdot \Vert$ converge uniformly to $1$ on $X_{\IC_\nu}^{\mathrm{an}}$. A \textit{$\nu$-metrized line bundle} $\overline{L}= (L,\Vert \cdot \Vert)$ consists of a line bundle $L$ on $X$ and a $\nu$-metric on $L$. An \textit{isometry} $f: (L,\Vert \cdot \Vert) \rightarrow (M, \Vert \cdot \Vert^\prime)$ between two $\nu$-metrized line bundles is an isomorphism $f: L \rightarrow M$ of line bundles that transports $\Vert \cdot \Vert$ to $\Vert \cdot \Vert^\prime$. The set of isometry classes of $\nu$-metrized line bundles on $X$ is denoted by $\overline{\Pic}_{\nu}(X)$. For $\overline{L}, \overline{M} \in \overline{\Pic}_{\nu}(X)$, $\No \in \overline{\Pic}_{\nu}(Y)$, and every algebraic map $f: X \rightarrow Y$, we define $\Lo + \Mo$, $-\Lo$, and $f^\ast \No$ as elements of $\overline{\Pic}_{\nu}(X)$ in the obvious way. If $\Mo = (L^{\otimes e}, \Vert \cdot \Vert)$ is a $\nu$-metrized line bundle for some non-zero integer $e$, there is a unique $\nu$-metrized line bundle $\Lo=(L, \Vert \cdot \Vert^{1/e})$ such that $e \Lo \approx \Mo$. For later applications we also set $\overline{\Pic}_{\nu}(X)_\IQ = \overline{\Pic}_{\nu}(X) \otimes_\IZ \IQ$.

If $\nu\in \Sigma_\infty(K)$, a $\nu$-metric is just a $\mathscr{C}^0$-hermitian metric on $L^{\mathrm{an}}_{\IC_\nu}$. A $\nu$-metrized line bundle $(L, \Vert \cdot \Vert)$ is called \textit{semipositive} (resp.\ \textit{strictly positive}) if $dd^c c_1(L, \Vert \cdot \Vert) \geq 0$ (resp.\ $dd^c c_1(L, \Vert \cdot \Vert)>0$).

For $\nu \in \Sigma_f(K)$, every formal $\nu$-model $(\mathcal{X}_\nu,\mathcal{L}_\nu)$ of $(X,L^{\otimes e})$ induces a $\nu$-metric on $L$ (see \cite[Section 7]{Gubler1998}): %Comparing the constructions in \cite[Section 3.4]{Berkovich1990} and \cite[Section 1]{Berkovich1994}, we see that $X^{\mathrm{an}}_{\IC_\nu}$ (resp.\ $L^{\mathrm{an}}_{\IC_\nu}$) coincides with the generic fiber $\mathcal{X}_{\nu, \eta}$ (resp.\ $\mathcal{L}_{\nu, \eta}$). %In particular, there is a reduction map $\mathrm{red}: X^{\mathrm{an}}_{\IC_\nu} \rightarrow |\mathcal{X}_{\nu}|$ where $|\mathcal{X}_{\nu}|$ is the topological space underlying $\mathcal{X}_{\nu}$.
Let $\{ \mathfrak{U}_i \}$ be a covering of $\mathcal{X}_{\nu}$ by formal open subschemes such that there are isomorphisms $\varphi_{i}: \mathcal{L}_\nu|_{\mathfrak{U}_i} \rightarrow \IA^1_{\mathfrak{U}_i}$ and each $\mathfrak{U}_{i,\eta} \subseteq X^{\mathrm{an}}_{\IC_\nu}$ is an affinoid $\IC_\nu$-analytic space. The maps $\varphi_i$ induce isomorphisms $\varphi_{i,\eta}: (L^{\otimes e})^{\mathrm{an}}_{\IC_\nu}|_{\mathfrak{U}_{i,\eta}} \rightarrow \IA^1_{\mathfrak{U}_{i,\eta}}$ over $\mathfrak{U}_{i,\eta}$. For every $x \in \mathfrak{U}_{i,\eta}$, we set $\Vert v \Vert_{i} =|\varphi_{i,\eta}(v)|_\nu$ for all $v \in (L^{\otimes e})^{\mathrm{an}}_{\IC_\nu}|_x$. On overlaps $\mathfrak{U}_{i} \cap \mathfrak{U}_{j}$, the composites $\varphi_j \circ \varphi_i^{-1}$ are described by elements $f_{ij} \in \caO^\times(\mathfrak{U}_{i} \cap \mathfrak{U}_{j})$. Each $f_{ij}$ induces a meromorphic function $f_{ij,\eta}$ on $\mathfrak{U}_{i,\eta} \cap \mathfrak{U}_{j,\eta}$ with supremum norm $\leq 1$ such that $\varphi_j \circ \varphi_i^{-1}|_{\mathfrak{U}_{i.\eta} \cap \mathfrak{U}_{j,\eta}}$ is multiplication by $f_{ij,\eta} \in \caO^\times(\mathfrak{U}_{i,\eta} \cap \mathfrak{U}_{j,\eta})$. Since $f_{ij,\eta}=f_{ji,\eta}^{-1}$, this implies $|f_{ij,\eta}(x)|=1$ for all $x\in \mathfrak{U}_{i,\eta} \cap \mathfrak{U}_{j,\eta}$. Consequently, the $\nu$-metrics $\{ \Vert \cdot \Vert_{i} \}$ glue to a $\nu$-metric $\Vert \cdot \Vert_{\mathcal{L}_\nu}$ on $(L^{\otimes e})^{\mathrm{an}}_{\IC_\nu}$ so that we obtain a $\nu$-metric $\Vert \cdot \Vert_{\mathcal{L}_\nu}^{1/e}$ on $L_{\IC_\nu}^{\mathrm{an}}$. The $\nu$-metric $\Vert \cdot \Vert_{\mathcal{L}_\nu}^{1/e}$ on $L_{\IC_\nu}^{\mathrm{an}}$ is called \textit{formally semipositive} if $\mathcal{L}_{\nu,s}$ can be chosen to be a nef line bundle on the special fiber $\mathcal{X}_{\nu,s}$. A general $\nu$-metrized line bundle is called \textit{semipositive} if its $\nu$-metric is the uniform limit of formally semipositive  $\nu$-metrics.

For every $\Gal(\IC_\nu/K_\nu)$-invariant  $g \in \mathscr{C}^0(X^{\mathrm{an}}_{\IC_\nu})$, we define the $\nu$-metrized line bundle $\overline{\caO}_X(g) = (\caO_X, \Vert \cdot \Vert)$ by setting $\Vert 1_x \Vert_\nu = e^{-g(x)}$ for all $x \in X^{\mathrm{an}}_{\IC_\nu}$. For a $\nu$-metrized line bundle $\Lo$, we write $\Lo(g)$ instead of $\Lo \otimes \overline{\caO}_X(g)$.
%For a finite extension $K \subset K^\prime$ and a place $\nu^\prime \in \Sigma(K^\prime)$ above $\nu \in \Sigma(K)$, there is a natural notion of a base change $\overline{L}_{\nu^\prime} \in \overline{\Pic}_{\nu^\prime}(X_{K^\prime})$ of a $\nu$-metrized line bundle $\overline{L} = (L, \Vert \cdot \Vert) \in \overline{\Pic}_\nu(X)$. The $\nu$-metric $\Vert \cdot \Vert: L_{\IC_\nu}^{\mathrm{an}} \rightarrow \IR_{\geq 0}$ gives rise to a $\nu^\prime$-metric $\Vert \cdot \Vert_{\nu^\prime}: (L_{K^\prime})^{\mathrm{an}}_{\IC_{\nu^\prime}} \rightarrow \IR_{\geq 0}$ on $L_{K^\prime}$. In fact, we can choose a non-unique identification $\IC_\nu \approx \IC_{\nu^\prime}$ that gives rise to $(L_{K^\prime})^{\mathrm{an}}_{\IC_{\nu^\prime}} \approx L^{\mathrm{an}}_{\IC_\nu}$. We then set $\overline{L}_{K^\prime}=(L_{K^\prime}, \Vert \cdot \Vert_{\nu}^{[K^\prime_{\nu^\prime}:K_\nu]/\delta_?})$; our assumption (d) above implies that this is independent of the identification $\mathbb{C}_\nu \approx \mathbb{C}_{\nu^\prime}$ used.

A \textit{$\nu$-metrized pseudo-divisor} on $X$ is a triple $(\overline{L}, Y, \mathbf{s})$ consisting of a $\nu$-metrized line bundle $\overline{L}$ over $X$, an algebraic subvariety $Y \subset X$, and a nowhere vanishing section $\mathbf{s}: X \setminus Y \rightarrow L$. This is an analogue of the pseudo-divisors used in algebraic intersection theory \cite[Definition 2.2.1]{Fulton1998}. Our definition is more restrictive than the one introduced in \cite[3.4]{Gubler2007a}, but fully suffices for our purposes.

%We say that $\widehat{D}$ is a $\nu$-metrized pseudo-divisor on a $K$-variety $X$ if it is a metrized pseudo-divisor on $X \times_K K_\nu$ as defined in \cite[3.4]{Gubler2007a}. In the terminology of \cite{Gubler2003}, a $\nu$-metrized pseudo-divisor on $X$ for some $\nu \in \Sigma_\infty(K)$ is the same as a hermitian pseudo-divisor on $X \times_K K_\nu$.

\subsection{Borel measures}
\label{section::borelmeasure}

For both archimedean and non-archimedean places $\nu \in \Sigma(K)$, a collection of semipositive $\nu$-metrized line bundles $\overline{L}_1,\overline{L}_2,\dots,\overline{L}_{d} \in \overline{\Pic}_\nu(X)$ gives rise to a finite regular Borel measure $c_1(\overline{L}_1) \wedge c_1(\overline{L}_2) \wedge
\cdots \wedge c_1(\overline{L}_{d})$ on $X_{\IC_\nu}^{\mathrm{an}}$. If $X$ is smooth, $\nu \in \Sigma_\infty(K)$,  and the metrics of $\overline{L}_1,\overline{L}_2,\dots,\overline{L}_{d}$ are $\mathscr{C}^\infty$, we just take the Borel measure given by integrating with the wedge product of the Chern forms $c_1(\Lo_i)$ ($i \in \{ 1,\dots, d\} $). If the metrics of $\Lo_1,\dots,\Lo_d$ are only $\mathscr{C}^0$ or $\nu \in \Sigma_f(K)$, the definition of $c_1(\Lo_1) \wedge c_1(\Lo_2) \wedge \cdots \wedge c_1(\Lo_d)$ is more involved, but we nevertheless retain the notation from the smooth archimedean case for ease of notation.

We start with defining $c_1(\Lo_1) \wedge c_1(\Lo_2) \wedge \cdots \wedge c_1(\Lo_d)$ for archimedean $\nu \in \Sigma_\infty(K)$. Let $U \subseteq X_{\IC_\nu}^{\mathrm{an}}$ be a sufficiently small open set such that there exist non-vanishing sections $\mathbf{s}_i: U \rightarrow (L_i)_{\IC_\nu}^{\mathrm{an}}$. By assumption, the Chern currents $c_1(\Lo_i|_U)= dd^c (- \log \Vert \mathbf{s}_i \Vert)$ are semipositive. Shrinking $U$ if necessary, their local potentials $(- \log \Vert \mathbf{s}_i \Vert)$ are bounded on $U$. As proposed by Bedford and Taylor \cite{Bedford1982} (see \cite[Chapter 3]{Guedj2017} for smooth $X$ and \cite{Demailly1985} for general $X$), we can define a semipositive, closed current
\begin{equation*}
T_U = dd^c (-\log \Vert \mathbf{s}_1 \Vert ) \wedge dd^c (-\log \Vert \mathbf{s}_2 \Vert ) \wedge \cdots \wedge dd^c (-\log \Vert \mathbf{s}_d \Vert ) 
\end{equation*}
on $U$. %\footnote{The reader is advised that we do not assume here that $X$ is smooth but work with an arbitrary reduced complex space $X$, following \cite{Demailly1985}, in contrast to (most of) the literature in complex geometry \cite{Bedford1982, Guedj2017}. %An alternative way would be to first define $T_U$ for smooth $X$ and to deduce the general case via a resolution of singularities; this is done e.g.\ in the proof of \cite[Proposition-Définition 1.7]{Chambert-Loir2000}. In ou} 
This current depends only on $c_1(\Lo_i|_{U})$ and not on the local potentials $(-\log \Vert \mathbf{s}_i \Vert)$.
Consequently, the currents $T_{U_1}$ and $T_{U_2}$ agree on $U_1 \cap U_2$ for any two open sets $U_1$ and $U_2$ as above. A partition of unity argument (see \cite[Theorem 2.2.4]{Hoermander2003}) shows that there is a unique, semipositive, closed current $T$ on $X^{\mathrm{an}}_{\IC_\nu}$ of bidegree $(d,d)$ restricting to $T_U$ on every open $U$ as above. Because of its non-negativity (see \cite[Theorem 2.1.7]{Hoermander2003}), $T$ is actually a distribution of order $0$. Using Riesz representation theorem \cite[Theorem 2.14]{Rudin1987}, we obtain a unique Borel measure $c_1(\Lo_1) \wedge c_1(\Lo_2) \wedge \cdots \wedge c_1(\Lo_d)$ on $X_{\IC_\nu}^{\mathrm{an}}$. It is a consequence of the Chern-Levine-Nirenberg inequalities \cite{Chern1969} that $c_1(\Lo_1) \wedge c_1(\Lo_2) \wedge \cdots \wedge c_1(\Lo_d)$ does not charge locally pluripolar sets and has finite mass (see \cite[Theorems 3.9 and 3.14]{Guedj2017}\footnote{\label{footnote}Literally, these theorems only apply if $X_{\IC_\nu}^{\mathrm{an}}$ is smooth. By Hironaka's resolution theorem \cite{Hironaka1964} (see also \cite{Kollar2007}), there always exists a smooth variety $\widetilde{X}$ and a birational, projective morphism $f: \widetilde{X} \rightarrow X$. One can then use \cite[Theorem 3.14]{Guedj2017} to prove that $c_1(f^\ast\Lo_1) \wedge c_1(f^\ast\Lo_2) \wedge \cdots \wedge c_1(f^\ast\Lo_d)$ does not charge locally pluripolar sets in $\widetilde{X}$. In particular, no mass is attached to the ramification locus $E$ of $f$. We can then obtain the same assertion for $c_1(\Lo_1) \wedge c_1(\Lo_2) \wedge \cdots \wedge c_1(\Lo_d)$ by means of Lemma \ref{lemma::chernforms} (c).}).\footnote{Finiteness also implies the claimed regularity by \cite[Theorem 2.18]{Rudin1987}.}

For non-archimedean $\nu \in \Sigma_f(K)$, we use the measures introduced by Chambert-Loir (see \cite[Section 2]{Chambert-Loir2006}) and define $c_1(\Lo_1) \wedge \cdots \wedge c_1(\Lo_d)$ as in \cite[(3.8)]{Gubler2007a}. For this, we choose non-zero rational sections $\mathbf{s}_i: X \dashrightarrow L_i$ ($i=1,\dots,d$). These define $\nu$-metrized pseudo-divisors $\widehat{\Div}_\nu(\mathbf{s}_i) =(\overline{L}_i, \left\vert\Div(\mathbf{s}_{i})\right\vert, \mathbf{s}_{i})$ on $X$. With each function $g \in \mathscr{C}^0(X^{\mathrm{an}}_{\IC_\nu})$, we can furthermore associate the $\nu$-metrized pseudo-divisor $(\overline{\caO}_X(g), \emptyset, 1)$. Defining $\lambda_{(\overline{\caO}_X(g), \emptyset, 1), \widehat{\Div}_\nu(\mathbf{s}_1), \dots, \widehat{\Div}_\nu(\mathbf{s}_d)}(X)$ as the local height of $X$ in the sense of Gubler \cite[Section 9]{Gubler2003}, we obtain a functional
\begin{equation*}
\mathscr{C}^0(X^{\mathrm{an}}_{\IC_\nu}) \longrightarrow \IR, \ g \longmapsto \lambda_{(\overline{\caO}_X(g), \emptyset, 1), \widehat{\Div}_\nu(\mathbf{s}_1), \dots, \widehat{\Div}_\nu(\mathbf{s}_d)}(X),
\end{equation*}
which induces a Borel measure on $X^{\mathrm{an}}_{\IC_\nu}$ by the Riesz representation theorem \cite[Theorem 2.14]{Rudin1987}. The ensuing measure is independent of the choice of $\mathbf{s}_i$ by \cite[Theorem 3.5 (c)]{Gubler2007a}, and we denote it by $c_1(\overline{L}_1) \wedge c_1(\overline{L}_2) \wedge
\cdots \wedge c_1(\overline{L}_{d})$ in the sequel. Moreover, the measure is finite by \cite[Corollary 3.9 (c)]{Gubler2007a} and hence regular by \cite[Theorem 2.18]{Rudin1987}.

In spite of the different definitions, the following two lemmas allow to treat both cases $\nu \in \Sigma_f(K)$ and $\nu \in \Sigma_\infty(K)$ in a uniform manner.

\begin{lemma} 
\label{lemma::uniformconvergence}
For each $i \in \{ 1, \dots, d\}$, let $L_i$ be a line bundle on $X$ and $\Vert \cdot \Vert^{(n)}_i: L_i^{\mathrm{an}} \rightarrow \IR_{\geq 0}$ a sequence of semipositive $\nu$-metrics converging uniformly to $\Vert \cdot \Vert_i$. Writing $\Lo_i^{(n)} = (L_i, \Vert \cdot \Vert^{(n)}_i)$ and $\Lo_i = (L_i, \Vert \cdot \Vert_i)$, there is then a weak convergence
\begin{equation*}
 c_1(\Lo_1^{(n)}) \wedge c_1(\Lo_2^{(n)}) \wedge \cdots \wedge c_1(\Lo_d^{(n)})
\longrightarrow
 c_1(\Lo_1) \wedge c_1(\Lo_2) \wedge \cdots \wedge c_1(\Lo_d)
\ (n \rightarrow \infty)
\end{equation*}
of measures.
\end{lemma}

\begin{proof} For $\nu \in \Sigma_\infty(K)$, this is \cite[Corollary 1.6]{Demailly1993}. The non-archimedean case $\nu \in \Sigma_f(K)$ is \cite[Proposition 3.12]{Gubler2007a}.
\end{proof}

\begin{lemma}
\label{lemma::chernforms}
Let $\Lo_1,\dots,\Lo_d,\Lo_1^\prime \in \overline{\Pic}_{\nu}(X)$ be semipositive.
\begin{enumerate}
\item[(a)] (Multilinearity) We have
\begin{equation*}
c_1(\Lo_1+\Lo_1^\prime) \wedge c_1(\Lo_2) \wedge \cdots \wedge c_1(\Lo_d)
= 
c_1(\Lo_1) \wedge c_1(\Lo_2) \wedge \cdots \wedge c_1(\Lo_d)
+
c_1(\Lo_1^\prime) \wedge c_1(\Lo_2)\wedge \cdots \wedge c_1(\Lo_d).
\end{equation*}
\item[(b)] (Commutativity) For any permutation $\sigma: \{ 1, \dots, n \} \rightarrow \{1, \dots, n\}$, we have 
\begin{equation*}
c_1(\Lo_{\sigma(1)}) \wedge c_1(\Lo_{\sigma(2)}) \wedge \cdots\wedge c_1(\Lo_{\sigma(d)}) = c_1(\Lo_{1}) \wedge c_1(\Lo_{2}) \wedge \cdots \wedge c_1(\Lo_{d}).
\end{equation*}
\item[(c)] (Projection Formula) Let $Y$ be an irreducible, projective $K$-variety, and let $f: Y \rightarrow X$ be a generically finite surjective map of degree $\deg(f)$. Then, the push-forward measure
\begin{equation*}
(f^{\mathrm{an}}_{\IC_\nu})_\ast( c_1(f^\ast\Lo_1) \wedge c_1(f^\ast\Lo_2) \wedge \cdots \wedge c_1(f^\ast\Lo_d))
\end{equation*} 
equals 
\begin{equation*}
\deg(f) \cdot  c_1(\Lo_1) \wedge c_1(\Lo_2) \wedge \cdots \wedge c_1(\Lo_d).
\end{equation*}
\item[(d)] (Total mass) We have 
\begin{equation*}
\int_{X^{\mathrm{an}}_{\IC_\nu}}  c_1(\Lo_{1}) \wedge c_1(\Lo_2) \wedge \cdots \wedge c_1(\Lo_d) =  L_1 \cdot L_2 \cdots L_d.
\end{equation*}
\end{enumerate}
%{\color{red} (e) Base change: 1 (only $\log \Vert s\Vert $ does $[K_\mu^\prime:K_\nu]$); with components. unfortunately, I need this also in case $Y$ is not irreducible, no you just restrict to the case where $Y$ is an irreducible component. no so, just need to for $Y$ irreducible now!}
\end{lemma}

\begin{proof} 

In case of $\nu \in \Sigma_\infty(K)$, the first three assertions are evident if all metrics are $\mathscr{C}^\infty$. The local nature of the first three statements allows as above to use plurisubharmonic smoothings (see \cite[Proposition 1.42]{Guedj2017}) and Lemma \ref{lemma::uniformconvergence}. For the fourth statement, which is stated as \cite[Corollary 9.3]{Demailly1993}, one needs a global $\mathscr{C}^\infty$-regularization of the $\mathscr{C}^0$-metrics (compare our Appendix \ref{appendixA}).

The non-archimedean case $\nu \in \Sigma_f(K)$ is \cite[Corollary 3.9]{Gubler2007a} for the first three assertions and \cite[Proposition 3.12]{Gubler2007a} for the fourth one.
\end{proof}

We say that $\Lo \in \overline{\Pic}_\nu(X)_{\IQ}$ is \textit{integrable} if there exists a non-zero integer $n$ and semipositive $\Lo_1, \Lo_2 \in  \overline{\Pic}_{\nu}(X)$ such that $n \Lo = \Lo_1 - \Lo_2$. By (a) of the above lemma, we can define signed Borel measures $c_1(\Lo_{1}) \wedge c_1(\Lo_2) \wedge \cdots \wedge c_1(\Lo_d)$ for integrable $\Lo_i \in \overline{\Pic}_{\nu}(X)_\IQ$.

\subsection{Hermitian line bundles on arithmetic varieties}
\label{section::hermitianlinebundles}

Let $\mathcal{X}$ be a flat, integral, projective $S$-scheme of relative dimension $d$. A hermitian line bundle $\overline{\mathcal{L}}$ on $\mathcal{X}$ is a collection $(\mathcal{L}, {\{\Vert\cdot\Vert_\nu\}_{\nu \in \Sigma_\infty(K)}})$ consisting of a line bundle $\mathcal{L}$ on $\mathcal{X}$ and a $\nu$-metric $\Vert\cdot\Vert_\nu$ on $\mathcal{L}_K$ for each archimedean place $\nu \in \Sigma_\infty(K)$; if $K_\nu = \IR$, we assume additionally that the $\nu$-metric is invariant under $\Gal(\IC_\nu/K_\nu)$ (i.e., under complex conjugation on $X_{\IC_\nu}^{\mathrm{an}}$). We say that the hermitian line bundles $\overline{\mathcal{L}}$ and $\overline{\mathcal{M}}$ are \textit{isometric} if there is an isomorphism $\mathcal{L} \approx \mathcal{M}$ preserving the metrics at all archimedean places. The \textit{arithmetic Picard group} $\widehat{\Pic}(\mathcal{X})$ is the set of isometry classes of hermitian line bundles on $\mathcal{X}$.

A hermitian line bundle $\overline{\mathcal{L}} \in \widehat{\Pic}(\mathcal{X})=(\mathcal{L},\{ \Vert \cdot \Vert_\nu\}_{\nu \in \Sigma_\infty(K)})$ is called \textit{vertically semipositive}, if $\mathcal{L}$ is relatively nef with respect to $\mathcal{X} \rightarrow S$ and each $(\mathcal{L}_K, \Vert \cdot \Vert_\nu)$ is a semipositive $\nu$-metrized line bundle. This definition extends naturally to $\widehat{\Pic}(\mathcal{X})_\IQ = \widehat{\Pic}(\mathcal{X}) \otimes_\IZ \IQ$.

\subsection{Metrized line bundles on $K$-varieties}
\label{section::adelicmetrics}

A collection of $\nu$-metrics $\{ \Vert \cdot \Vert_\nu \}_{\nu \in \Sigma(K)}$ on $L$ is called an \textit{adelic metric} if there exists an $S$-model $(\mathcal{X},\mathcal{L})$ of $(X,L^{\otimes e})$ such that $\Vert \cdot \Vert_\nu= \Vert \cdot \Vert_{\mathcal{L}_\nu}^{1/e}$ on $L_{\IC_\nu}^{\mathrm{an}}$ for all but finitely many $\nu \in \Sigma_f(K)$ (\textit{coherence condition}). The pair $\Ltil=(L, \{ \Vert \cdot \Vert_\nu \}_{\nu \in \Sigma(K)})$ is then called \textit{a(n adelically) metrized line bundle}. For any place $\nu \in \Sigma(K)$, we denote the $\nu$-metrized line bundle $(L,  \Vert \cdot \Vert_{\nu} )$
associated with $\Ltil$ by $\Lo_{\nu}$.

Again, there is a natural notion of \textit{isometry} between metrized line bundles and the isometry classes of metrized line bundles form a Picard group $\widehat{\mathrm{Pic}}(X)$. If $\widetilde{L}$ and $\widetilde{M}$ are metrized line bundles with underlying line bundles $L$ and $M$, there is an obvious way to endow $L \otimes M$ and $L^{-1}$ with the structure of metrized line bundles. We write $\widetilde{L} \otimes \widetilde{M}$ and $\widetilde{L}^{-1}$, respectively, for these metrized line bundles. If $f: Y \rightarrow X$ is an algebraic map between irreducible, projective $K$-varieties and $\widetilde{L}$ is a metrized line bundle, we can endow the pull-back $f^\ast L$ with a canonical adelic metric, obtaining a metrized line bundle $f^\ast \widetilde{L}$. For a closed immersion $f: Y \hookrightarrow X$ of $K$-varieties, we write $\Ltil|_Y$ instead of $f^\ast \Ltil$.

We define the \textit{base change} $\widetilde{L}_{K^\prime} = (L_{K^\prime}, \{ \Vert \cdot \Vert_{\nu^\prime} \}_{\nu^\prime \in \Sigma(K^\prime)})$ of a $\nu$-metrized line bundle $\widetilde{L}$ as follows: For each place $\nu \in \Sigma(K^\prime)$ lying above $\nu \in \Sigma(K)$, we have a (non-canonical) identification $\IC_\nu \approx \IC_{\nu^\prime}$ extending $K_\nu \hookrightarrow K^\prime_{\nu^\prime}$. This yields an isomorphism $ L_{\IC_{\nu}}^{\mathrm{an}} = (L_{K^\prime})_{\IC_{\nu^\prime}}^{\mathrm{an}}$, which we use to set $\Vert \cdot \Vert_{\nu^\prime} = \Vert \cdot \Vert_\nu^{[K^\prime_{\nu^\prime}:K_\nu]}$ if $\delta_\nu = \delta_{\nu}^\prime$ and $\Vert \cdot \Vert_{\nu^\prime} = \Vert \cdot \Vert_\nu$ elsewise\footnote{Note that the latter case means that $\delta_\nu = 1$ and $\delta_{\nu^\prime}=2$ (i.e., $\nu$ is a real archimedean and $\nu$ is a complex archimedean place).}. As $\Vert \cdot \Vert_\nu$ is $\Gal(\IC_\nu/K_\nu)$-invariant (property (d) in Subsection \ref{section::metrics})), this gives a well-defined metrized line bundle $\widetilde{L}_{K^\prime}$ on $X_{K^\prime}$ not depending on the identification $\IC_\nu \approx \IC_{\nu^\prime}$.

Every hermitian line bundle on an $S$-model induces a metrized line bundle. In fact, let $(\mathcal{X},\mathcal{L})$ be an $S$-model of $(X,L^{\otimes e})$ and let $\overline{\mathcal{L}} = (\mathcal{L}, \{\Vert\cdot\Vert_\nu \}_{\nu \in \Sigma_\infty(K)}) \in \widehat{\Pic}(\mathcal{X})$ be a hermitian line bundle. For each $\nu \in \Sigma_f(K)$ (resp.\ $\nu \in \Sigma_\infty(K)$), we set $\Vert \cdot \Vert_\nu^\prime = \Vert \cdot \Vert _{\mathcal{L}_\nu}^{1/e}$ (resp.\ $\Vert \cdot \Vert_\nu^{\prime} = \Vert \cdot \Vert_\nu^{1/e}$). Then $\{ \Vert \cdot \Vert_\nu^{\prime} \}_{\nu \in \Sigma(K)}$ is an adelic metric on $L$, and we call adelic metrics of this type \textit{algebraic}. The corresponding metrized line bundles are called \textit{algebraically metrized}.% If the $\nu$-metrics $\Vert \cdot \Vert_\nu$ ($\nu \in \Sigma_\infty(K)$) are Galois-invariant, then the algebraically metrized line bundle $\overline{\mathcal{L}}$ is also Galois-invariant.

All other adelic metrics of interest for us arise from algebraic adelic metrics by means of a limit process. Let $\{ \Vert \cdot \Vert_{i,\nu} \}$, $i\in \IN$, be a sequence of adelic metrics on $L$. We say that these metrics \textit{converge uniformly} to an adelic metric $\{ \Vert \cdot \Vert_{\nu} \}$ on $L$ if there exists a finite set of places $\Sigma_0 \subset \Sigma(K)$ such that $\Vert \cdot \Vert_{i,\nu} / \Vert \cdot \Vert_{\nu} \rightarrow 1$ ($i \rightarrow \infty$) uniformly on $X_{\IC_\nu}^{\mathrm{an}}$ for all $\nu \in \Sigma_0$ and $\Vert \cdot \Vert_{i,\nu} = \Vert \cdot \Vert_{\nu}$ for all $\nu \notin \Sigma_0$ and $i$.

The standard metrics $\{\vert \cdot \vert_\nu \}_{\nu \in \Sigma(K)}$ on $\caO_X$ yield a metrized line bundle $\widetilde{\caO}_X=(\caO_X, \{\vert \cdot \vert_\nu \})$. Similarly, we define for each $\Gal(\IC_{\nu_0}/K_{\nu_0})$-invariant 
function $f \in \mathscr{C}^0(X^{\mathrm{an}}_{\IC_{\nu_0}})$, $\nu_0 \in \Sigma(K)$, a metrized line bundle $\widetilde{\caO}_X(f)=(\caO_X, \{\Vert \cdot \Vert_\nu \}_{\nu \in \Sigma(K)})$ by setting $\Vert \cdot \Vert_{\nu_0} = e^{-f}\vert \cdot \vert_{\nu_0}$ and $\Vert \cdot \Vert_\nu = \vert \cdot \vert_\nu$ if $\nu \neq \nu_0$. For each $\Ltil \in \widehat{\Pic}(X)$, we write $\Ltil(f)$ instead of $\Ltil \otimes \widetilde{\caO}_X(f)$.

Finally, an element $\Ltil \in \widehat{\Pic}(X)$ is called \textit{vertically semipositive (resp.\ vertically integrable)} if each $\Lo_\nu \in \widehat{\Pic}_\nu(X)$, $\nu \in \Sigma(K)$, is semipositive (resp.\ integrable). Again, we can extend this terminology to $\widehat{\Pic}(X)_\IQ = \widehat{\Pic}(X) \otimes_\IZ \IQ$. 

%The reader should be warned that we do \textit{not} use an analogue definition of semipositivity. In fact, this is defined differently in Section \ref{section::positivity}. 

\subsection{Arithmetic intersection pairings}
\label{section::arithmeticintersectionpairing}
Given vertically integrable $\Ltil_i = (L_i, \{\Vert \cdot \Vert_{i,\nu} \}_{\nu \in \Sigma(K)}) \in \widehat{\Pic}(X)$ ($0 \leq i \leq d^\prime \leq d$), we next define an \textit{arithmetic intersection number}
\begin{equation*}
\Ltil_1 \cdot \Ltil_2 \cdots \Ltil_{d^\prime+1} \cdot \mathfrak{Z} \in \IR
\end{equation*}
for every $d^\prime$-cycle $\mathfrak{Z}$ on $X$. We rely on Gubler's theory of local heights \cite{Gubler1997,Gubler1998,Gubler2003} for this task.
%Since the arithmetic intersection number should be linear in $\mathfrak{Z}$, it suffices to prescribe it for $\mathfrak{Z}= [Z]$ where $Z$ is an irreducible $K$-subvariety of $X$ of dimension $d^\prime$. 
We start by choosing non-zero rational sections $\mathbf{s}_i: X \dashrightarrow L_i$ such that
\begin{equation}
\label{equation::disjointintersectioncondition}
|\Div(\mathbf{s}_1)| \cap |\Div(\mathbf{s}_2)| \cap \cdots \cap |\Div(\mathbf{s}_{d^\prime+1})| \cap |\mathfrak{Z}| = \emptyset.
\end{equation}
There always exist rational sections $\mathbf{s}_i$ meeting this condition. Each section $\mathbf{s}_i$ defines a $\nu$-metrized pseudo-divisor $\widehat{\Div}_\nu(\mathbf{s}_{i})=(\overline{L}_{i,\nu}, \left\vert\Div(\mathbf{s}_{i})\right\vert, \mathbf{s}_{i})$ for any place $\nu \in \Sigma_f(K)$. By our assumption on vertical integrability, we can use \cite[Theorem 10.6]{Gubler2003} to obtain a collection of (unique) \textit{local heights}
\begin{equation*}
\lambda_{\widehat{\Div}_\nu(\mathbf{s}_{1}), \widehat{\Div}_\nu(\mathbf{s}_{2}),\dots,\widehat{\Div}_\nu(\mathbf{s}_{d^\prime+1})}(\mathfrak{Z}) \in \IR, \ \nu\in \Sigma_f(K).
\end{equation*}

A similar definition for $\nu \in \Sigma_\infty(K)$ is given in \cite[Theorem 10.6]{Gubler2003} under the assumption that the hermitian metrics $\Vert \cdot \Vert_{i,\nu}$ are $\mathscr{C}^\infty$. This assumption can be lifted by using the induction formula (see \cite[Proposition 3.5]{Gubler2003}). For each subvariety $Z$ of $X$ having dimension $d^\prime$, this formula states that the local height $\lambda_{\widehat{\Div}_\nu(\mathbf{s}_{1}), \widehat{\Div}_\nu(\mathbf{s}_{2}),\dots,\widehat{\Div}_\nu(\mathbf{s}_{d^\prime+1})}([Z])$ equals
\begin{multline}
\label{equation::induction_formula}
\lambda_{\widehat{\Div}_\nu(\mathbf{s}_{1}), \widehat{\Div}_\nu(\mathbf{s}_{2}),\dots,\widehat{\Div}_\nu(\mathbf{s}_{d^\prime})}(\Div(\mathbf{s}_{d^\prime + 1}|_Z)) 
\\ - \int_{Z^{\mathrm{an}}_{\IC_\nu}} \log \Vert \mathbf{s}_{d^\prime + 1} \Vert_{d^\prime+1,\nu}  c_1(\Lo_{1,\nu}) \wedge c_1(\Lo_{2,\nu}) \wedge \dots \wedge c_1(\Lo_{d^\prime+1,\nu})
\end{multline}
if all $\Vert \cdot \Vert_{i,\nu}$ are $\mathscr{C}^\infty$ and $\Div(\mathbf{s}_{d^\prime+1}|_Z)$ is considered as a $(d^\prime-1)$-cycle on $Z$. The induction formula can be also used as a recursive definition of the local height (as is done in \cite{Chambert-Loir2006,Chambert-Loir2011, Chambert-Loir2009}). By \cite[Théorème 4.1]{Chambert-Loir2009}, the integrals appearing in \eqref{equation::induction_formula} are always finite. Using a regularization lemma \cite[Théorème 4.6.1]{Maillot2000}, one can then deduce the standard properties for the local heights thus defined in the $\mathscr{C}^0$-case from the $\mathscr{C}^\infty$-case. This is straightforward, but requires some checking. The reader is referred to Appendix \ref{appendixA} for details.

Arithmetic intersection numbers can be simply defined as sums of local heights, by setting
\begin{equation}
\label{equation::definition_intersectnumber}
\Ltil_1 \cdot \Ltil_2 \cdots \Ltil_{d^\prime+1} \cdot [\mathfrak{Z}] = \sum_{\nu \in \Sigma(K)} \delta_\nu \cdot \lambda_{\widehat{\Div}_\nu(\mathbf{s}_{1}), \widehat{\Div}_\nu(\mathbf{s}_{2}),\dots,\widehat{\Div}_\nu(\mathbf{s}_{d^\prime+1})}(\mathfrak{Z}).
\end{equation}
For this to be a valid definition, the right-hand side has to be independent of the chosen rational sections $\mathbf{s}_i$ and all except finitely many summands have to be zero. The former fact follows from the product formula by \cite[Propositions 3.7 and 9.4]{Gubler2003}, and the latter one follows from the compatibility with ordinary intersection theory \cite[Section 6]{Gubler1998}. This compatibility also shows that the intersection numbers in (\ref{equation::definition_intersectnumber}) generalize those defined by Gillet and Soulé \cite{Gillet1990, Soule1992} and their extension by Zhang \cite{Zhang1995a}.

To simplify notation, we write $\Ltil_1 \cdot \Ltil_2 \cdots \Ltil_{d+1}$ instead of $\Ltil_1 \cdot \Ltil_2 \cdots \Ltil_{d+1} \cdot [X]$. From \cite[Theorem 10.6 (b)]{Gubler2003}, we know that
\begin{equation*}
\Ltil_1 \cdot \Ltil_2 \cdots \Ltil_{d^\prime+1} \cdot [Z] = \Ltil_1|_Z \cdot \Ltil_2|_{Z} \cdots \Ltil_{d^\prime+1}|_{Z}
\end{equation*}
for any irreducible subvariety $Z \subset X$ of dimension $d^\prime$.

\begin{lemma}
\label{lemma::intersectionnumber}
Let $\Ltil_1,\dots,\Ltil_{d+1},\Ltil_1^\prime \in \widehat{\Pic}(X)$ be vertically integrable. 
\begin{enumerate}
\item[(a)] (Multilinearity) We have
\begin{equation*}
(\Ltil_1 + \Ltil_1^\prime) \cdot \Ltil_2 \cdots \Ltil_{d^\prime+1} =
\Ltil_1 \cdot \Ltil_2 \cdots \Ltil_{d^\prime+1} +\Ltil_1^\prime \cdot \Ltil_2 \cdots \Ltil_{d^\prime+1}.
\end{equation*}
\item[(b)] (Commutativity) For any permutation $\sigma: \{ 1, \dots, d^\prime + 1 \} \rightarrow \{1, \dots, d^\prime + 1 \}$, we have 
\begin{equation*}
\Ltil_{\sigma(1)} \cdot \Ltil_{\sigma(2)} \cdots \Ltil_{\sigma(d^\prime+1)} = \Ltil_1 \cdot \Ltil_2 \cdots \Ltil_{d^\prime+1}.
\end{equation*}
\item[(c)] (Projection Formula) Let $Y$ be an irreducible, projective $K$-variety, and let $f: Y \rightarrow X$ be a generically finite surjective map of degree $\deg(f)$. Then,
\begin{equation*}
f^\ast \Ltil_1 \cdot f^\ast \Ltil_2 \cdots f^\ast \Ltil_{d+1}= \deg(f) \cdot ( \Ltil_1 \cdot \Ltil_2 \cdots \Ltil_{d+1}) .
\end{equation*}
\item[(d)] (Uniform Limits) Given sequences $(\Ltil_i^{(n)}) \in \widehat{\Pic}(X)^{\IN}$, $i \in \{ 1, \dots, d\}$, of vertically semipositive metrized line bundles such that $\Ltil_i^{(n)}$ converges uniformly to $\Ltil_i$, we have
\begin{equation*}
\Ltil_1^{(n)} \cdot \Ltil_2^{(n)} \cdots \Ltil_{d+1}^{(n)}
\longrightarrow
\Ltil_1 \cdot \Ltil_2 \cdots \Ltil_{d+1} \quad (n \rightarrow \infty).
\end{equation*}
\item[(e)] (Scaling invariance) If $\Ltil_1 = (L_1, \{ \Vert \cdot \Vert_\nu \}_{\nu \in \Sigma(K)})$ and $\Ltil_1^\prime = (L_1, \{ |c|_\nu \cdot \Vert \cdot \Vert_\nu \}_{\nu \in \Sigma(K)})$ for some $c \in K^\times$, we have
\begin{equation*}
\Ltil_1^\prime \cdot \Ltil_2 \cdots \Ltil_{d+1} = \Ltil_1 \cdot \Ltil_2 \cdots \Ltil_{d+1}.
\end{equation*}
\item[(f)] (Base change) If $K^\prime/K$ is a finite extension, then
\begin{equation*}
\Ltil_{1,K^\prime} \cdot \Ltil_{2,K^\prime} \cdots \Ltil_{d+1,K^\prime} = [K^\prime:K] \cdot (\Ltil_1 \cdot \Ltil_2 \cdots \Ltil_{d+1}).
\end{equation*}
\end{enumerate}
\end{lemma}

\begin{proof} Each of the first four statements follows from the respective property of Gubler's local heights (\cite[Theorem 10.6]{Gubler2003}), which extends to archimedean $\mathscr{C}^0$-metrics by Appendix \ref{appendixA}.

Statement (e) is a consequence of \cite[Propositions 3.7 and 9.4]{Gubler2003} and the product formula. 

Statement (f) is a consequence of our normalization, as can be seen from the induction formula (\cite[Proposition 3.5 or Remark 9.5]{Gubler2003}), using that $[K^\prime:K] = \sum_{\nu^\prime | \nu} [K^\prime_{\nu^\prime}: K_\nu]$. 
\end{proof}

We need also an extension of the projection formula mentioned in the above lemma, which is a slight generalization of \cite[Proposition 1.3]{Moriwaki2000} to our setting.

\begin{lemma} 
\label{lemma::heightfibrations}
Let $Y$ be an irreducible, projective $K$-variety, and let $f: Y \rightarrow X$ be a proper surjective map. Set $d=\dim(X)$ as well as $d^\prime=\dim(Y)$. Then,
\begin{equation*}
\Ltil_1 \cdot \Ltil_2 \cdots \Ltil_{d-d^\prime} \cdot f^\ast \Ntil_{1} \cdot f^\ast \Ntil_2 \cdots f^\ast \Ntil_{d^\prime+1} = (L_{1,\eta_Y} \cdot L_{2,\eta_Y} \cdots L_{d-d^\prime,\eta_Y}) ( \Ntil_{1} \cdot \Ntil_2 \cdots \Ntil_{d^\prime+1} ).
\end{equation*}
for all $\Ltil_1,\dots,\Ltil_{d-d^\prime} \in \widehat{\Pic}(Y)$, $\Ntil_{1},\dots,\Ntil_{d^\prime+1} \in \widehat{\Pic}(X)$.
\end{lemma}

\begin{proof} 
The proof of \cite[Proposition 2.3]{BurgosGil2016a} can be straightforwardly adapted to our situation, starting from Lemma \ref{lemma::intersectionnumber} (c) and using the induction formula \cite[Proposition 3.5 and Remark 9.5]{Gubler2003}; note that, for any place $\nu \in \Sigma(K)$, the integral occurring in the induction formula is finite by \cite[Théorème 4.1]{Chambert-Loir2009}.
\end{proof}

We conclude with two further direct consequences of the induction formula \cite[Proposition 3.5 and Remark 9.5]{Gubler2003}. Fix some $\nu \in \Sigma(K)$ and $f \in \mathscr{C}^0(X^{\mathrm{an}}_{\IC_\nu})$ such that $\overline{\caO}_X(f) \in \overline{\Pic}_\nu(X)_{\IQ}$ is integrable. (This implies that $\widetilde{\caO}_X(f)$ is vertically integrable.) Then, we have
\begin{equation}
\label{equation::intersection_integral}
\widetilde{L}_1 \cdot \widetilde{L}_2 \cdots \widetilde{L}_{d} \cdot \widetilde{\caO}_X(f)^i
= \delta_\nu \int_{X^{\mathrm{an}}_{\IC_\nu}} f c_1(\Lo_{1,\nu}) \wedge c_1(\Lo_{2,\nu}) \wedge \dots \wedge c_1(\Lo_{d+1-i,\nu}) \wedge c_1(\overline{\caO}_{X}(f))^{i-1}.
\end{equation}
If $\nu \in \Sigma_\infty(K)$ and we consider $\kappa \in \IR$ as a constant function on $X_{\IC_\nu}^{\mathrm{an}}$, we hence have
\begin{equation*}
(\Ltil+ \widetilde{\caO}_{X}(\kappa))^{d+1} = \Ltil^{d+1} + \delta_\nu \kappa (d+1) L^d.
\end{equation*}
If $L$ is nef, this means
\begin{equation}
\label{equation::intersection_Lc}
(\Ltil+ \widetilde{\caO}_{X}(\kappa))^{d+1} = \Ltil^{d+1} + \delta_\nu \kappa (d+1) \vol(L)
\end{equation}
by the algebraic Riemann-Roch Theorem (\cite[Corollary 1.4.41]{Lazarsfeld2004}).

As for the Borel measures defined in Subsection \ref{section::borelmeasure}, Lemma \ref{lemma::intersectionnumber} (a) allows to extend the definition of the arithmetic intersection number to vertically integrable elements of $\widehat{\Pic}(X)_\IQ$. The above results evidently remain valid in this generality.

\subsection{Heights}
\label{section::heights}

Using the intersection numbers defined above, we can define the \textit{height} $h_{\Ltil}(Y)$ of an irreducible subvariety $Y \subseteq X$ with respect to a metrized line bundle $\Ltil = (L, \{\Vert \cdot \Vert_\nu\}_{\nu \in \Sigma(K)})$ such that $L$ is ample. In fact, we set
\begin{equation}
\label{equation::definition_height}
h_{\Ltil}(Y)= \frac{(\Ltil|_Y)^{d+1}}{[K:\IQ](\dim(Y)+1)(L|_{Y})^{d}}
\end{equation}
for an irreducible subvariety $Y \subseteq X$ of pure dimension $d$.

We can make the above definition more explicit if $Y$ is a closed point $x \in X$. As $(L|_x)^{0} = [K(x):K]$ (compare e.g.\ \cite[Proposition VI.2.7]{Kollar1996}), the definition in \eqref{equation::definition_height} simplifies to $h_{\Ltil}(x) = (\Ltil|_x)/[K(x):K]$. Recall that we write $\mathbf{O}_\nu(x)$ for $(x \otimes_K \IC_\nu)^{\mathrm{an}}$. Combining (\ref{equation::definition_intersectnumber}) with the induction formula \cite[Proposition 3.5 and Remark 9.5]{Gubler2003}, we obtain
\begin{equation}
\label{equation::height_section}
h_{\Ltil}(x) = - \frac{1}{[K(x):\IQ]} \sum_{\nu \in \Sigma(K)} \sum_{y \in \mathbf{O}_\nu(x)}
 \delta_\nu \log \Vert \mathbf{s}(y) \Vert_{\nu}
\end{equation}
for any non-zero rational section $\mathbf{s} \in H^0(X,L)$ such that $x \notin\left\vert \Div(\mathbf{s}) \right\vert$.

%One can verify directly that the height is invariant under $\mathrm{Gal}(\IQbar/K)$. Formally, this means the following: Each automorphism $\sigma \in \mathrm{Gal}(\IQbar/K)$ induces a map $\sigma_X: X \rightarrow X$ over $\sigma: \IQbar \rightarrow \IQbar$. Invariance of the height means that $h_{\widetilde{L}}(Y)=h_{\widetilde{L}}(\sigma_X(Y))$ for all subvarieties $Y \subseteq X$. The proof is by induction on the dimension of $Y$ and amounts to a permutation of the inner sum in \eqref{equation::height_section} for $\dim(Y)=0$.

Our height is compatible with base changes in the following sense: For every finite extension $K^\prime \supseteq K$, every irreducible variety $Y\subseteq X$, and every irreducible component $Z \subseteq Y_{K^\prime}$, we have an equality $h_{\widetilde{L}}(Y) = h_{\widetilde{L}_{K^\prime}}(Z)$.

\subsection{Positivity} 
\label{section::positivity}
Having arithmetic intersection numbers and heights at our disposal, we collect here various notions of positivity for metrized line bundles.

An algebraic adelic metric $\{ \Vert \cdot \Vert_{\nu} \}_{\nu \in \Sigma(K)}$ on $L$ that arises from an $S$-model $(\mathcal{X},\mathcal{L})$ of $(X,L^{\otimes e})$ and a hermitian line bundle $\overline{\mathcal{L}} = (\mathcal{L}, \{ \Vert \cdot \Vert_\nu \}_{\nu\in \Sigma_\infty(K)}) \in \widehat{\Pic}(\mathcal{X})$ is called
\begin{enumerate}
\item[(a)] \textit{vertically semipositive} if $\mathcal{L}$ is relatively nef with respect to $\mathcal{X} \rightarrow S$ and each $(L, \Vert \cdot \Vert_\nu )$, $\nu \in \Sigma_\infty(K)$, is semipositive,
\item[(b)] \textit{horizontally semipositive} if $h_{(L, \{ \Vert \cdot \Vert_\nu\}) }(x) \geq 0$ for all closed points $x \in X$,
\item[(c)] \textit{semipositive} if it is both vertically and horizontally semipositive.
\end{enumerate}

A metrized line bundle is called \textit{semipositive} if its adelic metric is the uniform limit of semipositive algebraic adelic metrics. It is called \textit{integrable} if it is the difference of two semipositive metrized line bundles.

For a vertically semipositive $\Ltil \in \widehat{\Pic}(X)$, each $\nu$-metrized line bundle $\Lo_\nu \in \overline{\Pic}_\nu(X)$, $\nu \in \Sigma(K)$, is semipositive. Consequently, an integrable metrized line bundle is also vertical integrable in the sense of Subsection \ref{section::adelicmetrics}. An element $\widetilde{L} \in \widehat{\Pic}(X)_\IQ$ is called \textit{semipositive} (resp.\ \textit{integrable}) if there exists some integer $n \geq 1$ such that $n\widetilde{L}$ is contained in the image of $\widehat{\Pic}(X)$ and semipositive (resp.\ integrable) according to the above definition.

% This follows from the definitions if $\nu \in \Sigma_f(K)$, and form the fact that $dd^c$ is weakly continous a uniform limit of plurisubharmonic functions is plurisubharmonic (\cite[Theorem 2.9.14 (iii)]{Klimek1991}) if $\nu \in \Sigma_\infty(K)$.

%\footnote{Our choice of notation in this subsection is dictated by Lemma \ref{lemma::siu} below. It seems, however, that the lemma remains true after replacing ``(a) ...'' by ``(a) vertically semipositive, if each $\overline{L}_\nu$, $\nu \in \Sigma(K)$, is semipositive.'' The intermediate step through a uniform approximation by algebraic adelic metrics may also not be necessary then because only (a) needs an $S$-model $(\mathcal{X}, \mathcal{L})$ in its formulation. This would lead to more efficient notations, but also seems more laborious.}

\subsection{Arithmetic volumes}
\label{section::arithmeticvolumes}
Let again $\Ltil = (L, \{ \Vert \cdot \Vert_\nu \}_{\nu \in \Sigma(K)})$ be a metrized line bundle on $X$ and write $\Ltil^{\otimes N} = (L^{\otimes N}, \{ \Vert \cdot \Vert^{\otimes N}_\nu \} )$. For each integer $N \geq 1$, we consider the global sections $V_N=H^0(X,L^{\otimes N})$ as a $r_N$-dimensional $K$-vector space and form the tensor product $V_{N,\IA} = V_N \otimes_K \IA_K$ with the adeles $\IA_K$ of $K$. For each $\nu \in \Sigma(K)$, we can additionally endow $V_N$ with a sup-norm 
\begin{equation*}
\Vert \mathbf{s} \Vert_{\nu}^{(\infty)} = \max_{x \in X_{\IC_\nu}^{\mathrm{an}}} \{ \Vert \mathbf{s}(x) \Vert_\nu^{\otimes N} \}, \ \mathbf{s} \in V_N,
\end{equation*}
which extends to a $K_\nu$-linear norm $\Vert \cdot \Vert_{\nu}^{(\infty)}$ on $V_{N,\nu} = V_N \otimes_K K_\nu$. %For $\mathbf{s}=(\dots,\mathbf{s}_\nu,\dots) \in V_\IA$, we write $\Vert \mathbf{s} \Vert_{\nu}^{(\infty)}$ instead of $\Vert \mathbf{s}_\nu \Vert_{\nu}^{(\infty)}$. 

Assume first that $V_N \neq \{ 0 \}$. Since $V_N$ is a co-compact subgroup of $V_{N,\IA}$, there is a unique invariant Haar measure $\vol_N(\cdot)$ on $V_{N,\IA}$ that is normalized such that the induced quotient measure on $V_{N,\IA}/V_N$ has total mass $1$. We then define the adelic unit ball
\begin{equation*}
B_N=\{ \mathbf{s}=(\dots,\mathbf{s}_\nu,\dots) \in V_{N,\IA} \ | \ \forall \nu \in \Sigma(K):\Vert\mathbf{s}_\nu \Vert_{\nu}^{(\infty)} \leq 1 \}.
\end{equation*}
We claim that $\vol_N(B_N)$ is a non-zero real. By coherence, it suffices to prove this for algebraic adelic metrics (i.e., those induced by hermitian line bundles on $S$-models). For these, $\vol_N(B_N) \in \IR_{\geq 0}$ follows from the comparision of lattice norms and sup-norms (e.g.\ \cite[Theorem 5.14 and Lemma 5.15 (iii)]{Boucksom2018}). This allows us to set 
\begin{equation*}
\chisup(\Ltil^{\otimes N})= \log \vol_N(B_N).
\end{equation*}
If $V_N = \{ 0\}$, we simply set $\chisup (\Ltil^{\otimes N}) = 0$. 

The \textit{arithmetic volume of a metrized line bundle} $\Ltil\in \widetilde{\Pic}(X)$ is defined by
\begin{equation*}
\widehat{\vol}_{\chi}(\Ltil) = \limsup_{N \rightarrow \infty} \frac{\chisup(\Ltil^{\otimes N})}{N^{d+1}/(d+1)!}.
\end{equation*}
If the adelic metric of $\widetilde{L}$ is induced by a hermitian line bundle $\overline{\mathcal{L}}$ on some $S$-model of $X$, the volume $\widehat{\vol}_{\chi}(\Ltil)$ agrees with the volume denoted $\widehat{\vol}^{\widehat{\chi}}(\overline{\mathcal{L}})$ in \cite{Ikoma2013}. We collect some standard results on $\widehat{\vol}_\chi(\cdot)$ in the following lemma.

\begin{lemma} 
\label{lemma::arithmeticvolumes_basics}
Let $\widetilde{L} = (L, \{ \Vert \cdot \Vert_\nu \})$ be a metrized line bundle.
\begin{enumerate}
\item[(a)] For any $\nu \in \Sigma_\infty(K)$ and any real $\kappa \in \IR$, which is considered as a real-valued constant function on $X_{\IC_\nu}^{\mathrm{an}}$, we have
\begin{equation*}
\volh_\chi(\Ltil(\kappa)) = \volh_\chi(\Ltil) + \delta_\nu \kappa (d+1) \vol(L).
\end{equation*}
\item[(b)] Let $\Ltil_i = (L, \{ \Vert \cdot \Vert_\nu \} )$ be a sequence of metrized line bundles such that the adelic metric of $\Ltil_i$ converges to $\Ltil$. If $L$ is big, we have
\begin{equation*}
\frac{|\volh_\chi(\Ltil_i)-\volh_\chi(\Ltil)|}{\vol(L)} \longrightarrow 0
\end{equation*}
Otherwise, we have $\volh_\chi(\Ltil_i)=\volh_\chi(\Ltil)$ for all $i$.
\item[(c)] Assume that the adelic metric of $\widetilde{L}$ is a uniform limit of algebraic adelic metrics. For each integer $k \geq 1$, we have $\volh_\chi(\Ltil^{\otimes k}) = k^{d+1}\volh_\chi(\Ltil)$. 
\end{enumerate}
\end{lemma}
\begin{proof} Part (a) follows from $\chisup(\widetilde{L}^{\otimes N}(\kappa)) =\chisup(\widetilde{L}^{\otimes N}) + \kappa \cdot r_N N$ for real $\nu$ and similarly for complex $\nu$.\footnote{Note that if $\vol_{K_\nu}(\cdot)$ is the Haar measure on $K_\nu$, $\nu \in \Sigma(K)$, we have $\vol_{K_\nu}(c_\nu S)= |c_\nu|_\nu^{\delta_\nu} \vol_{K_\nu}(S)$ for all measurable sets $S \subseteq K_\nu$ and each $c_\nu \in K_\nu$.} 

For part (b), we define a $\mathscr{C}^0$-function $e^{-\phi_{i,\nu}} = \Vert \cdot \Vert_{i,\nu}/ \Vert \cdot \Vert_{\nu}$ on $X_{\IC_\nu}^{\mathrm{an}}$ for each $i$ and $\nu \in \Sigma(K)$. There exists a finite subset $\Sigma \subset \Sigma(K)$ such that $\phi_{i,\nu} = 0$ for all $\nu \in \Sigma(K) \setminus \Sigma$. For each $\nu \in \Sigma$, we have $\phi_{i,\nu} \rightarrow 0$ ($i \rightarrow \infty$) uniformly on $X_{\IC_\nu}^{\mathrm{an}}$. Furthermore, we have
\begin{equation}
\label{equation::chisup_bound}
\chisup(\Ltil^{\otimes N}) - \vert \phi_i \vert \cdot r_N N \leq \chisup(\Ltil_i^{\otimes N}) \leq \chisup(\Ltil^{\otimes N}) + \vert \phi_i \vert \cdot r_N N
\end{equation}
where 
\begin{equation*}
\vert \phi_i \vert = \sum_{\nu \in \Sigma_f(K)} \vert \phi_{i,\nu} \vert _{\mathrm{sup}}+ \sum_{\text{real } \nu \in \Sigma_\infty(K)} \vert \phi_{i,\nu} \vert_{\mathrm{sup}} + 2 \sum_{\text{complex } \nu \in \Sigma_\infty(K)} \vert \phi_{i,\nu} \vert_{\mathrm{sup}},
\end{equation*}
which converges to $0$ as $i \rightarrow \infty$. Dividing \eqref{equation::chisup_bound} by $N^d/(d+1)!$ yields the claim as $i \rightarrow \infty$.

Part (c) is \cite[Theorem 3.3.2]{Ikoma2013} %\footnote{Ikoma uses always $\IQ$ as a base field so that $X$ would be considered as a $\IQ$-variety in \cite{Ikoma2013}. However, his volume ``$\volh^{\widehat{\chi}}(\Ltil)$'' agrees with our $\volh_\chi(\Ltil)$ if the adelic metric of $\Ltil$ is algebraic.} 
in the case where the adelic metric of $\widetilde{L}$ is algebraic. Using part (b) of the lemma, we can extend this result.
\end{proof}

Lemma \ref{lemma::arithmeticvolumes_basics} (c) allows us to extend the definition of $\widehat{\vol}_\chi(\cdot)$ to $\widehat{\Pic}(X)_\IQ$.
The next lemma is a straightforward consequence of Ikoma's version of Yuan's bigness theorem \cite{Yuan2008}.

\begin{lemma} 
\label{lemma::siu}
Let $\Ltil, \Mtil \in \widehat{\Pic}(X)_\IQ$ be semipositive. Then,
\begin{equation*}
\volh_\chi(\Ltil - \Mtil) \geq \Ltil^{d+1} - (d+1) \Ltil^d \cdot \Mtil.
\end{equation*}
\end{lemma}

\begin{proof} Because of homogenity (Lemma \ref{lemma::intersectionnumber} (a) and Lemma \ref{lemma::arithmeticvolumes_basics} (c)), we can assume that $\widetilde{L}, \widetilde{M} \in \widehat{\Pic}(X)$. By assumption, there exists a sequence $(\Ltil_i)$ (resp.\ $(\Mtil_i)$) of semipositive algebraically metrized line bundles whose adelic metrics converge uniformly towards the metric of $\Ltil$ (resp.\ $\Mtil$). We can assume that both $\widetilde{L}_i$ and $\widetilde{M}_i$ are given by hermitian line bundles on the same $S$-model $\mathcal{X}_i$ of $X$.\footnote{In fact, if $\Ltil_i$ and $\Mtil_i$ are induced from hermitian line bundles $\overline{\mathcal{L}}_i \in \widehat{\Pic}(\mathcal{X}_i)$ and $\overline{\mathcal{M}}_i \in \widehat{\Pic}(\mathcal{X}_i^\prime)$, then we can replace $\mathcal{X}_i$ and $\mathcal{X}_i^\prime$ with the Zariski closure $\mathcal{X}$ of the diagonally embedded copy of $X$ in $\mathcal{X}_i \times \mathcal{X}_i^\prime$ and the line bundles $\overline{\mathcal{L}}_i$ and $\overline{\mathcal{M}}_i$ with their pullbacks to $\mathcal{X}$.} Let $\varepsilon>0$ be a real number. By Lemma \ref{lemma::arithmeticvolumes_basics} (b), we have
\begin{equation*}
\volh_\chi(\Ltil - \Mtil) \geq \volh_\chi(\Ltil_i - \Mtil_i)- \varepsilon
\end{equation*}
for all integers $i \gg_{\varepsilon,(\Ltil_i),(\Mtil_i)} 1$. Lemma \ref{lemma::intersectionnumber} (d) implies
\begin{equation*}
\Ltil_{i}^{d+1}-(d+1) \Ltil_{i}^d \cdot \Mtil_{i}
\geq
\Ltil^{d+1}-(d+1)\Ltil^d \cdot \Mtil- \varepsilon
\end{equation*}
for all integers $i \gg_{\varepsilon,(\Ltil_i),(\Mtil_i)} 1$. According to \cite[Theorem 3.5.3 and Remark 3.5.4]{Ikoma2013}, we have
\begin{equation*}
\volh_\chi(\Ltil_i - \Mtil_i) \geq \Ltil_i^{d+1} -(d+1) \Ltil_{i}^d \cdot\Mtil_{i} .
\end{equation*}
Combining the three above inequalities, we obtain
\begin{equation*}
\volh_\chi(\Ltil - \Mtil) \geq \Ltil^{d+1} -(d+1) \Ltil^d \cdot \Mtil - 2\varepsilon.
\end{equation*}
Taking the limit $\varepsilon \rightarrow 0$ finishes the proof.
\end{proof}

\subsection{Minkowski's Theorem}
\label{section::minkowski}

The following lemma is typical in applications of Arakelov theory to diophantine geometry.

\begin{lemma} 
\label{lemma::minkowski}
Let $\nu \in \Sigma(K)$ and a real $\varepsilon>0$ be given. For each $\Ltil = (L, \{ \Vert \cdot \Vert_\nu \}) \in \widehat{\Pic}(X)$ with nef and big $L$, there exists a non-zero section $\mathbf{s} \in H^0(X, L^{\otimes N_0})$ and a sufficiently large positive integer $N_0$, such that
\begin{equation*}
\delta_\nu \log \Vert \mathbf{s} \Vert_{\nu}^{(\infty)} \leq \left( - \frac{\volh_\chi(\Ltil)}{(d+1)L^d} + \varepsilon\right) N_0
\end{equation*}
and
\begin{equation*}
\log \Vert \mathbf{s} \Vert_\mu^{(\infty)}  \leq 0
\end{equation*}
for all other places $\mu \neq \nu$ of $K$.
\end{lemma}

\begin{proof} In this proof, we use the notations $V,V_N, r_N, V_{N,\nu}, B_N, \vol_N$ as in Subsection \ref{section::arithmeticvolumes}. The lemma is implied by an adelic version of Minkowski's Second Theorem \cite[Theorem C.2.11]{Bombieri2006}.  To use the theorem, fix an identification $V_N \approx K^{r_N}$ and note that the Haar measure used there agrees with $\vol_N$ (\cite[Proposition C.1.10]{Bombieri2006}). Setting
\begin{multline*}
S = \left\{ \mathbf{s} \in V_{N,\nu} \ | \ \delta_\nu \log \Vert \mathbf{s}  \Vert_{\nu}^{(\infty)} \leq - \frac{\chisup(\Ltil^{\otimes N})}{r_N} + \log(2)[K:\IQ]\right\} \\ \times \prod_{\mu \neq \nu}\{ \mathbf{s} \in V_{N,\mu} \ | \ \Vert \mathbf{s} \Vert_{\mu}^{(\infty)} \leq 1 \} \subseteq V_\IA,
\end{multline*}
we have
\begin{equation*}
\vol_N(S)= \exp\left(-\chisup(\Ltil^{\otimes N})+\log(2)[K:\IQ]r_N\right) \vol_N(B_N)\geq 2^{[K:\IQ]r_N}.
\end{equation*}
The theorem yields hence a non-zero section $\mathbf{s} \in S$, which means that
\begin{align*}
\delta_\nu \log \Vert \mathbf{s} \Vert_\nu^{(\infty)} 
&\leq  - \frac{\chi_{\sup}(\Ltil^{\otimes N})}{r_N} + \log(2) [K:\IQ] \\
&=
\left( - \frac{1}{(d+1)} \cdot \frac{\chi_{\sup}(\Ltil^{\otimes N})}{N^{d+1}/(d+1)!} \cdot \frac{N^d/d!}{r_N} + \frac{\log(2) [K:\IQ]}{N}
\right) N
\end{align*}
and $\log \Vert \mathbf{s} \Vert_\mu^{(\infty)} \leq 0$ for all other places $\mu \neq \nu$ of $K$. Using the algebraic Riemann-Roch Theorem (\cite[Corollary 1.4.41]{Lazarsfeld2004}), we infer the assertion of the lemma.
\end{proof}

\section{Semiabelian Varieties}
\label{section::semiabelian}

We start by summarizing the basic facts on homomorphisms and compactifications of semiabelian varieties that are essential for our main proof. The reader may also compare with \cite[Sections 1 and 2]{Kuehne2017a}. % We then continue in Section \ref{section::semiabelianvarietyheights} by constructing the metrized line bundles needed to apply arithmetic intersection theory. Finally, we establish a local trivialization of a complex semiabelian variety with certain good properties that is useful for the study of measures.

\subsection{Basics} 
\label{section::basics}
A general reference for this subsection is \cite[Chapter VII]{Serre1988} and we only summarize briefly what we need.
A semiabelian variety $G$ over a field $k$ is a connected smooth algebraic $k$-group that is the extension of an abelian variety $A$ over $k$ by a $k$-torus $T$. This means that there exists an exact sequence
\begin{equation}
\label{equation::semiabelian_def}
0 \longrightarrow T \longrightarrow G \longrightarrow A \longrightarrow 0
\end{equation}
in the abelian category of commutative $k$-algebraic groups of finite type (\cite[Th\'eor\`eme VI$_\text{A}$.3.2]{SGA3I}). Both $T$ and $A$ are uniquely determined by $G$ so that we may call $T$ \textit{the} toric part of $G$ and $G \rightarrow G/T=A$ (or just $A$) \textit{the} abelian quotient of $G$ in the following. Furthermore, the exact sequence (\ref{equation::semiabelian_def}) describes a Yoneda extension class in $\Ext^1_k(A,T)$ (see \cite[Section I.3]{Oort1966}). In the sequel, we write $\eta_G$ for the extension class associated to a semiabelian variety $G$ in this way. Each homomorphism $\varphi: B \rightarrow A$ (resp.\ $\varphi: T \rightarrow S$) of abelian varieties (resp.\ tori) induces a pullback $\varphi^\ast: \Ext^1_k(A,T) \rightarrow \Ext^1_k(B,T)$ (resp.\ a pushforward $\varphi_\ast: \Ext^1_k(A,T) \rightarrow \Ext^1_k(A,S)$). 

The Weil-Barsotti formula (see \cite[Section III.18]{Oort1966} or the appendix to \cite{Moret-Bailly1981}) gives a canonical identification $\Ext^1_k(A,\mathbb{G}_m) = A^\vee(k)$. If $T$ is split so that we can identify it with $\Gm^t$, we make frequent use of the identity $\Ext^1_k(A,\mathbb{G}_m^t) = \Ext^1_k(A,\mathbb{G}_m)^t = (A^\vee)^t(k)$. With this identification, it is easy to describe pushforwards. If $\varphi: \mathbb{G}_m^{t} \rightarrow \mathbb{G}_m^{t^\prime}$ is the homomorphism described by $\varphi^\ast(Y_v)= \prod_{u=1}^{t} X_u^{a_{uv}}$ in standard coordinates $X_1,\dots, X_{t}$ (resp.\ $Y_1,\dots, Y_{t^\prime}$) on $\mathbb{G}_m^{t}$ (resp.\ $\mathbb{G}_m^{t^\prime}$), then the pushforward $\varphi_\ast: \Ext^1_k(A,\Gm^{t}) \longrightarrow \Ext^1_k(A,\Gm^{t^\prime})$ corresponds to the homomorphism $(A^\vee)^{t} \rightarrow (A^\vee)^{t^\prime}$ sending $(\eta_1, \dots, \eta_{t})$ to $(\sum_{u=1}^{t}a_{uv}\eta_u)_{1 \leq v\leq t^\prime}$.

\subsection{Homomorphisms and compactifications}

\label{section::semiabelian_compact}

Let $G$ (resp.\ $G^\prime$) be a semiabelian variety over a field $k$ with abelian quotient $A$ (resp.\ $A^\prime$) and split toric part $\Gm^t$ (resp.\ $\Gm^{t^\prime}$). Recall from Lemma \cite[Lemma 1]{Kuehne2017a} that homomorphisms $\varphi: G \rightarrow G^\prime$ give rise to commutative diagrams
\begin{equation} \label{equation::morphisms} 
\begin{tikzcd} 
0 \ar[r] & \Gm^t \ar[r] \ar[d, "\varphi_{\mathrm{tor}}"] & G \ar[r] \ar[d, "\varphi"] & A \ar[r] \ar[d, "\varphi_{\mathrm{ab}}"] & 0 \\
0 \ar[r] & \Gm^{t^\prime} \ar[r] & G^\prime \ar[r] & A^\prime \ar[r] & 0
\end{tikzcd}
\end{equation}
with homomorphisms $\varphi_{\mathrm{tor}}: \Gm^t \rightarrow \Gm^{t^\prime}$ and $\varphi_{\mathrm{ab}}: A \rightarrow A^\prime$. Conversely, a (unique) homomorphism $\varphi$ exists for a pair $(\varphi_{\mathrm{tor}}, \varphi_{\mathrm{ab}})$ if and only if $(\varphi_{\mathrm{tor}})_\ast\eta_G = \varphi_{\mathrm{ab}}^\ast \eta_{G^\prime} \in \Ext^1_k(\Gm^{t^\prime},A)$.

We use the compactification $\overline{G}$ of $G$ given in \cite[Construction 5]{Kuehne2017a}, which differs from the one used in \cite{Chambert-Loir2000}. This choice of compactification is notationally convenient for us but other compactifications are equally legitimate. Indeed, these lead to a different Arakelov height whose main features are however the same than those of the Arakelov height used here: Torsions points have height $0$ and the height of $\overline{G}$ is strictly negative (compare \cite[Section 4]{Chambert-Loir2000}). Consider $\Gm^t$ with its natural $\Gm^t$-action and endow $(\IP^1)^t$ with the unique $\Gm^t$-action $\cdot_{(\IP^1)^t}$ such that the inclusion $\Gm^t \hookrightarrow (\IP^1)^t$ is $\Gm^t$-equivariant. We endow $G \times_k (\IP^1)^t$ with the $\Gm^t$-action given by
\begin{equation*}
t \cdot (g, x) = (t \cdot_G g, t^{-1} \cdot_{(\IP^1)^t} x), \ t \in \Gm^t(Z), \ x \in (\IP^1)^t(Z), \ g \in G(Z),
\end{equation*}
on $Z$-points. We can then form the (categorial) quotient $(G \times (\IP^1)^t) / \Gm^t$ in the category of $k$-schemes, which is a smooth $k$-variety $\overline{G}$ into which $G$ embeds. The abelian quotient $\pi: G \rightarrow A$ extends to an algebraic map $\pio: \overline{G} \rightarrow A$. Writing $\eta_G = (Q_1,\dots, Q_{t}) \in A^\vee(k)^t = \Pic(A)^t$, our compactification can be also described as the embedding
\begin{equation}
\label{equation::Gbar}
G \hookrightarrow \overline{G} = \mathrm{Proj}(\mathrm{Sym}(\caO_A \oplus \mathcal{F}(Q_1)^\vee)) \times_A \cdots \times_A \mathrm{Proj}(\mathrm{Sym}(\caO_A \oplus \mathcal{F}(Q_t)^\vee))
\end{equation}
(see \cite[Section 4]{Chambert-Loir1999}). The ideal $\mathrm{Sym}(0 \oplus \mathcal{F}(Q_i)^\vee)$ (resp.\ $\mathrm{Sym}(\caO_A \oplus 0)$) describes a divisor in $\mathrm{Proj}(\mathrm{Sym}(\caO_A \oplus \mathcal{F}(Q_i)^\vee))$. Replacing the factor $\mathrm{Proj}(\mathrm{Sym}(\caO_A \oplus \mathcal{F}(Q_i)^\vee))$ in (\ref{equation::Gbar}) with this divisor, we obtain a Weil divisor $D_i^{(0)}$ (resp.\ $D_i^{(\infty)}$) on $\overline{G}$. The boundary $\overline{G} \setminus G$ is the union of the $2t$ Weil divisors $D_i^{(0)}, D_i^{(\infty)}$ ($1 \leq i \leq t$). Each divisor $D_i^{(\kappa)}$ ($\kappa \in \{ 0, \infty \}$, $1 \leq i \leq t$) gives rise to an effective line bundle $M_{\overline{G},i}^{(\kappa)}$ whose sheaf of sections is $\mathcal{O}_{\overline{G}}(D_i^{(\kappa)})$. In the sequel, we write $M_{\overline{G}} = \otimes_{i=1}^t (M_{\overline{G},i}^{(0)} \otimes M_{\overline{G},i}^{(\infty)})$. By \cite[Lemma 3]{Kuehne2017a}, the line bundle $M_{\overline{G}} \otimes \pio^\ast N$ is ample. For each integer $n \geq 1$, the same lemma shows that $(M_{\overline{G}_i}^{(\kappa)})^{\otimes n} \otimes M_{\overline{G}} \otimes \pio^\ast N$ is ample and hence $M_{\overline{G},i}^{(\kappa)}$ is nef.

We also have to compactify (the graphs of) homomorphisms. A general procedure for this is outlined in \cite[Construction 7]{Kuehne2017a}, but we only need some special cases and we describe these in detail here. Assume that $\varphi: G \rightarrow G^\prime$ is a homomorphism of semiabelian varieties such that $\varphi_{\mathrm{tor}} = [n]_{\Gm^t}$ in the notation of (\ref{equation::morphisms}). The homomorphism $[n]_{\Gm^t}: \Gm^t \rightarrow \Gm^t$ extends to a map $\overline{[n]}_{\Gm^t}: (\IP^1)^t \rightarrow (\IP^1)^t$ so that there is a map
\begin{equation*}
\varphi \times \overline{[n]}_{\Gm^t}: G \times (\IP^1)^t \longrightarrow G^\prime \times (\IP^{1})^{t}.
\end{equation*}
It is easy to see that this map descends to a map
\begin{equation*}
\overline{\varphi}: \overline{G} = G \times^{\Gm^t} (\IP^1)^t \longrightarrow G^\prime \times^{\Gm^t} (\IP^1)^t = \overline{G}^\prime
\end{equation*}
on contraction products, which is the unique extension of $\varphi: G \rightarrow G^\prime$ to an algebraic map $\overline{G} \rightarrow \overline{G}^\prime$. In the particular case where $\varphi=[n]$, we obtain an extension $\overline{[n]}: \overline{G} \rightarrow \overline{G}$ of the multiplication-by-$n$ map $[n]: G \rightarrow G$. One can compute that $\overline{[n]}^\ast D_i^{(\kappa)}= n D_i^{(\kappa)}$ ($\kappa \in \{ 0, \infty \}$, $1 \leq i \leq t$) and hence $\overline{[n]}^\ast M_{\overline{G},i}^{(\kappa)} \approx (M_{\overline{G},i}^{(\kappa)})^{\otimes n}$ (see \cite[Section 2]{Kuehne2018}). 

\subsection{Canonical metrics and heights} 
\label{section::semiabelianvarietyheights}

Our references for this subsection are \cite{Chambert-Loir1999, Chambert-Loir2000, Zhang1995a}. We let $G$ be a semiabelian variety over a number field $K$. Denote by $\pi: G \rightarrow A$ its abelian quotient and by $T$ its toric part. Enlarging $K$ if necessary, we may assume that $T = \Gm^t$. In this situation, the last subsection yields a compactification $\overline{G}$ and line bundles $M_{\overline{G},i}^{(\kappa)}$ ($\kappa \in \{ 0, \infty \}$, $1 \leq i \leq t$) on $\overline{G}$. Let $N$ be an ample symmetric line bundle on $A$. 

We next aim to decorate $M_{\overline{G},i}^{(\kappa)}$ and $N$ with adelic metrics, following Zhang \cite{Zhang1995a}. The technical result that we need is summarized in the following lemma.

\begin{lemma} 
\label{lemma::canonicalmetrics}
Let $X$ be a projective variety, $L$ a line bundle on $X$, $f: X \rightarrow X$ a surjective algebraic map, $d>1$ an integer, and $\phi: L^{\otimes d} \rightarrow f^\ast L$ an isomorphism of line bundles over $X$. Then,
\begin{enumerate}
\item[(a)] for each $\nu \in \Sigma(K)$, there exists a unique $\nu$-metrized line bundle $\Lo_\nu=(L,\Vert \cdot \Vert_\nu)$ such that $\phi$ is an isometry $\Lo^{\otimes d}_\nu \rightarrow f^\ast \Lo_\nu$,
\item[(b)] on replacing $\phi$ with $c \phi$, the metric $\Vert \cdot \Vert_\nu$ changes to $|c|_\nu^{1/(d-1)} \Vert \cdot \Vert_\nu$,
\item[(c)] the metrics $\Vert \cdot \Vert_\nu$ thus obtained combine to an adelic metric $\{ \Vert \cdot \Vert_\nu \}_{\nu \in \Sigma(K)}$ on $L$,
\item[(d)] if there exists a vertically semipositive algebraic adelic metric on $L$, then the metrized line bundle $\Ltil = (L, \{ \Vert \cdot \Vert_\nu \}_{\nu \in \Sigma(K)})$ is vertically semipositive.
\end{enumerate}
\end{lemma}

\begin{proof} (a), (b): This is the content of \cite[Theorem 2.2]{Zhang1995a}. The $\Gal(\IC_\nu/K_\nu)$-invariance demanded by property (d) in Subsection \ref{section::metrics} follows from uniqueness.

(c): This is an extension of the ample case considered in \cite[(2.3)]{Zhang1995a}. Let $L_0$ be a very ample line bundle on $X$ such that $L \otimes L_0$ is also very ample. In other words, the global sections yield projective embeddings $\iota_{L_0}: X \hookrightarrow \IP^{k_1}_K$ and $\iota_{L \otimes L_0}: X \hookrightarrow \IP^{k_2}_K$. Composing with the diagonal map, we thus obtain an embedding $\iota: X \hookrightarrow \IP^{k_1}_K \times \IP^{k_2}_K$ such that $L_0 = \iota^\ast \pr_1^\ast \caO_{\IP^{k_1}_K}(1)$ and $L\otimes L_0 = \iota^\ast \pr_2^{\ast} \caO_{\IP^{k_2}_K}(1)$ where $\pr_i$ ($i=1,2$) denotes the projection to the $i$-th factor. Let $\mathcal{X}$ be the Zariski closure of $X$ in $\IP_S^{k_1} \times \IP_S^{k_2}$. We set $\mathcal{L} = ( \pr_2^{\ast} \caO_{\IP^{k_2}_S}(1) \otimes \pr_1^\ast \caO_{\IP^{k_1}_S}(1))|_\mathcal{X}$ so that $(\mathcal{X}, \mathcal{L})$ is an $S$-model of $(X, L)$. 

For each integer $n \geq 1$, we consider  the graph embedding $\iota_\Gamma=(\id_{X}, f^{\circ n}): X \hookrightarrow \Gamma(f^{\circ n}) \subset X \times X$ and the Zariski closure $\overline{\Gamma(f^{\circ n})}$ of $\Gamma(f^{\circ n})$ in $\mathcal{X} \times \mathcal{X}$. Writing $\mathrm{pr}_2: \mathcal{X} \times \mathcal{X} \rightarrow \mathcal{X}$ for the projection to the second factor, we define the line bundle $\mathcal{L}_n = (\mathrm{pr}_2^\ast \mathcal{L})|_{\overline{\Gamma(f^{\circ n})}}$. The isomorphism $\phi$ induces an isomorphism $\phi_n: L^{\otimes d^n} \rightarrow (f^{\circ n})^\ast L$ over $X$. Through $\iota_\Gamma$ and $\phi_n$, the tuple $(\overline{\Gamma(f^{\circ n})}, \mathcal{L}_n)$ is an $S$-model of $(X,L^{\otimes d^n})$. For each $\nu \in \Sigma_f(K)$, we have induced formal $\nu$-metrics $\Vert \cdot \Vert_{\mathcal{L}_{n,\nu}}^{1/d^n}$ on $L$. An inspection of the argument of \cite[Theorem 2.2 (a,b)]{Zhang1995a} shows that the $\nu$-metric $\Vert \cdot \Vert_{\mathcal{L}_{n,\nu}}^{1/d^n}$ converges uniformly to the $\nu$-metric $\Vert \cdot \Vert_\nu$ from (a) as $n \rightarrow \infty$. 

Furthermore, there exists a non-empty open $U \subseteq S$ such that each iterate $f^{\circ n}$ extends to a map $\widetilde{f^{\circ n}}: \mathcal{X}|_U \rightarrow \mathcal{X}|_U$ and each $\phi_n$ extends to an isomorphism $\widetilde{\phi_n}: \mathcal{L}^{\otimes d^n}|_U \rightarrow (\widetilde{f^{\circ n}})^\ast \mathcal{L}|_U$ over $\mathcal{X}|_U$. For each $\nu \in U$, this implies $\Vert \cdot \Vert_\nu = \Vert \cdot \Vert_{\mathcal{L}_{n,\nu}}^{1/d^n} = \Vert \cdot \Vert_{\mathcal{L}_\nu}$, which shows that $\{ \Vert \cdot \Vert_\nu \}_{\nu \in \Sigma(K)}$ is an adelic metric.

(d): By assumption, there exists an $S$-model $(\mathcal{X},\mathcal{L})$ of $(X,L^{\otimes e})$ and a hermitian line bundle $\overline{\mathcal{L}} = (\mathcal{L}, \{ \Vert \cdot \Vert_\nu \}_{\nu\in \Sigma_\infty(K)}) \in \widehat{\Pic}(\mathcal{X})$ such that $\mathcal{L}$ is relatively nef with respect to $\mathcal{X} \rightarrow S$ and each $(L, \Vert \cdot \Vert_\nu )$, $\nu \in \Sigma_\infty(K)$, is semipositive. Proceeding as in (c), we obtain a sequence of algebraic adelic metrics $(\{ \Vert \cdot \Vert_{\mathcal{L}_{n,\nu}}^{1/d^ne}\} )$ that are vertically semipositive and converge uniformly to $\{ \Vert \cdot \Vert_\nu \}$.
\end{proof}

Using $\overline{[n]}^\ast M_{\overline{G},i}^{(\kappa)} \approx (M_{\overline{G},i}^{(\kappa)})^{\otimes n}$ and $\overline{[n]}^\ast N \approx N^{\otimes n^2}$, the above lemma yields metrized line bundles $\widetilde{M}_{\overline{G},i}^{(\kappa)}$ and $\widetilde{N}$. The adelic metrics are unique up to a rescaling as in part (b) of the above theorem. Since this has no influence on arithmetic intersection numbers due to Lemma \ref{lemma::intersectionnumber} (e), we can ignore this indeterminancy in the following and just work with a fixed arbitrary choice. The heights $h_{\Mtil_{\overline{G},i}^{(\kappa)}}$ and $h_{\pio^\ast \widetilde{N}}$ coincide with the corresponding N\'eron-Tate heights $\hhat_{M_{\overline{G},i}^{(\kappa)}}$ and $\hhat_{\pio^\ast N}$ defined in \cite[Lemma 8]{Kuehne2017a}.

Since $N$ is ample, the global sections of some power $N^{\otimes e}$ induce a projective embedding $A \hookrightarrow \IP_K^{k}$. Taking the Zariski closure $\mathcal{A}$ of $A$ in $\IP_S^{k}$, we obtain an $S$-model $(\mathcal{A},\caO_{\IP^k_S}(1)|_{\mathcal{A}})$ of $(A,N^{\otimes e})$. Since $\caO_{\IP^k_S}(1)|_{\mathcal{A}}$ is ample, there exists a vertically semipositive hermitian line bundle $\overline{\mathcal{N}}=(\caO_{\IP^k_S}(1)|_{\mathcal{A}}, \{ \Vert \cdot \Vert_{\nu \in \Sigma_\infty(K)}\})\in \widehat{\Pic}(\mathcal{A})$. The algebraic adelic metric on $N$ induced by $\overline{\mathcal{N}}$ is vertically semipositive. By Lemma \ref{lemma::canonicalmetrics} (d), $\widetilde{N}$ is likewise vertically semipositive. Furthermore, $\widetilde{N}$ is also horizontally semipositive, since this just means that the Néron-Tate height associated with $N$ is non-negative.

As $M_{\overline{G},i}^{(\kappa)}$ is only nef (see \cite[Lemma 3]{Kuehne2017a}), the vertical semipositivity of $\widetilde{M}_{\overline{G},i}^{(\kappa)}$ is not as easy to establish as for $\widetilde{N}$. Nevertheless, Chambert-Loir proved this fact in \cite[Proposition 3.6]{Chambert-Loir2000}, relying on specific regular models of abelian varieties constructed by Künnemann \cite{Kuennemann1998}.

%Finally, let us note that both $\widetilde{M}_{\overline{G},i}^{(\kappa)}$ and $\widetilde{N}$ are Galois-invariant by construction. The associated heights are hence compatible with passing to a finite extension $K^\prime$ of $K$ (see Subsection \ref{section::heights}).

\subsection{Heights and Homomorphisms}

We next recall a lemma that controls the behavior of our canonical height under homomorphisms.

\begin{lemma} 
\label{lemma::heightcomparison}
Let $G_i$, $i=1,2$, be a semiabelian variety with toric part $T_i=\Gm^{t_i}$ and abelian quotient $\pi: G_i \rightarrow A_i$. Let further $N_i$, $i \in \{ 1,2 \}$, be symmetric line bundles on $A_i$. Assume that $N_1$ is ample. For every homomorphism $\varphi: G_1 \rightarrow G_2$, we then have
\begin{equation*}
h_{\overline{\pi}^\ast_2 \widetilde{N}_2}(\varphi(x))
\ll_{N_i, \varphi}
h_{\overline{\pi}^\ast_1 \widetilde{N}_1}(x)
\qquad
\text{and}
\qquad
h_{\widetilde{M}_{\overline{G}_2}}(\varphi(x)) \ll_{\varphi} 
h_{\widetilde{M}_{\overline{G}_1}}(x)
\end{equation*}
for all closed points $x \in G_1$.
\end{lemma}

We note that all four terms in the above two inequalities are non-negative. This is horizontal semipositivity for the terms in the first inequality. For the terms of the second inequality, we can use \cite[Lemma 10]{Kuehne2017a} (see also \cite[Lemme 3.9]{Chambert-Loir2000}).

\begin{proof} The first inequality is equivalent to $h_{\widetilde{N}_2}(\varphi_{\mathrm{ab}}(x)) \ll_{N_i,\varphi} h_{\widetilde{N}_1}(x)$ for all closed points $x \in A_1$. It follows hence from an application of \cite[Proposition 2.3]{Vojta1999} to the map $\varphi: A_1 \rightarrow A_2$. As usual, the term $O(1)$ in \textit{loc.\ cit.}\ disappears by Tate's limit argument. 

The second inequality can be deduced from \cite[Lemma 10]{Kuehne2017a}. Invoking the said lemma for the pair $(\varphi_{\mathrm{tor}},0) \in \Hom(\Gm^{t_1},\Gm^{t_2})$ yields 
\begin{equation}
\label{equation::someheightcomp}
|\hhat_{M_{\overline{\Gamma(\varphi_{\mathrm{tor}})}}}(x)| \ll_\varphi \hhat_{M_{\overline{G}_1}}(x)
\end{equation}
for all $x \in G(\IQbar)$ where $M_{\overline{\Gamma(\varphi_{\mathrm{tor}})}}$ is as in \cite[Construction 6]{Kuehne2017a}. Functoriality \cite[Lemma 9]{Kuehne2017a} implies that $\hhat_{M_{\overline{\Gamma(\varphi_{\mathrm{tor}})}}}(x)=\hhat_{M_{\overline{G}_2}}(\varphi(x))$ (see also \cite[Construction 7]{Kuehne2017a}). We conclude by noting the canonical heights constructed in the last subsection coincide with the corresponding Néron-Tate heights used in \cite{Kuehne2017a}.
\end{proof}

\subsection{Local trivializations}
\label{section::localtrivialization} 
Let $\nu$ be an archimedean place of $K$. For use in Section \ref{section::equilibrium}, we state next an auxiliary result describing the $\nu$-metrized line bundles constructed in Subsection \ref{section::semiabelianvarietyheights}. The construction therein, carried out for $G=\Gm$ and $A=0$, yields metrized line bundles $\Mtil_{\IP^1}^{(\kappa)} = \Mtil_{\IP^1,1}^{(\kappa)}$, $\kappa \in \{ 0, \infty \}$, on the compactification $\IP^1$ of $\Gm$. By Lemma \ref{lemma::canonicalmetrics} (a), we have $[n]^\ast \Mtil_{\IP^1}^{(\kappa)} \approx (\Mtil_{\IP^1}^{(\kappa)})^{\otimes n}$ for any positive integer $n$. We write $\widetilde{M}_{\IP^1} = \widetilde{M}_{\IP^1}^{(0)} \otimes \widetilde{M}_{\IP^1}^{(\infty)}$ and $\widetilde{M}_{\overline{G},i} = \widetilde{M}_{\overline{G},i}^{(0)} \otimes \widetilde{M}_{\overline{G},i}^{(\infty)}$ ($i=1,\dots,t$) to ease notation.

\begin{lemma} 
\label{lemma::standardcoordinates} There exists a finite collection $\{ (U_j, \psi_j)\}_{j \in J}$ where $U_j$, $j \in J$, are open sets covering $A_{\IC_\nu}^{\mathrm{an}}$ and $\psi_j$, $j \in J$, are holomorphic maps
\begin{equation*}
\psi_j= (\psi_1^{(j)},\dots,\psi_t^{(j)}): \overline{G}^{\mathrm{an}}_{\IC_\nu}|_{U_j} \longrightarrow ((\IP^1_{\IC_\nu})^t)^{\mathrm{an}}
\end{equation*}
such that 
\begin{equation*}
\pio^{\mathrm{an}}_{\IC_\nu}|_{U_j} \times \psi_j: \overline{G}^{\mathrm{an}}_{\IC_\nu}|_{U_j} \longrightarrow U_j \times ((\IP^1_{\IC_\nu})^t)^{\mathrm{an}}
\end{equation*}
is a biholomorphism and 
\begin{equation}
\label{equation::pullback_logmetric}
\Mo_{\overline{G},i,\nu}|_{U_j} \approx (\psi_i^{(j)})^\ast \Mo_{\IP^1,\nu}, \ i \in \{ 1, \dots, t \}.
\end{equation}
For each $j, j^\prime \in J$ and each $i \in \{ 1,\dots, t \}$, the quotient 
\begin{equation}
\label{equation::absolutevalue}
\psi_i^{(j)}/\psi_{i}^{(j^\prime)}: (\pi^{\mathrm{an}}_{\IC_\nu})^{-1}(U_j \cap U_{j^\prime}) \rightarrow \IC
\end{equation}
is a locally constant function with values in $S^1 = \{ z \in \IC \ | \ |z|=1 \}$. Furthermore, we have
\begin{equation}
\label{equation::tangentrank}
\ker(d\psi_i^{(j)})_y \cap \ker(d\pio)_y = \ker(d\psi_i^{(j^\prime)})_y \cap \ker(d\pio)_y, \ i \in \{ 1, \dots, t \},
\end{equation}
for any $j, j^\prime \in J$ and all $y \in (\pi^{\mathrm{an}}_{\IC_\nu})^{-1}(U_j \cap U_{j^\prime})$.
\end{lemma}

In the following, Weil functions are used to describe archimedean metrics. The reader is referred to \cite[Chapter 10]{Lang1983} and \cite[Chapter I]{Lang1988} for basics on Weil functions. In addition, \cite[Section 5.1]{Kuehne2017a} and \cite[Section 2]{Vojta1996} contain further information on the specific Weil functions used here.

\begin{proof} Write $\eta_G = (Q_1,Q_2,\dots,Q_t) \in A^\vee(\IQbar)$. Recall from Subsection \ref{section::semiabelian_compact} the Weil divisors $D_i^{(0)}, D_i^{(\infty)}$ ($1 \leq i \leq t$) on $\overline{G}$. Let $\mathbf{s}_i$ ($1 \leq i \leq t$) be a non-zero rational section of $Q_i^\vee$ such that $e_A \notin \left\vert \Div(\mathbf{s}_i) \right\vert$. From \eqref{equation::Gbar}, we see that each $\mathbf{s}_i$ induces a rational function $f_i$ on $\overline{G}$ with divisor 
\begin{equation}
\label{equation::divisor}
\Div(f_i)= D_{i}^{(0)} - D_{i}^{(\infty)} + \overline{\pi}^\ast \Div(\mathbf{s}_i).
\end{equation}
In fact, $\pio^\ast Q_i$ is the line bundle associated with $D_i^{(0)}-D_i^{(\infty)}$ (see \cite[Lemme 4.1]{Chambert-Loir1999}). Let $U$ be an open, simply connected subset of $A_{\IC_\nu}^{\mathrm{an}}$, and let us assume additionally that $U$ and $\Div(\mathbf{s}_i)^{\mathrm{an}}_{\IC_\nu}$ are disjoint. We can then lift the inclusion $\pio^{\mathrm{an}}_{\IC_\nu}|_U: U \hookrightarrow A_{\IC_\nu}^{\mathrm{an}}$ to a map $\widetilde{\pi}: U \rightarrow \IC^{g}$ where $\IC^g$ is interpreted as the universal covering of $A^{\mathrm{an}}_{\IC_\nu}$. By \cite[Chapter X]{Lang1982}, there exists a normalized theta function $\vartheta_i$ on $\IC^g$ whose (analytic) divisor $\Div(\vartheta_i)$ is the pullback of $\Div(\mathbf{s}_i)_{\IC_\nu}^{\mathrm{an}}$ along the universal covering $\IC^g \twoheadrightarrow A_{\IC_\nu}^{\mathrm{an}}$. Rescaling if necessary, we may assume that $f_i(e_G)=1$ and $\vartheta_i(0,\dots,0) = 1$. We define a map $\psi: \overline{G}^{\mathrm{an}}_{\IC_\nu}|_{U} \rightarrow U \times ((\IP^1_{\IC_\nu})^t)^{\mathrm{an}}$ by setting
\begin{equation*}
\psi(y) = \left(\pi(y), \psi_1(y), \dots, \psi_t(y) \right), \ \psi_i(y) = f_i(y)/\vartheta_i(\widetilde{\pi}(y)),
\end{equation*}
for $y \in G^{\mathrm{an}}_{\IC_\nu}|_{U}$; this extends uniquely to a biholomorphism $\overline{G}^{\mathrm{an}}_{\IC_\nu}|_{U} \longrightarrow U \times ((\IP^1_{\IC_\nu})^t)^{\mathrm{an}}$. Varying the rational sections $\mathbf{s}_i$ and the simply connected subset $U \subseteq A_{\IC_\nu}^{\mathrm{an}}$, we can find a finite covering $\{ U_j\}_{j \in J}$ of $A^{\mathrm{an}}_{\IC_\nu}$ as well as associated holomorphisms $\{ \psi_j = (\psi_1^{(j)},\dots,\psi_t^{(j)}) \}_{j \in J}$ in this way. 

To show the assertion about the $\nu$-metrized line bundles, we identify
\begin{equation*}
\Mo_{\overline{G},i,\nu} = (\caO_{\overline{G}} (D_{i}^{(0)}+ D_{i}^{(\infty)}), \Vert\cdot\Vert_{G,i})
\quad \text{and} \quad 
\Mo_{\IP^1,\nu} = (\caO_{\IP^1}([0]+ [\infty]),\Vert \cdot \Vert_{\Gm}).
\end{equation*}
This presentation allows us to interpret the sections of $M_{\overline{G},i}$ and $M_{\IP^1}$ as meromorphic functions on $\overline{G}$ and to work with Weil functions. The pullback of meromorphic functions along $[n]$ gives rise to canonical isomorphisms $\phi_{\overline{G},i}^{(n)}: M_{\overline{G},i}^{\otimes n} \rightarrow [n]^\ast M_{\overline{G},i}$ and $\phi_{\IP^1}^{(n)}: M_{\IP^1}^{\otimes n} \rightarrow [n]^\ast M_{\IP^1}$; by Lemma \ref{lemma::canonicalmetrics} (b) we can assume that these induce isometries $\overline{M}_{\overline{G},i,\nu}^{\otimes n} \rightarrow [n]^\ast \overline{M}_{\overline{G},i,\nu}$ and $\overline{M}^{\otimes n}_{\IP^1,\nu} \rightarrow [n]^\ast \overline{M}_{\IP^1,\nu}$.

 %The pullback of meromorphic functions along $[n]$ induces \textit{unique} isomorphisms $[n]^\ast \Mtil_{G,i}^{(\kappa)} = (\Mtil_{G,i}^{(\kappa)})^{\otimes n}$ and $[n]^\ast \Mtil_{\Gm}^{(\kappa)} = (\Mtil_{\Gm}^{(\kappa)})^{\otimes n}$.

By \cite[Proposition 2.6]{Vojta1996} (and its proof combined with \cite[Theorem 13.1.1]{Lang1983}), we have $|\psi_i^{(j)}(y)| = |\psi_i^{(j^\prime)}(y)|$ for all $y \in U_{j} \cap U_{j^\prime}$. As $\psi_i^{(j)}/\psi_{i}^{(j^\prime)}$ is a meromorphic function on $U_j \cap U_{j^\prime}$, this implies (\ref{equation::absolutevalue}), whence also (\ref{equation::tangentrank}). We can also define $\mathscr{C}^\infty$-functions $\lambda_i: G_{\IC_\nu}^{\mathrm{an}} \rightarrow \IR$ ($1\leq i \leq t$) by setting $\lambda_i(y) = \log |\psi_i^{(j)}(y)|$ for all $y \in U_j$. The mentioned proposition in \cite{Vojta1996} states then that $\lambda_i$ is the unique Weil function for $D_{i}^{(\infty)}-D_{i}^{(0)}$ such that $
\lambda_i(y_1+y_2)=\lambda_i(y_1)+\lambda_i(y_2)$ for all $y_1,y_2 \in G^{\mathrm{an}}_{\IC_\nu}$. Since $\phi_{\overline{G},i}^{(n)}$ induces isometries $\overline{M}_{\overline{G},i,\nu}^{\otimes n} \rightarrow [n]^\ast \overline{M}_{\overline{G},i,\nu}$ and 
\begin{equation*}
(\caO_{\overline{G}} (D_{i}^{(0)}+ D_{i}^{(\infty)}), e^{-|\lambda_i|}|\cdot|)^{\otimes n} \longrightarrow [n]^\ast (\caO_{\overline{G}} (D_{i}^{(0)}+ D_{i}^{(\infty)}), e^{-|\lambda_i|}|\cdot|),
\end{equation*}
Lemma \ref{lemma::canonicalmetrics} (a) implies that $\Vert \cdot \Vert_{\overline{G},i}= e^{-|\lambda_i|}|\cdot|$ where $|\cdot|$ is the ordinary absolute value on rational functions. Similarly, we can argue with $\phi_{\IP^1}^{(n)}$ and obtain that $\Vert \cdot \Vert_{\IP^1} = e^{-|\log|z||}|\cdot|$. Combining these two identities and using $\left\vert \log(\cdot)\right\vert \circ \psi_i^{(j)} = \lambda_i$, we infer that $\Mo_{\overline{G},i,\nu}|_{U_j} \approx (\psi_i^{(j)})^\ast \Mo_{\IP^1,\nu}$, $i \in \{ 1, \dots, t \}$.
\end{proof}

\section{The Equidistribution Conjecture}

\label{section::toricrank1}

As in the introduction, we let $K$ be a number field, $\nu \in \Sigma(K)$ an arbitrary place, $G$ a semiabelian variety over $K$ with split toric part $T=\Gm^t$ and abelian quotient $\pi: G \rightarrow A$. Thus, $G$ is given by an exact sequence
\begin{equation*}
0 \longrightarrow \Gm^t \longrightarrow G \longrightarrow A \longrightarrow 0,
\end{equation*}
which is described by a $t$-tuple $\underline{\eta} = (\eta_1,\dots, \eta_t)\in A^\vee(K)^t = \Ext^1_{K}(A, \Gm^t)$. We may then use the compactification $\overline{G}$ and the map $\overline{\pi}: \overline{G} \rightarrow A$ described in Subsection \ref{section::semiabelian_compact}. Write $\Mtil$ (resp.\ $\Ntil$) for the vertically semipositive metrized line bundle $\Mtil_{\overline{G}} \in \widehat{\Pic}(\overline{G})$ (resp.\ $\Ntil \in \widehat{\Pic}(A)$) defined in Subsection \ref{section::semiabelianvarietyheights} and set $\Ltil = \Mtil + \pio^\ast \Ntil$. %Note that $\Ltil$, $\Mtil_{\overline{G}}$ and $\Ntil$ are Galois-invariant.

Our aim in this section is to prove $(\mathrm{EC})$ in a slighter stronger form, namely for arbitrary subvarieties $X \subseteq G$. This extra strength is needed in Section \ref{section::bogomolov} for the proof of (BC). The case $X = G$ corresponds to (EC).

\begin{proposition} \label{proposition1} Let $X \subseteq G$ be a geometrically irreducible algebraic subvariety and let $\overline{X}$ denote its Zariski closure in $\overline{G}$. Set $d=\dim(X)$, $d^\prime = \dim(\pio(X))$, and define the Borel measure
\begin{equation*}
\mu_\nu = \frac{ c_1(\Mo_{\overline{G},\nu}|_{\overline{X}})^{d-d^\prime} \wedge c_1(\pio^\ast \No_\nu|_{\overline{X}})^{d^\prime}}{(M_{\overline{G}}|_{\overline{X}})^{d-d^\prime} \cdot (\pio^\ast N|_{\overline{X}})^{d^\prime}}
\end{equation*}
on ${\overline{X}}_{\IC_\nu}^\mathrm{an}$. Furthermore, let $(x_i) \in X^{\IN}$ be an $X$-generic sequence of small points. For any $f \in \mathscr{C}^0({\overline{X}}_{\IC_\nu}^{\mathrm{an}})$, we have then
\begin{equation}
\label{equation::equidistribution_explicit2}
\frac{1}{\# \mathbf{O}_\nu(x_i)}\sum_{y \in \mathbf{O}_\nu(x_i)} f(y) \longrightarrow \int_{{\overline{X}}_{\IC_\nu}^{\mathrm{an}}}f \mu_\nu \quad (i \rightarrow \infty).
\end{equation}
\end{proposition}

The proposition is proven at the end of this section, after a series of preparatory lemmas. Before starting with them, we have to introduce some further objects. For each integer $n \geq 1$, we choose (arbitrary) $\eta_i^{(n)} \in A^\vee(\IQbar)$, $1 \leq i \leq t$, such that $n \cdot \eta_i^{(n)} = \eta_i$. Let $G_n$ be the semiabelian variety given by the extension class $(\eta_1^{(n)},\dots, \eta_t^{(n)})\in A^\vee(\IQbar)^t = \Ext^1_{\IQbar}(A, \Gm^t)$. Note that $G_n$ is in general not a semiabelian variety over $K$ but only over a finite extension $K_n \supseteq K$. We consider $G_n$ as a $K_n$-variety in the sequel. From Subsection \ref{section::semiabelian_compact}, we know that there exists an isogeny $\varphi_n: G_n \rightarrow G_{K_n}$ such that, in the notation of (\ref{equation::morphisms}), we have $\varphi_{n,\mathrm{tor}}=[n]$ and $\varphi_{n,\mathrm{ab}}= \id_{A_{K_n}}$. 

We have again a standard compactification $\overline{G}_n$ of $G_n$ and a map $\pio_n: \overline{G}_n \rightarrow A_{K_n}$ from Subsection \ref{section::semiabelian_compact}. In addition, the homomorphism $\varphi_n$ extends to a map $\overline{\varphi}_n: \overline{G}_n \rightarrow A_{K_n}$. Let $\Mtil_n \in \widehat{\Pic}(\overline{G}_n)$ denote the vertically semipositive metrized line bundle $\Mtil_{\overline{G}_n}$ defined in Subsection \ref{section::semiabelianvarietyheights}. The construction in Subsection \ref{section::semiabelian_compact} shows that $\overline{\varphi}_n^\ast M_{K_n} \approx M_n^{\otimes n}$, which we can use to ensure $\overline{\varphi}_n^\ast \Mtil_{K_n} \approx \Mtil_n^{\otimes n}$ by Lemma \ref{lemma::canonicalmetrics} (a) and (b). Writing $\Ltil_n = \Mtil_n + \pio^\ast_n \Ntil_{K_n}$, we have an isomorphism
\begin{equation}
\label{equation::referee_request}
\Ltil_n \approx \overline{\varphi}_n^\ast ( n^{-1} \Mtil_{K_n} + (\pio^\ast \Ntil)_{K_n}),
\end{equation}
which we invoke frequently. 

In contrast to $\overline{X}_{K_n}$, the preimage $\overline{\varphi}_n^{-1}(\overline{X}_{K_n}) \subseteq \overline{G}_n$ is not irreducible in general. We work therefore with an irreducible component $\overline{Y}$ of $\overline{\varphi}_n^{-1}(\overline{X}_{K_n})$ in the sequel. In addition, we write $\delta(\overline{Y})$ for the degree of $\overline{\varphi}_n|_{\overline{Y}}: \overline{Y} \rightarrow \overline{X}_{K_n}$, which is a surjective finite map. (In the case $X=G$, which corresponds to (EC), we have $\overline{Y} = \overline{G}_n$ and $\delta(\overline{Y})=n^t$.) We write $\Sigma_n(\nu)$ for the places in $\Sigma(K_n)$ that lie above the place $\nu$ in Proposition \ref{proposition1}. For each $\nu^\prime \in \Sigma_n(\nu)$, we fix an identification $\IC_\nu \approx \IC_{\nu^\prime}$; our arguments below are independent of this arbitrary choice and we use it mostly without further explicit mention.

The next lemma controls the growth of geometric degrees as $n \rightarrow \infty$. In its proof and the one of Lemma \ref{lemma::integral} below, we use the standard notation from Fulton's book \cite{Fulton1998} freely.

\begin{lemma} \label{lemma::geometricdegree}
 Let $X \subseteq G$ be a geometrically irreducible algebraic subvariety and $\overline{Y}$ an irreducible component of $\overline{\varphi}_n^{-1}(\overline{X}_{K_n})$. Then, $(L_n|_{\overline{Y}})^d \geq \delta(\overline{Y})n^{-d+d^\prime}$.
\end{lemma}

\begin{proof}[Proof of Lemma \ref{lemma::geometricdegree}] By the projection formula (\cite[Proposition 2.5 (c)] {Fulton1998}), we have
\begin{align*}
c_1(L_n)^d \cap [\overline{Y}]
&= \sum_{i=0}^{d} \binom{d}{i} \delta(\overline{Y})n^{-d+i} c_1(M_{K_n})^{d-i} \cap c_1(\pio^\ast N_{K_n})^i \cap [\overline{X}_{K_n}]
\end{align*}
because $\overline{\varphi}_n^\ast M_{K_n} \approx n M_n$ and $\overline{\varphi}_n^\ast (\pio^\ast N)_{K_n} \approx \pio^\ast_n N_{K_n}$. Since $K_n/K$ is flat, this yields
\begin{equation}
\label{equation::chernexpansion}
c_1(L_n)^d \cap [\overline{Y}] =
\sum_{i=0}^{d} \binom{d}{i} \delta(\overline{Y})n^{-d+i} c_1(M)^{d-i} \cap c_1(\pio^\ast N)^i \cap [\overline{X}] 
\end{equation}
by \cite[Proposition 2.5 (d)]{Fulton1998}.
Note that $M$ is nef by \cite[Lemma 3]{Kuehne2017a} and that $\pio^\ast N$ is nef because $N$ is ample. Using \cite[Theorem III.2.1]{Kleiman1966}, we infer
\begin{equation}
\label{equation::degreelowerbound}
(L_n|_{\overline{Y}})^d  \geq \binom{d}{d^\prime} \delta(\overline{Y})n^{-d+d^\prime} \deg (c_1(M)^{d-d^\prime} \cap c_1(\pio^\ast N)^{d^\prime} \cap [\overline{X}]).
\end{equation}
The lemma is proven if we can show that the degree on the right-hand side of (\ref{equation::degreelowerbound}) is positive. Set $\eta= \eta_{\overline{\pi}(\overline{X})}$. By an ascending induction on the fiber dimension $d-d^\prime$, we can deduce
\begin{equation*}
\deg(c_1(M)^{d-d^\prime} \cap c_1(\overline{\pi}^\ast N)^{d^\prime} \cap [\overline{X}]) = \deg(c_1(M_\eta)^{d-d^\prime} \cap [\overline{X}_{\eta}]) \cdot \deg(c_1(N)^{d^\prime} \cap [\overline{\pi}(\overline{X})])
\end{equation*}
from the projection formula. Since $M_\eta$ is ample on $\overline{X}_\eta$ and $N$ is ample on $\overline{\pi}(\overline{X})$, the two factors on the right-hand side of this identity are strictly positive by \cite[Lemma 12.1]{Fulton1998}.
\end{proof}

The next lemma justifies the choice of the measure $\mu_\nu$ in Proposition \ref{proposition1}.

\begin{lemma} \label{lemma::integral} Assume that $f \in \mathscr{C}^0({\overline{X}}_{\IC_\nu}^{\mathrm{an}})$ is $\Gal(\IC_\nu/K_\nu)$-invariant. 
For each place $\nu^\prime \in \Sigma_n(\nu)$, we set $f_{n,\nu^\prime} = f \circ (\overline{\varphi}_n|_{\overline{Y}})^{\mathrm{an}}_{\IC_{\nu^\prime}} \in \mathscr{C}^0({\overline{Y}}_{\IC_{\nu^\prime}}^{\mathrm{an}})$. Writing $$\mu_{n,\nu^\prime} = c_1(\overline{L}_{n,\nu^\prime}|_{\overline{Y}})^{d}/(L_n|_{\overline{Y}})^{d},$$ we have
\begin{equation*}
\int_{\overline{Y}_{\IC_{\nu^\prime}}^\mathrm{an}} f_{n,\nu^\prime} \mu_{n, {\nu^\prime}} \longrightarrow \int_{\overline{X}_{\IC_{\nu}}^{\mathrm{an}}} f \mu_{\nu} \qquad (n \rightarrow \infty).
\end{equation*}
\end{lemma}

Note that the assumption that $f$ is $\Gal(\IC_\nu/K_\nu)$-invariant is needed to make sure that $f_{n,\nu^\prime}$ is well-defined (i.e., independent of the chosen identification between $\IC_{\nu^\prime}$ and $\IC_{\nu}$).

\begin{proof} Another use of the projection formula reveals that
\begin{align*}
\deg\left(c_1(\pio^\ast N)^i \cap (c_1(M)^{d-i} \cap [\overline{X}])\right)
&=
\deg\left(c_1(N|_{\pio({\overline{X}})})^i \cap (\pio|_{\overline{X}})_\ast (c_1(M|_{\overline{X}})^{d-i} \cap [\overline{X}])\right),
\end{align*}
which is clearly zero whenever $i > d^\prime$. With (\ref{equation::chernexpansion}), we infer that
\begin{equation*}
\left\vert (L_n|_{{\overline{Y}}})^{d} - \delta(\overline{Y})n^{-d+d^\prime} \binom{d}{d^\prime} (M|_{\overline{X}})^{d-d^\prime} (\pio^\ast N|_{\overline{X}})^{d^\prime}\right\vert \ll_{X,M,N} \delta(\overline{Y})n^{-d+d^\prime-1}.
\end{equation*}
Recall from the proof of Lemma \ref{lemma::geometricdegree} above that $ (M|_{\overline{X}})^{d-d^\prime} \cdot (\pio^\ast N|_{\overline{X}})^{d^\prime} > 0$. Invoking Lemma \ref{lemma::chernforms} (a) and (c), we obtain similarly
\begin{multline*}
\left\vert \int_{\overline{Y}_{\IC_{\nu^\prime}}^\mathrm{an}} f_{n,\nu^\prime} c_1(\overline{L}_{n,\nu^\prime}|_{{\overline{Y}}})^{d} - \delta(\overline{Y})n^{-d+d^\prime} \binom{d}{d^\prime}\int_{{\overline{X}}_{\IC_\nu}^\mathrm{an}} fc_1(\Mo_\nu)^{d-d^\prime} \wedge c_1(\pio^\ast \No_\nu)^{d^\prime}\right\vert 
\\ 
\ll_{X,M,N,f} \delta(\overline{Y})n^{-d+d^\prime-1}
\end{multline*}
The lemma follows by combining these two asymptotic estimates.
\end{proof}

\begin{lemma} \label{lemma::lambdanestimate} Assume that $f$ is $\Gal(\IC_\nu/K_\nu)$-invariant and that $\overline{\caO}_{\overline{X}}(f) \in \overline{\Pic}_\nu({\overline{X}})_{\IQ}$ is integrable.
%Let $X \subseteq G$ be a geometrically irreducible algebraic subvariety and $\overline{Y}$ an irreducible component of $\overline{\varphi}_n^{-1}(\overline{X}_{K_n})$. 
For each $\nu^\prime \in \Sigma_n(\nu)$, each positive integer $n$ and each rational number $\lambda \in [-n^{-1},n^{-1}]$, we have
\begin{equation*}
\left| h_{\Ltil_{n,\lambda}}({\overline{Y}})
-h_{\Ltil_{n}}({\overline{Y}}) 
- \frac{\lambda \delta_{\nu}}{[K:\IQ]} \int_{\overline{Y}_{\IC_{\nu^\prime}}} f_{n,\nu^\prime} \mu_{n,\nu^\prime} \right| \ll_{X,f} |\lambda|^2 n
\end{equation*}
where we write $\Ltil_{n,\lambda} := \Ltil_n + \overline{\varphi}^\ast_n (\widetilde{\mathcal{O}}_{\overline{X}}(\lambda f)_{K_n})$.
\end{lemma}

We notice that
\begin{equation}
\label{equation::Ltil_decomposition}
\Ltil_{n,\lambda} = \Ltil_n + \sum_{\substack{\nu^\dprime \in \Sigma_n(\nu) \\ \delta_{\nu^\dprime} = \delta_\nu}} \widetilde{\mathcal{O}}_{\overline{Y}}(\lambda [K_{n,\nu^\dprime}:K_\nu] f_{n,\nu^\dprime}) + \sum_{\substack{\nu^\dprime \in \Sigma_n(\nu) \\ \delta_{\nu^\dprime} = 2\delta_\nu}} \widetilde{\mathcal{O}}_{\overline{Y}}(\lambda f_{n,\nu^\dprime}),
\end{equation}
which is a direct consequence of our definition of base change in Subsection \ref{section::adelicmetrics}. The advantage of working with $\Ltil_{n,\lambda}$ instead of manipulating $\Ltil_n$ at a single place $\nu^\prime \in \Sigma_n(\nu)$ is that $\Ltil_{n,\lambda}$ is evidently a pullback.

\begin{proof} Note first that $\overline{\caO}_{\overline{X}}(\lambda f) \in \overline{\Pic}_\nu({\overline{X}})_{\IQ}$ is integrable for any rational number $\lambda$. We have
\begin{equation*}
h_{\Ltil_{n,\lambda}}({\overline{Y}}) = \frac{(\Ltil_{n,\lambda}|_{\overline{Y}})^{d+1}}{[K_n:\IQ](d+1)(L_n|_{{\overline{Y}}})^{d}}
\end{equation*}
by definition and
\begin{equation*}
(\Ltil_{n,\lambda}|_{\overline{Y}})^{d+1} = \sum_{i=0}^{d+1} \binom{d+1}{i} (\widetilde{L}_n|_{{\overline{Y}}})^{d+1-i}\cdot \overline{\varphi}^\ast_n (\widetilde{\mathcal{O}}_{\overline{X}}(\lambda f)_{K_n})^i \nonumber 
\end{equation*}
by Lemma \ref{lemma::intersectionnumber} (a) and (b). By definition, we have
\begin{equation*}
 \frac{(\Ltil_{n}|_{{\overline{Y}}})^{d+1}}{[K_n:\IQ](d+1)(L_n|_{{\overline{Y}}})^{d}} = h_{\widetilde{L}_n}({\overline{Y}}).
\end{equation*}
If $\nu$ is not a real archimedean place, we note that $\delta_{\nu^\dprime}=\delta_\nu$ for all $\nu^\dprime \in \Sigma_n(\nu)$ so that
\begin{align*}
\frac{(\Ltil_{n}|_{{\overline{Y}}})^{d} \cdot \overline{\varphi}^\ast_n (\widetilde{\mathcal{O}}_{\overline{X}}(\lambda f)_{K_n})}{(L_n|_{{\overline{Y}}})^{d}} 
&=
\frac{\lambda}{[K_n:\IQ]} \sum_{\nu^\dprime \in \Sigma_n(\nu)} \left( \delta_{\nu^\dprime}[K_{n,\nu^\dprime}:K_\nu]\int_{{\overline{Y}}_{\IC_{\nu^\dprime}}^\mathrm{an}} f_{n,\nu^\dprime} \mu_{n,\nu^\dprime} \right) \\
&= \frac{\lambda \delta_{\nu}}{[K:\IQ]} \int_{\overline{Y}_{\IC_{\nu^\prime}}}  f_{n,\nu^\prime} \mu_{n,\nu^\prime}.
\end{align*}
by (\ref{equation::intersection_integral}) and (\ref{equation::Ltil_decomposition}), using also the $\Gal(\IC_\nu/K_\nu)$-invariance of $f$. The same equality is also true if $\nu$ is a real archimedean prime though there is then also a contribution of the second sum in \eqref{equation::Ltil_decomposition}; indeed, one just has to use that $\sum_{\nu^\dprime \in \Sigma_n(\nu)} \delta_{\nu^\dprime} = [K_n:K]$ and $\delta_\nu=1$ in this case. The lemma boils hence down to
\begin{equation} \label{equation::upperbound1}
\left|\sum_{i=2}^{d+1} \sum_{j=0}^{d+1-i} \binom{d+1}{i} \binom{d+1-i}{j} \frac{(\Mtil_n|_{{\overline{Y}}})^{d-i-j+1} \cdot (\overline{\pi}_n^\ast
\Ntil_{K_n}|_{{\overline{Y}}})^{j}\cdot \overline{\varphi}^\ast_n (\widetilde{\mathcal{O}}_{\overline{X}}(\lambda f)_{K_n})^i}{[K_n:\IQ](d+1)(L_n|_{{\overline{Y}}})^{d}}\right| \ll_{X,f} |\lambda|^2 n.
\end{equation}
As $\overline{\varphi}_n^\ast \Mtil_{K_n} = n \Mtil_n$ and $\overline{\varphi}_n^\ast (\overline{\pi}^\ast \Ntil)_{K_n} = \overline{\pi}^\ast_n \Ntil_{K_n}$, Lemma \ref{lemma::intersectionnumber} (c) and (f) imply that
\begin{multline}
\label{equation::someintersectionnumber}
(\Mtil_n|_{{\overline{Y}}})^{d-i-j+1} \cdot (\overline{\pi}_n^\ast
\Ntil_{K_n}|_{{\overline{Y}}})^{j}\cdot \overline{\varphi}^\ast_n (\widetilde{\mathcal{O}}_{\overline{X}}(\lambda f)_{K_n})^i \\
= \delta(\overline{Y})[K_n:K]n^{-d+i+j-1}(\Mtil|_{\overline{X}})^{d-i-j+1} \cdot (\overline{\pi}^\ast \Ntil|_{\overline{X}})^j \cdot \cOt_{\overline{X}}(\lambda f)^i.
\end{multline}
By Lemma \ref{lemma::intersectionnumber} (a) and (\ref{equation::intersection_integral}), we have
\begin{align*} 
(\Mtil|_{{\overline{X}}})^{d+1-i-j} \cdot (\overline{\pi}^\ast \Ntil|_{\overline{X}})^j \cdot \cOt_{\overline{X}}(\lambda f)^i \nonumber
&= \lambda^i (\Mtil|_{{\overline{X}}})^{d+1-i-j} \cdot (\overline{\pi}^\ast \Ntil|_{\overline{X}})^j \cdot \cOt_{\overline{X}}(f)^i \\
&= \lambda^i \delta_\nu \int_{{\overline{X}}_{\IC_\nu}^\mathrm{an}} fc_1(\Mo_\nu)^{d-i-j+1} \wedge c_1(\overline{\pi}^\ast \No_\nu)^j \wedge c_1(\overline{\caO}_{\overline{X}}(f))^{i-1}
\end{align*}
for any $i \geq 2$. By the integrability assumption, there exist semipositive $\overline{P}_1,\overline{P}_2 \in \overline{\Pic}_\nu({\overline{X}})_\IQ$ such that $\overline{\caO}_{\overline{X}}(f) = \overline{P}_1 - \overline{P}_2$. The above integral is then bounded from above by
\begin{equation*}
\Vert f \Vert_{\sup} \sum_{k=0}^{i-1} \binom{i-1}{k} (M|_{\overline{X}})^{d-i-j+1} \cdot (\pio^\ast N|_{\overline{X}})^{j} \cdot (P_1|_{\overline{X}})^{i-k-1} \cdot (P_2|_{\overline{X}})^{k} 
\end{equation*}
because of Lemma \ref{lemma::chernforms} (d). By the projection formula, this is zero if $j>d^\prime$. Consequently, the term in (\ref{equation::someintersectionnumber}) is $\ll_{X,f} \delta(\overline{Y})[K_n:K]|\lambda|^i n^{-d+d^\prime+i-1}$. In combination with Lemma \ref{lemma::geometricdegree}, this gives that each term on the left-hand side of (\ref{equation::upperbound1}) is $\ll_{X,f} |\lambda|^i n^{i-1}$. Since $|\lambda|^i n^{i-1} \leq |\lambda|^2 n$, we obtain (\ref{equation::upperbound1}).

\end{proof}

Starting with the following lemma, we fix some place $\nu_0 \in \Sigma_\infty(K)$ so that we can regard a real number $\kappa_n$ as a constant function on ${\overline{X}}_{\IC_{\nu_0}}^{\mathrm{an}}$ and define $\widetilde{\caO}_{\overline{X}}(\kappa_n)$ as in Subsection \ref{section::adelicmetrics}.\footnote{Note that $\widetilde{\caO}_{\overline{X}}(\kappa_n)$ is clearly integrable if $\nu_0 \in \Sigma_\infty(K)$, which is not immediately clear if we would allow some non-archimedean place $\nu_0 \in \Sigma_f(K)$ instead. However, this could be archieved as in the beginning of the proof of Proposition 12 below.} For our purposes, it is immaterial which place $\nu_0$ we choose, even whether $\nu = \nu_0$ or $\nu \neq \nu_0$, so that we omit any further reference to the place $\nu_0$ in the following.

\begin{lemma} 
\label{lemma::heightbounds}
For each integer $n \geq 1$, there exists some positive real $\kappa_n \ll_{G} n^{-2}$ such that $\Ltil_n \otimes \overline{\varphi}_n^\ast (\widetilde{\mathcal{O}}_{\overline{X}}(\kappa_n)_{K_n})$ is horizontally semipositive. If $X \subset G$ contains an $X$-generic sequence of small points, we have
\begin{equation}
\label{equation::heightbounds}
-n^{-2} \ll_{G} h_{\Ltil_n}({\overline{Y}}) \leq 0
\end{equation}
for every irreducible component $\overline{Y}$ of $\overline{\varphi}_n^{-1}(\overline{X}_{K_n})$. In this situation, we also have
\begin{equation}
\label{equation::heightprojection}
h_{\Ntil_{K_n}}(\pio_n({\overline{Y}}))=0.
\end{equation}
\end{lemma}

For the compactification of $G$ and the associated Arakelov heights used by Chambert-Loir, an explicit height formula \cite[Th\'eor\`eme 4.2]{Chambert-Loir2000} implies the same asymptotics in the case $X=G$. By Zhang's proof of $(\mathrm{BC})$, (\ref{equation::heightprojection}) implies that $\pio_n(\overline{Y})$ is a translate of an abelian subvariety of $A_{K_n }$ by a torsion point. However, we have no use for this information, and it does not simplify the arguments below.

\begin{proof} 
By definition, $\Ltil_n \otimes \overline{\varphi}_n^\ast (\widetilde{\mathcal{O}}_{\overline{X}}(\kappa_n)_{K_n})$ is horizontally semipositive if and only if
\begin{equation*}
\inf_{\substack{x \in \overline{G}_n \\ \text{closed}}} \{h_{\Ltil_n \otimes \overline{\varphi}_n^\ast (\widetilde{\mathcal{O}}_{\overline{X}}(\kappa_n)_{K_n})}(x) \} = \inf_{\substack{x \in \overline{G}_n \\ \text{closed}}} \{ h_{\Ltil_n}(x) \} + \frac{\kappa_n}{[K:\IQ]} \geq 0
\end{equation*}
Our first assertion hence follows directly from the statement of \cite[Lemme 4.5]{Chambert-Loir2000} if $t=1$. In general, the compactification used there differs from ours and the argument has to be slightly adjusted. But this is straightforward and hence left to the reader. It should also be noted that Chambert-Loir \cite{Chambert-Loir2000} uses a different definition of height (see Subsection 1.4 in his article) than the one we introduced in Section \ref{section::heights}; ours has an additional factor $[K:\IQ]$ to make it absolute, which allows us to work with varying base fields $K_n$.

By Zhang's ampleness theory in the incarnation of \cite[Théorème 1.5]{Chambert-Loir2000}, we have
\begin{equation*}
h_{\Ltil_n}(Z) \geq \inf_{\substack{x \in Z \\ \text{closed}}} \{ h_{\Ltil_n}(x) \} \geq \inf_{\substack{x \in \overline{G}_n \\ \text{closed}}} \{h_{\Ltil_n  \otimes \overline{\varphi}_n^\ast (\widetilde{\mathcal{O}}_{\overline{X}}(\kappa_n)_{K_n})}(x) \} - \frac{\kappa_n}{[K:\IQ]} \geq - \frac{\kappa_n}{[K:\IQ]}
\end{equation*}
for any irreducible subvariety $Z \subseteq \overline{G}_n$, which proves the lower bound in (\ref{equation::heightbounds}).
	
Let now $(x_i) \in X^{\IN}$ be an $X$-generic sequence of small points. As $h_{\Ltil_n}$ and $h_{\Mtil}$ are non-negative on the closed points of $G$ (see the remark after Lemma \ref{lemma::heightcomparison}), we obtain
\begin{equation}
\label{equation::height_of_preimages}
0 \leq \max_{x \in \overline{\varphi}_n^{-1}(x_i) \cap \overline{Y}} \{ h_{\Ltil_n}(x)\} = n^{-1} h_{\Mtil}(x_i) + h_{\pio^\ast \Ntil}(x_i) \leq h_{\Ltil}(x_i) \longrightarrow 0 \ \ (i \rightarrow \infty)
\end{equation}
by using \eqref{equation::referee_request}.
This shows that $\overline{Y}$ contains a $\overline{Y}$-generic sequence of small points. In this way, the upper bound in (\ref{equation::heightbounds}) is another consequence of Zhang's inequalities.

Finally, $(\pio(x_i))$ is a $\pio_n(\overline{X})$-generic sequence of small points so that, taking into account the horizontal semipositivity of $\Ntil$, we obtain $h_{\Ntil_{K_n}}(\pio_n({\overline{Y}}))=h_{\Ntil}(\pio(\overline{X}))=0$.
\end{proof}

We next give an asymptotic lower bound on the arithmetic volume of $\Ltil_{n,\lambda} |_{{\overline{Y}}}$. Recall that $\Ltil_{n,\lambda} := \Ltil_n + \overline{\varphi}^\ast_n \widetilde{\mathcal{O}}_{\overline{X}}(\lambda f)$ and set analogously $\Ltil_{n,\kappa_n}=\Ltil_n + \overline{\varphi}_n^\ast (\widetilde{\mathcal{O}}_{\overline{X}}(\kappa_n)_{K_n})$.

\begin{lemma} \label{lemma::chisup} Assume that $\widetilde{\caO}_{\overline{X}}(f) \in \widehat{\Pic}({\overline{X}})_{\IQ}$ is integrable and that $\Ltil_n(\kappa_n)$, $\kappa_n>0$, is horizontally semipositive. For any positive integer $n$ and any rational number $\lambda \in [-n^{-1},n^{-1}]$, we have
\begin{equation*}
\volh_\chi(\Ltil_{n,\lambda}|_{{\overline{Y}}}) - (\Ltil_{n,\lambda}|_{{\overline{Y}}})^{d+1} \gg_{X,f} - \delta(\overline{Y})[K_n:K]|\lambda|^2n^{-d+d^\prime+1}(1+\kappa_n n).
\end{equation*}
\end{lemma}

Note that the two terms on the left-hand side are equal if $\Ltil_{n,\lambda} |_{{\overline{Y}}}$ is vertically semipositive (see \cite[Theorem 3.5.1 and Remark 3.5.4]{Ikoma2013}). In our situation, there are however two obstructions to this line of reasoning. First, $\Ltil_{n,\lambda}|_{{\overline{Y}}}$ may not be vertically semipositive for small $\lambda$ even if $\Ltil_n$ is so. This problem has to be dealt with already in the almost split case (see \cite{Chambert-Loir2000}). The second problem is that $\Ltil_n$ is not horizontally semipositive, which is a new problem for general semiabelian varieties. To work around this, we follow the argument given in \cite[Subsection 3.2]{Yuan2008} but have to pay additional attention to the errors terms suppressed therein.

\begin{proof} 
By assumption, there exist semipositive $\Ptil_1, \Ptil_2 \in \widehat{\Pic}({\overline{X}})_\IQ$ such that $\cOt_{\overline{X}}(f) = \Ptil_1 - \Ptil_2$. Applying Lemma \ref{lemma::siu} to the decomposition
\begin{equation*}
\Ltil_{n,\lambda,\kappa_n}|_{{\overline{Y}}} := \left( \Ltil_{\kappa_n} + \lambda \overline{\varphi}^\ast_n \Ptil_{1,K_n}|_{{\overline{Y}}} \right) - \lambda \overline{\varphi}^\ast_n \Ptil_{2,K_n}|_{{\overline{Y}}},
\end{equation*}
we obtain that 
\begin{multline}
\label{equation::volumelowerbound1}
\volh_\chi(\Ltil_{n,\lambda,\kappa_n}|_{{\overline{Y}}}) 
\\
\geq ( \Ltil_{n,\kappa_n} |_{{\overline{Y}}} + \lambda \overline{\varphi}^\ast_n \Ptil_{1,K_n}|_{{\overline{Y}}})^{d+1} - (d+1) (\Ltil_{n,\kappa_n} |_{{\overline{Y}}} + \lambda \overline{\varphi}^\ast_n \Ptil_{1,K_n}|_{{\overline{Y}}})^d \cdot \lambda \overline{\varphi}^\ast_n \Ptil_{2,K_n}|_{{\overline{Y}}}.
\end{multline}
By Lemma \ref{lemma::intersectionnumber} (a), subtracting $(\Ltil_{n,\lambda,\kappa_n}|_{{\overline{Y}}})^{d+1}$ from the right-hand side of this inequality results in
\begin{equation}
-\sum_{i=2}^{d+1} \binom{d+1}{i} ( \Ltil_{n,\kappa_n} |_{{\overline{Y}}}+\lambda \overline{\varphi}_n^\ast \Ptil_{1,K_n}|_{{\overline{Y}}})^{d+1-i} \cdot (-\lambda \overline{\varphi}_n^\ast \Ptil_{2,K_n}|_{{\overline{Y}}})^i.
\end{equation}
We claim that the absolute value of this difference is $$\ll_{X,f} \delta(\overline{Y})[K_n:K]|\lambda|^2n^{-d+d^\prime+1}(1+\kappa n).$$ For this purpose, we expand the intersection number
\begin{equation*}
( \Ltil_{n,\kappa_n} |_{{\overline{Y}}}+\lambda \overline{\varphi}_n^\ast \Ptil_{1,K_n}|_{{\overline{Y}}})^{d+1-i} \cdot (-\lambda \overline{\varphi}_n^\ast \Ptil_{2,K_n}|_{{\overline{Y}}})^i
\end{equation*}
as
\begin{equation*}
(-1)^i \sum_{j=0}^{d+1-i} \lambda^{i+j}\binom{d+1-i}{j} (\Ltil_{n,\kappa_n}|_{{\overline{Y}}})^{d+1-i-j} \cdot (\overline{\varphi}_n^\ast \Ptil_{1,K_n}|_{{\overline{Y}}})^j
\cdot (\overline{\varphi}_n^\ast \Ptil_{2,K_n}|_{{\overline{Y}}})^i.
\end{equation*}
Expanding this sum even further, we obtain
\begin{equation}
\label{equation::difference}
\sum_{j=0}^{d+1-i} \sum_{k_1+k_2+k_3 = d+1-i-j} \lambda^{i+j} \alpha(k_1,k_2,k_3,i,j)  c(k_1,k_2,k_3,i,j)
\end{equation}
with
\begin{equation*}
\alpha(k_1,k_2,k_3,i,j) =  [K_n:K] \delta(\overline{Y}) n^{-k_1} (\Mtil|_{{\overline{X}}})^{k_1} \cdot (\overline{\pi}^\ast \Ntil|_{{\overline{X}}})^{k_2} \cdot \Otil_{{\overline{X}}}(\kappa_n)^{k_3} \cdot (\Ptil_1|_{{\overline{X}}})^j
\cdot (\Ptil_2|_{{\overline{X}}})^i
\end{equation*}
and
\begin{equation*}
c(k_1,k_2,k_3,i,j) = (-1)^i \binom{d+1-i}{j} \binom{d+1-i-j}{k_1,k_2,k_3}.
\end{equation*}
From \eqref{equation::heightprojection}, we know that $\Ntil^{d^\prime+1} \cdot [\pio_n({\overline{Y}})]=0$. Applying Lemma \ref{lemma::heightfibrations} to $\pio_n$ and ${\overline{X}} \subseteq \overline{G}$, it follows that $\alpha(k_1,k_2,k_3,i,j)$ is zero if $k_2>d^\prime$.  In addition, $\alpha(k_1,k_2,k_3,i,j)$ is zero if $k_3>1$ because of (\ref{equation::intersection_integral}) and Lemma \ref{lemma::chernforms} (d). If $k_3=0$ and $\alpha(k_1,k_2,k_3,i,j) \neq 0$, we have $k_1 \geq d+1-i-j-d^\prime$ and hence $|\alpha(k_1,k_2,k_3,i,j)| \ll_{X,f} \delta(\overline{Y})[K_n:K] n^{-d+d^\prime+i+j-1}$. If $k_3 = 1$, we have
\begin{equation*}
\alpha(k_1,k_2,k_3,i,j) = \kappa_n \delta(\overline{Y}) n^{-k_1} \delta_{\nu_0} \int_{{\overline{X}}^{\mathrm{an}}_{\IC_\nu}}
c_1(\Mo_{\nu_0}|_{\overline{X}})^{k_1} \wedge
c_1(\pio^\ast \No_{\nu_0}|_{\overline{X}})^{k_2} \wedge (\overline{P}_{1,{\nu_0}}|_{\overline{X}})^j \wedge (\overline{P}_{2,{\nu_0}}|_{\overline{X}})^i
\end{equation*}
by (\ref{equation::intersection_integral}). In combination with $k_1 \geq d-i-j-d^{\prime}$, this implies $$|\alpha(k_1,k_2,k_3,i,j)| \ll_{X,f} \kappa_n \delta(\overline{Y}) n^{-d+d^{\prime}+i+j}.$$ These estimates imply that the absolute value of each summand in (\ref{equation::difference}) is 
\begin{equation*}
\ll_{X,f} \delta(\overline{Y})|\lambda|^{i+j}n^{-d+d^\prime+i+j-1}(1+\kappa_n n) \leq \delta(\overline{Y}) |\lambda|^2 n^{-d+d^{\prime}+1}(1+\kappa_n n).
\end{equation*}
In summary, we conclude that
\begin{equation*}
\volh_\chi(\Ltil_{n,\lambda,\kappa_n}|_{{\overline{Y}}})  - (\Ltil_{n,\lambda,\kappa_n}|_{{\overline{Y}}})^{d+1} \gg_{X,f} - \delta(\overline{Y}) |\lambda|^2 n^{-d+d^{\prime}+1}(1+\kappa_n n).
\end{equation*}
By (\ref{equation::intersection_Lc}) and Lemma \ref{lemma::arithmeticvolumes_basics} (a), the left-hand side of the above inequality equals
\begin{equation*}
\volh_\chi(\Ltil_{n,\lambda}|_{{\overline{Y}}})  - (\Ltil_{n,\lambda}|_{{\overline{Y}}})^{d+1}.
\end{equation*}
\end{proof}

We can now reap the proceeds of the above lemmas.

\begin{proof}[Proof of Proposition \ref{proposition1}] 
	
It is well-known that $\cOt_{\overline{X}}(f)$ is integrable if $\nu \in \Sigma_\infty(K)$. Enlarging $K$, we can also assume that $K_\nu=\IC_\nu$ and hence that $f$ is trivially $\Gal(\IC_\nu/K_\nu)$-invariant. If $\nu \in \Sigma_f(K)$, we can use a Stone-Weierstrass approximation argument of Yuan \cite[p.\ 638, ``Equivalence'']{Yuan2008} to reduce the proof to the case where $f$ is $\Gal(\IC_\nu/K_\nu)$-invariant. A similar approximation argument \cite[Theorem 7.12]{Gubler1998} and \cite[Proposition 3.4]{Gubler2008} (see also \cite[Lemma 3.5]{Yuan2008} and \cite[Subsection 10.4]{Yuan2012}) allows us to assume that $\cOt_{\overline{X}}(f)$ is integrable. Both approximations potentially involve a replacement of $K$ by a finite extension.

From now on, let $n$ be a fixed integer. We choose also a non-zero rational $\lambda \in [-n^{-1},n^{-1}]$ and a real $\varepsilon>0$; an explicit choice of $\lambda$ is given below. By Lemma \ref{lemma::heightbounds}, we can find some positive real $\kappa_n \ll_{G} n^{-2}$ such that $\Ltil_{n,\kappa_n}=\Ltil_n + \overline{\varphi}_n^\ast (\widetilde{\mathcal{O}}_{\overline{X}}(\kappa_n)_{K_n})$ is horizontally semipositive.

Let $\nu \in \Sigma(K)$ and a real $\varepsilon>0$ be given. For the sequel, we keep fixed an arbitrary place $\nu^\prime \in \Sigma_n(\nu)$ above $\nu \in \Sigma(K)$. In addition, let us again write $\Ltil_{n,\lambda} = \Ltil_n + \overline{\varphi}^\ast_n \widetilde{\mathcal{O}}_{\overline{X}}(\lambda f) = (L_n, \{ \Vert \cdot \Vert_{\mu} \}_{\mu \in \Sigma(K_n)} )$. By Lemma \ref{lemma::minkowski}, there exists some positive integer $N_0$ and a non-zero section $\mathbf{s} \in H^0(\overline{Y}, (L_n|_{\overline{Y}})^{\otimes N_0})$ such that
\begin{equation*}
\frac{\delta_{\nu^\prime} \log \Vert \mathbf{s}(x) \Vert_{\nu^\prime}^{1/N_0}}{[K_n:\IQ]} \leq  - \frac{\volh_\chi(\Ltil_{n,\lambda}|_{{\overline{Y}}})}{[K_n:\IQ](d+1)(L_n|_{{\overline{Y}}})^d} + \varepsilon
\end{equation*}
for every point $x \in {\overline{Y}}_{\IC_{\nu^\prime}}^{\mathrm{an}}$ and $\log \Vert \mathbf{s}(x) \Vert_{\mu} \leq 0$ ($\mu \in \Sigma(K_n) \setminus \{ \nu^\prime \}$) for every point $x \in {\overline{Y}}_{\IC_\mu}^{\mathrm{an}}$. 
Using Lemmas \ref{lemma::geometricdegree} and \ref{lemma::chisup}, we infer that there exists some constant $c_1=c_1(X,f,[K:\IQ])>0$ such that
\begin{equation*}
\frac{\delta_{\nu^\prime} \log \Vert \mathbf{s}(x) \Vert_{\nu^\prime}^{1/N_0}}{[K_n:\IQ]} \leq  - h_{\Ltil_{n,\lambda}}({\overline{Y}}) + c_1|\lambda|^2n(1+\kappa_n n) + \varepsilon
\end{equation*}
for every point $x \in {\overline{Y}}_{\IC_{\nu^\prime}}^{\mathrm{an}}$. With Lemma \ref{lemma::lambdanestimate} and (\ref{equation::heightbounds}), we further deduce that
\begin{equation*}
\frac{\delta_{\nu^\prime} \log \Vert \mathbf{s}(x) \Vert_{\nu^\prime}^{1/N_0}}{[K_n:\IQ]} \leq
- \frac{\lambda \delta_\nu}{[K:\IQ]} \int_{\overline{Y}_{\IC_{\nu^\prime}}} f_{n,\nu^\prime} \mu_{n,\nu^\prime}
+ c_2 (n^{-2} + |\lambda|^2 n + \kappa_n |\lambda|^2n^2) + \varepsilon
\end{equation*}
for some constant $c_2=c_2(X,f,[K:\IQ])>0$ and all $x \in ({\overline{Y}}\setminus \left\vert\Div(\mathbf{s})\right\vert)_{\IC_{\nu^\prime}}^{\mathrm{an}}$. Through (\ref{equation::height_section}), we can use this to obtain the lower global bound
\begin{equation}
\label{equation::expansion_ref}
h_{\Ltil_{n,\lambda}}(x) \geq \frac{\lambda \delta_\nu}{[K:\IQ]} \int_{{\overline{Y}}_{\IC_{\nu^\prime}}^\mathrm{an}} f_{n,\nu^\prime} \mu_{n,\nu^\prime} - c_2 (n^{-2} + |\lambda|^2 n + \kappa_n |\lambda|^2n^2) - \varepsilon
\end{equation}
for all closed points $x \in \overline{Y} \setminus \left\vert\Div(\mathbf{s})\right\vert$.  Note that $\nu$ is not a real archimedean place by our above assumptions so that the second sum in \eqref{equation::Ltil_decomposition} is empty. For each closed point $x \in \overline{Y}$, we infer from \eqref{equation::height_section} and \eqref{equation::Ltil_decomposition} that
\begin{align*}
h_{\overline{\varphi}^\ast_n(\widetilde{\mathcal{O}}_{\overline{X}}(\lambda f))_{K_n}}(x) 
&=
\frac{\lambda}{[K_n(x):\IQ]} \sum_{\nu^\dprime \in \Sigma_n(\nu)} \sum_{y \in \mathbf{O}_\nu(x)} \delta_{\nu^\dprime} [K_{\nu^\dprime}:K_\nu] f_{n,\nu^\dprime}(y) \\
&=
\frac{\lambda\delta_\nu[K_n:K]}{[K_n(x):\IQ]} \sum_{y \in \mathbf{O}_{\nu^\prime}(x)} f_{n,\nu^\prime}(y) \\
&=
\frac{\lambda\delta_{\nu}}{[K:\IQ] \cdot (\# \mathbf{O}_{\nu^\prime}(x))} \sum_{y \in \mathbf{O}_{\nu^\prime}(x)} f_{n,\nu^\prime}(y),
\end{align*}
where we also used the $\Gal(\IC_\nu/K_\nu)$-invariance of $f$ in the second equality. Expanding the left-hand side of \eqref{equation::expansion_ref}, we thus obtain
\begin{multline}
\label{equation::mainestimate}
h_{\Ltil_n}(x) + \frac{\lambda \delta_{\nu}}{[K:\IQ]}\left( \frac{1}{\# \mathbf{O}_{\nu^\prime}(x) }\sum_{y \in \mathbf{O}_{\nu^\prime}(x)} f_{n,\nu^\prime}(y) - \int_{{\overline{Y}}_{\IC_{\nu^\prime}}^\mathrm{an}} f_{n,\nu^\prime} \mu_{n,\nu^\prime} \right)
\\
\geq - c_2 (n^{-2} + |\lambda|^2 n + \kappa_n |\lambda|^2n^2) - \varepsilon
\end{multline}
for all closed points $x \in {\overline{Y}} \setminus \left\vert\Div(\mathbf{s})\right\vert$. Since $(x_i) \in X^{\IN}$ is a generic sequence, there exists some integer $i_0=i_0(X,f,n)$ such that $x \notin \left\vert\Div(\mathbf{s})\right\vert$ for all $x \in \overline{\varphi}_n^{-1}(x_i) \cap \overline{Y}$, $i \geq i_0$. We recall that
\begin{equation*}
0 \leq \max_{x \in \overline{\varphi}_n^{-1}(x_i) \cap \overline{Y}} \{ h_{\Ltil_n}(x)\} \longrightarrow 0 \ \ (i \rightarrow \infty)
\end{equation*}
by \eqref{equation::height_of_preimages}. Combining this with \eqref{equation::mainestimate} and averaging over all closed points $x \in \overline{\varphi}_n^{-1}(x_i)$, it follows that
\begin{multline}
\label{equation::lambdacondition}
\liminf_{i \rightarrow \infty} \left( \frac{\lambda}{\# \mathbf{O}_\nu(x_i)}\sum_{y \in \mathbf{O}_\nu(x_i)} f(y) -  \lambda \int_{{\overline{Y}}_{\IC_\nu}^\mathrm{an}} f_{n,\nu^\prime} \mu_{n,\nu^\prime} \right)
\\
\geq
-  c_3(n^{-2} + |\lambda|^2 n + \kappa_n |\lambda|^2n^2) - \delta_\nu^{-1}[K:\IQ]\varepsilon
\end{multline}
for some constant $c_3 = c_3(X,f,[K:\IQ])>0$. Working with $-\lambda$ instead of $\lambda$ in our above reasoning, we obtain similarly
\begin{multline}
\label{equation::lambdacondition}
\limsup_{i \rightarrow \infty} \left( \frac{\lambda}{\# \mathbf{O}_\nu(x_i)}\sum_{y \in \mathbf{O}_\nu(x_i)} f(y) -  \lambda \int_{{\overline{Y}}_{\IC_\nu}^\mathrm{an}} f_{n,\nu^\prime} \mu_{n,\nu^\prime} \right)
\\
\leq
c_3(n^{-2} + |\lambda|^2 n + \kappa_n |\lambda|^2n^2) + \delta_\nu^{-1}[K:\IQ]\varepsilon.
\end{multline}
Combining these two inequalities, we can deduce that
\begin{equation*}
\limsup_{i \rightarrow \infty} \left| \frac{1}{\# \mathbf{O}_\nu(x_i)}\sum_{y \in \mathbf{O}_\nu(x_i)} f(y) - \int_{{\overline{Y}}_{\IC_\nu}^\mathrm{an}} f_{n,\nu^\prime} \mu_{n,\nu^\prime} \right|
\end{equation*}
is bounded from above by
\begin{equation}
\label{equation::equidistribution3}
c_3(n^{-2} |\lambda|^{-1} + |\lambda| n + \kappa_n |\lambda| n^2)+ \delta_\nu^{-1}[K:\IQ] |\lambda|^{-1}\varepsilon.
\end{equation}
Assume now that $n$ is a square integer. We further choose $\lambda = n^{-3/2}$ and $\varepsilon= n^{-1/2}|\lambda|$. In this way, we obtain
\begin{equation*}
\lim_{i \rightarrow \infty} \left| \frac{1}{\# \mathbf{O}_\nu(x_i)}\sum_{y \in \mathbf{O}_\nu(x_i)} f(y) - \int_{{\overline{Y}}_{\IC_\nu}^\mathrm{an}} f_{n,\nu^\prime} \mu_{n,\nu^\prime} \right| \ll_{X,f} n^{-1/2}.
\end{equation*}
Combining this with Lemma \ref{lemma::integral}, the proposition follows with $n \rightarrow \infty$.
\end{proof}

\section{Equilibrium Measures}
\label{section::equilibrium}

In preparation for the proof of ($\mathrm{BC}$) in Section \ref{section::bogomolov}, we investigate here the measures from Proposition \ref{proposition1} in more detail. We continue with the notation of Section \ref{section::toricrank1} but restrict to an archimedean place $\nu \in \Sigma_\infty(K)$ throughout this section. Choose a local trivialization $\{ (U_j, \psi_j) \}_{j \in J}$, $J$ finite, of $\overline{G}^{\mathrm{an}}_{\IC_\nu}$ as in Lemma \ref{lemma::standardcoordinates} and write $\psi_j= (\psi_1^{(j)},\dots,\psi_t^{(j)})$. Furthermore, we use the line bundles $M_{\overline{G},i}^{(0)}$ and $M_{\overline{G},i}^{(\infty)}$ ($1 \leq i \leq t$) as defined in Subsection \ref{section::semiabelian_compact} and set $M_{\overline{G},i}=M_{\overline{G},i}^{(0)} \otimes M_{\overline{G},i}^{(\infty)}$.

Let $X \subseteq G$ be a \textit{geometrically irreducible} algebraic subvariety of positive dimension and denote by $\overline{X}$ its Zariski closure in $\overline{G}$. Set $d=\dim(X)$, $d^\prime = \dim(\pio(X))$, and $t^\prime = d - d^\prime$. We let $I_X$ be the set of all $t^\prime$-tuples $(i_1, i_2, \dots, i_{t^\prime})$ such that 
\begin{equation}
\label{equation::fiberwiseintersection}
M_{\overline{G},i_1}|_{\eta_{\pio(X)}} \cdot M_{\overline{G},i_2}|_{\eta_{\pio(X)}} \cdots M_{\overline{G},i_{t^\prime}}|_{\eta_{\pio(X)}} > 0
\end{equation}
where $\eta_{\pio(X)}$ is the generic point of $\pio(X)$. As $M_{\overline{G}} = \otimes_{i=1}^t M_{\overline{G},i}$ is relatively ample with respect to $\overline{\pi}: \overline{G} \rightarrow A$, the set $I_X$ is non-empty. For each $\underline{i}=(i_1, i_2, \dots, i_{t^\prime}) \in I_X$, we define the subset 
\begin{equation*}
X_{\underline{i}} = \bigcup_{j \in J} {\{ y \in (\overline{\pi}_{\IC_\nu}^{\mathrm{an}})^{-1}(U_j) \cap \overline{X}_{\IC_\nu}^{\mathrm{an}} \ \vline \ | \psi^{(j)}_{i_1}(y)| = | \psi^{(j)}_{i_2}(y)| = \cdots = | \psi^{(j)}_{i_{t^\prime}}(y)| = 1 \}}
\end{equation*}
and the map
\begin{equation*}
\psi_{\underline{i}}^{(j)} = (\psi_{i_1}^{(j)},\dots,\psi_{i_{t^\prime}}^{(j)}) :\ \overline{G}^{\mathrm{an}}_{\IC_\nu}|_{U_j} \longrightarrow ((\IP^1_{\IC_\nu})^{t^\prime})^{\mathrm{an}}.
\end{equation*}
Using (\ref{equation::absolutevalue}), we see that $X_{\underline{i}}$ is a closed and hence compact \textit{real-analytic} subset of $\overline{X}_{\IC_\nu}^{\mathrm{an}}$. We next define \textit{complex-analytic} subsets $E_{\underline{i}} \subset {\overline{X}}_{\IC_\nu}^{\mathrm{an}}$, $\underline{i} \in I_X$, such that each $X_{\underline{i}}$ has a simple structure away from $E_{\underline{i}}$. For this, we first set
\begin{equation*}
E_{\underline{i}}^{(j)} = \left\{ y \in (\overline{\pi}_{\IC_\nu}^{\mathrm{an}})^{-1}(U_j) \cap \overline{X}_{\IC_\nu}^{\mathrm{an}} \ \vline \ \ker(d\psi_{\underline{i}}^{(j)}) \cap \ker(d\pio)^{\mathrm{an}}_{\IC_\nu} \cap T_{\IC,y}^{1,0} (X_{\IC_\nu}^{\mathrm{an}}) \neq \{ 0_y \} \right\}, \ j \in J,
\end{equation*}
and note that by (\ref{equation::tangentrank}) we have $E_{\underline{i}}^{(j)} \cap (\overline{\pi}_{\IC_\nu}^{\mathrm{an}})^{-1}(U_{j^\prime}) = E_{\underline{i}}^{(j^\prime)} \cap (\overline{\pi}_{\IC_\nu}^{\mathrm{an}})^{-1}(U_{j})$ for all $j,j^\prime \in J$. Consequently, their union $\bigcup_{j\in J}E_{\underline{i}}^{(j)}$ is a closed complex-analytic subset of $\overline{X}_{\IC_\nu}^{\mathrm{an}}$. We set
\begin{equation*}
E_{\underline{i}} 
= \bigcup_{j\in J}E_{\underline{i}}^{(j)}
\cup (\overline{X} \setminus \overline{X}^{\mathrm{sm}})^{\mathrm{an}}_{\IC_\nu} 
\cup (\pio^{-1}(\pio(\overline{X})\setminus \pio(\overline{X})^{\mathrm{sm}}))^{\mathrm{an}}_{\IC_\nu} \cup (\overline{X} \cap (\overline{G} \setminus G))^{\mathrm{an}}_{\IC_\nu}.
\end{equation*}
We collect the main properties of $X_{\underline{i}}$ and $E_{\underline{i}}$ in the following lemma.

\begin{lemma}\label{lemma::structuralinformation}
Let $j \in J$ and $\underline{i} \in I_X$. Then,
\begin{enumerate}
\item[(a)] $E_{\underline{i}}$ is a closed complex-analytic subset of $\overline{X}^{\mathrm{an}}_{\IC_\nu}$ having dimension $<d$,
\item[(b)] on $(\overline{X}^{\mathrm{an}}_{\IC_\nu} \cap (\overline{\pi}_{\IC_\nu}^{\mathrm{an}})^{-1}(U_j)) \setminus E_{\underline{i}}$, the map $\pio^{\mathrm{an}}_{\IC_\nu} \times \psi_{\underline{i}}^{(j)}$ restricts to a local biholomorphism with codomain $(\pio(\overline{X})^{\mathrm{sm}} \times \Gm^{t^\prime})^{\mathrm{an}}_{\IC_\nu}$,
\item[(c)] $X_{\underline{i}} \setminus E_{\underline{i}}$ is a union of finitely many (embedded) real-analytic submanifolds, each having dimension $d+d^\prime$, and 
\item[(d)] on $\left(X_{\underline{i}} \cap (\overline{\pi}_{\IC_\nu}^{\mathrm{an}})^{-1}(U_j)\right) \setminus E_{\underline{i}}$, the map  $\pio^{\mathrm{an}}_{\IC_\nu} \times \psi_{\underline{i}}^{(j)}$ restricts to a real-analytic local isomorphism with codomain $(\pio(\overline{X})^{\mathrm{sm}})^{\mathrm{an}}_{\IC_\nu} \times (S^1)^{t^\prime}$.
\end{enumerate}
\end{lemma}

\begin{proof} (a): It is enough to show that the closed complex-analytic subset $\bigcup_{j\in J}E_{\underline{i}}^{(j)}$ has dimension $<d$. Since $\overline{X}^{\mathrm{an}}_{\IC_\nu}$ is irreducible as a complex-analytic set, we only have to find a point $y \in \overline{X}^{\mathrm{an}}_{\IC_\nu}$ not contained in $\bigcup_{j\in J}E_{\underline{i}}^{(j)}$.

By assumption (\ref{equation::fiberwiseintersection}), there exists a closed point $z \in \pio(\overline{X})^{\mathrm{sm}}_{\IC_\nu}$ such that the fiber $\overline{X}|_z$ is of dimension $t^\prime$ and
\begin{equation}
\label{equation::yfiberintersection}
M_{\overline{G},i_1} \cdot M_{\overline{G},i_2} \cdots M_{\overline{G},i_{t^\prime}}\cdot [\overline{X}|_z]>0.
\end{equation}
Let $U_j$, $j \in J$, be such that $z^{\mathrm{an}} \in U_j$. Note that $\psi_{\underline{i}}^{(j)}|_{(\overline{X}|_{z})^{\mathrm{an}}}$ is the analytification of an algebraic map $f: \overline{X}|_z \rightarrow (\IP^1)^{t^\prime}$ (either by Chow's theorem \cite[Theorem M.3]{Gunning1990a} or by inspecting the proof of Lemma \ref{lemma::standardcoordinates}) such that $M_{\overline{G},i}|_{z} \approx f^\ast \mathrm{pr}_i^\ast M_{\IP^1}$ where $\mathrm{pr}_i: (\IP^1)^{t^\prime} \rightarrow \IP^1$ is the projection to the $i$-th factor. It is easy to see that $\dim(f(\overline{X}|_z)) = t^\prime$; for the projection formula (\cite[Proposition 2.5 (c)] {Fulton1998}) would else imply that the intersection number in (\ref{equation::yfiberintersection}) is zero. By \cite[Lemma L.6 and Theorem N.1]{Gunning1990a}, this implies that there exists some smooth point $y \in (\overline{X}|_z)^{\mathrm{an}}$ such that the rank of $(d\psi_{\underline{i}}^{(j)}|_{(\overline{X}|_{z})^{\mathrm{an}}})_y$  is $t^\prime=\dim(\overline{X}|_y)$. This means nothing else but $\ker(d\psi_{\underline{i}}^{(j)}) \cap \ker(d\pio)_{\IC_\nu}^{\mathrm{an}} \cap T_{\IC,y}^{1,0}(\overline{X}_{\IC_\nu}^{\mathrm{an}}) = \{ 0_y \}$, which shows that $y\notin E_{\underline{i}}^{(j)}$.

(b): We have
\begin{equation*}
\ker(d(\pio^{\mathrm{an}}_{\IC_\nu} \times \psi_{\underline{i}}^{(j)})) \cap T_{\IC,y}^{1,0}(\overline{X}_{\IC_\nu}^{\mathrm{an}}) = \ker (d\psi_{\underline{i}}^{(j)}) \cap \ker(d\pio)^{\mathrm{an}}_{\IC_\nu} \cap T_{\IC,y}^{1,0}(\overline{X}_{\IC_\nu}^{\mathrm{an}}) = \{ 0_y \} 
\end{equation*}
for all $y \in (\overline{X}^{\mathrm{an}}_{\IC_\nu} \cap (\overline{\pi}_{\IC_\nu}^{\mathrm{an}})^{-1}(U_j)) \setminus E_{\underline{i}}$.

(c), (d): Any real-analytic set is locally a union of finitely many real-analytic manifolds (e.g., by \cite[Theorem 2]{Lojasiewicz1964}). By compactness, we can hence write $X_{\underline{i}}$ as a union of finitely many real-analytic manifolds. It only remains to show that $X_{\underline{i}} \setminus E_{\underline{i}}$ has local dimension $d+d^\prime$ everywhere. Using
\begin{equation*}
X_{\underline{i}} \setminus E_{\underline{i}} = (\psi_{\underline{i}}^{(j)})^{-1}((S^1)^{t^\prime}) \cap (\overline{X}_{\IC_\nu}^{\mathrm{an}} \setminus E_{\underline{i}})
\end{equation*}
and (b), the two assertions follows from the standard fact that $(S^1)^{t^\prime} \subset (\Gm^{t^\prime})_{\IC_\nu}^{\mathrm{an}}$ is a real-analytic submanifold of dimension $t^\prime$.

\end{proof}

With the information of Lemma \ref{lemma::structuralinformation} at our disposal, we can conclude this section with an explicit description of the measures introduced in Proposition \ref{proposition1}. Let $\omega_{t^\prime}$ be the unique $(S^1)^{t^\prime}$-invariant $t^\prime$-form on the compact real Lie group $(S^1)^{t^\prime}$ such that $\int_{(S^1)^{t^\prime}} \omega_{t^\prime} = 2^{t^\prime}$. For each $j \in J$, the pullback $(\psi_{\underline{i}}^{(j)})^\ast \omega_{t^\prime}$ is a $t^\prime$-form on $(X_{\underline{i}}\cap \pi^{-1}(U_j) ) \setminus E_{\underline{i}}$. By \eqref{equation::absolutevalue} and $(S^1)^{t^\prime}$-invariance, these forms glue together to a $t^\prime$-form $\omega_{\underline{i}}$ on $X_{\underline{i}} \setminus E_{\underline{i}}$.

\begin{lemma} 
\label{lemma::measure_realanalytic}
For each $\underline{i} \in I_X$, the $(d+d^\prime)$-form $\omega_{\underline{i}} \wedge (c_1(\pio^\ast \No_\nu)|_{X_{\underline{i}} \setminus E_{\underline{i}}})^{\wedge d^\prime}$ is a positive $\mathscr{C}^\infty$-volume form on $X_{\underline{i}} \setminus E_{\underline{i}}$. For all $f \in \mathscr{C}^0(\overline{X}^{\mathrm{an}}_{\IC_\nu})$, we have 
\begin{equation}
\label{equation::measure_functional}
\int_{\overline{X}^{\mathrm{an}}_{\IC_\nu}} f c_1(\Mo_{\overline{G},\nu}|_{\overline{X}})^{\wedge t^\prime} \wedge c_1(\pio^\ast \No_\nu|_{\overline{X}})^{\wedge d^\prime} = t^\prime! \sum_{\underline{i} \in I_X } \int_{X_{\underline{i}} \setminus E_{\underline{i}}} f \omega_{\underline{i}} \wedge c_1(\pio^\ast \No_\nu)^{\wedge d^\prime}.
\end{equation}
\end{lemma}

\begin{proof} By linearity, it suffices to prove that
\begin{multline}
\label{equation::measure_description}
\int_{\overline{X}^{\mathrm{an}}_{\IC_\nu}} f 
c_1(\Mo_{\overline{G},i_1,\nu}|_{\overline{X}}) \wedge c_1(\Mo_{\overline{G},i_2,\nu}|_{\overline{X}}) \wedge \cdots \wedge
c_1(\Mo_{\overline{G},i_{t^\prime},\nu}|_{\overline{X}}) \wedge c_1(\pio^\ast \No_\nu)^{\wedge d^\prime} 
\\
= \int_{X_{\underline{i}} \setminus E_{\underline{i}}} f \omega_{\underline{i}} \wedge c_1(\pio^\ast \No_\nu)^{\wedge d^\prime}
\end{multline}
for each $\underline{i} \in I_X$ and all $f \in \mathscr{C}^{0}_c({\overline{X}}_{\IC_\nu}^{\mathrm{an}})$. Since the subsets $E_{\underline{i}} \subset {\overline{X}}_{\IC_\nu}^{\mathrm{an}}$ are locally pluripolar by Lemma \ref{lemma::structuralinformation} (a), we can further restrict to $f \in \mathscr{C}^{0}_c({\overline{X}}_{\IC_\nu}^{\mathrm{an}} \setminus E_{\underline{i}})$. Using a partition of unity and Lemma \ref{lemma::structuralinformation} (b), we can even restrict to the case where $f \in \mathscr{C}^{0}_c(U)$ with $U \subseteq {\overline{X}}_{\IC_\nu}^{\mathrm{an}} \setminus E_{\underline{i}}$ a relatively compact, open subset such that, for some $j \in J$, the map $\overline{\pi}_{\IC_\nu}^{\mathrm{an}} \times \psi_{\underline{i}}^{(j)}$ sends $U$ biholomorphically to some $U^\prime_0 \times \prod_{i=1}^{t^\prime} U_i^{\prime}$ where $U^\prime_0 \subset A_{\IC_\nu}^{\mathrm{an}}$ is an open subset and
\begin{equation*}
U_i^{\prime}=
\{ r e^{i\phi} \in \IC_\nu^\times \ | \ r \in (r_i,s_i), \phi \in (\alpha_i, \beta_i) \}, \ 0<r_i<s_i, \ |\alpha_i-\beta_i|<2\pi.
\end{equation*}
Let $\mathrm{pr}_0: U_0^\prime \times \prod_{i=1}^{t^\prime} U_i^\prime \rightarrow U_0^\prime$ be the standard projection. Using that $\Mo_{\overline{G},i,\nu}|_{U_j} \approx (\psi_i^{(j)})^\ast \Mo_{\IP^1,\nu}$ by \eqref{equation::pullback_logmetric} and that $c_1(\Mo_{\IP^1,\nu})=dd^c\left\vert \log |z| \right\vert$, a substitution along $\overline{\pi}_{\IC_\nu}^{\mathrm{an}} \times \psi_{\underline{i}}^{(j)}$ yields that the left-hand side of \eqref{equation::measure_description} equals
\begin{equation}
\label{equation::integral2}
\int_{U^\prime_0 \times \prod_{i=1}^{t^\prime} U_i^{\prime}} g dd^c \left\vert \log |z_1| \right\vert \wedge dd^c \left\vert \log |z_2| \right\vert \wedge \cdots \wedge dd^c \left\vert \log |z_{t^\prime}| \right\vert \wedge c_1(\pr_0^\ast \No_\nu|_{\overline{X}})^{\wedge d^\prime}
\end{equation}
with $g = f \circ (\overline{\pi}_{\IC_\nu}^{\mathrm{an}} \times \psi_{\underline{i}}^{(j)})|_U^{-1}$ and $z_1,\dots,z_{t^\prime}$ the standard coordinates on $\prod_{i=1}^{t^\prime} U_i^{\prime} \subseteq (\IC_\nu^\times)^{t^\prime}$. The Borel measure
\begin{equation*}
dd^c \left\vert \log |z_{1}| \right\vert \wedge dd^c \left\vert \log |z_{2}| \right\vert \wedge \cdots \wedge dd^c \left\vert \log |z_{t^\prime}| \right\vert \wedge c_1(\pr_0^\ast \No_\nu|_{U_0^\prime})^{\wedge d^\prime}
\end{equation*}
is the product of the measures induced by $dd^c \left\vert \log |z_{i}| \right\vert$, $i \in \{ 1,\dots, t^\prime \}$, on $U_i^\prime \subset \IC_\nu^\times$ and the measure $c_1(\pr_0^\ast \No_\nu|_{U_0^\prime})^{\wedge d^\prime}$ on $U_0^\prime$; indeed, this follows from the corresponding fact for $\mathscr{C}^\infty$-forms by plurisubharmonic smoothings (combine \cite[Proposition 1.42]{Guedj2017} and \cite[Corollary 1.6]{Demailly1993}). By Fubini's theorem \cite[Theorem 8.8]{Rudin1987}, the integral (\ref{equation::integral2}) hence equals
\begin{equation}
\label{equation::integral3}
\int_{U_0^\prime} \left( \int_{U_{t^\prime}^\prime} \left( \cdots \int_{U_2^\prime} \left( \int_{U_1^\prime} g dd^c \left\vert \log|z_1| \right\vert \right) dd^c \left\vert \log |z_2| \right\vert \right) \cdots \ dd^c \left\vert \log |z_{t^\prime}| \right\vert \right) c_1(\pr_0^\ast \No_\nu|_{\overline{X}})^{\wedge d^\prime}.
\end{equation}
It is an elementary exercise (see Appendix \ref{appendixB}) to compute that
\begin{equation*}
\int_{\IC_\nu^\times} h(z) dd^c \left\vert \log |z| \right\vert = \int_{[0,2\pi]} h(e^{i\phi})\frac{d\phi}{\pi}
\end{equation*}
for every $h \in \mathscr{C}^0_c(\IC_\nu^\times)$. Using Fubini's Theorem once again, we see that (\ref{equation::integral3}) equals
\begin{equation*}
\int_{U^\prime_0} \left( \int_{(S^1)^{t^\prime}} g \omega_{t^\prime} \right) c_1(\No_\nu|_{\overline{X}})^{\wedge d^\prime} = \int_{U^\prime_0 \times (S^1)^{t^\prime}} g \omega_{t^\prime} \wedge c_1(\mathrm{pr}_0^\ast \No_\nu|_{\overline{X}})^{\wedge d^\prime}.
\end{equation*}
By the substitution formula and Lemma \ref{lemma::structuralinformation} (d), this equals $\int_{U} f \omega_{\underline{i}}$ as claimed. 
\end{proof}

The $(d+d^\prime)$-form $\omega_{\underline{i}} \wedge c_1(\overline{\pi}^\ast \overline{N}_\nu)^{\wedge d^\prime}$ restricts to an everywhere non-zero volume form on $X_{\underline{i}} \setminus E_{\underline{i}}$ for each $\underline{i} \in I_X$. It hence prescribes an orientation on each $X_{\underline{i}} \setminus E_{\underline{i}}$. We tacitly mean this orientation in the following.

\begin{lemma} 
\label{lemma::riemannian_metric}
For each $\underline{i} \in I_X$, let $\mu_{\underline{i}}$ be the Borel measure on $\overline{G}$ associated with the functional
\begin{equation*}
\mathscr{C}^0(\overline{X}^{\mathrm{an}}_{\IC_\nu}) \longrightarrow \IR, \ f \longmapsto \int_{X_{\underline{i}} \setminus E_{\underline{i}}} f \omega_{\underline{i}} \wedge c_1(\pio^\ast \No_\nu)^{\wedge d^\prime}.
\end{equation*}
There exists a Riemannian metric $g_{\underline{i}}$ on $\overline{G}$ such that
\begin{equation}
\label{equation::boundmeasure}
\mu_{\underline{i}}(U) \leq \int_{U} \vol(g_{\underline{i}}|_{X_{\underline{i}}\setminus E_{\underline{i}}})
\end{equation}
for each open subset $U \subseteq X_{\underline{i}}\setminus E_{\underline{i}}$. Furthermore, $g_{\underline{i}}$ can be chosen independently of $\underline{i} \in I_X$.
\end{lemma}

\begin{proof} Consider $((\IP^1_{\IC_\nu})^{t^\prime})^{\mathrm{an}}$ with its standard $(S^1)^{t^\prime}$-action $l: (S^{1})^{t^\prime} \times ((\IP^1_{\IC_\nu})^{t^\prime})^{\mathrm{an}} \rightarrow ((\IP^1_{\IC_\nu})^{t^\prime})^{\mathrm{an}}$. For each $s \in (S^1)^{t^\prime}$, we have then a left-multiplication map
\begin{equation*}
l_s: ((\IP^1_{\IC_\nu})^{t^\prime})^{\mathrm{an}} \longrightarrow ((\IP^1_{\IC_\nu})^{t^\prime})^{\mathrm{an}}, \ x \longmapsto l(s,x).
\end{equation*}	
Let $g^\prime$ be an arbitrary non-degenerate Riemann metric on $((\IP^1_{\IC_\nu})^{t^\prime})^{\mathrm{an}}$. Setting 
\begin{equation*}
g_x(t,t^\prime)= \int_{s \in (S^1)^{t^\prime}} g^\prime_{l_s(x)}((dl_s)_x t, (dl_s)_{x}( t^\prime)) \omega_{t^\prime}
\end{equation*}	
for all $x \in ((\IP^1_{\IC_\nu})^{t^\prime})^{\mathrm{an}}$ and all $t,t^\prime \in T_{\IR,x} ((\IP^1_{\IC_\nu})^{t^\prime})^{\mathrm{an}}$, we obtain a $(S^1)^{t^\prime}$-invariant Riemann metric on $((\IP^1_{\IC_\nu})^{t^\prime})^{\mathrm{an}}$. Its restriction $g|_{(S^1)^{t^\prime}}$ is a Riemann metric on the compact manifold $(S^1)^{t^\prime}$. Because of its $(S^1)^{t^\prime}$-invariance, the associated volume form $\vol(g|_{(S^1)^{t^\prime}})$ is a Haar measure on $(S^1)^{t^\prime}$. We conclude that $\vol(g|_{(S^1)^{t^\prime}})$ is a positive multiple of the $t^\prime$-form $\omega_{t^\prime}$. By rescaling, we can even assure that
\begin{equation}
\label{equation::estimate1}
\vol(g|_{(S^1)^{t^\prime}})=\omega_{t^\prime}
\end{equation}
The $(S^1)^{t^\prime}$-invariance of $g$ allows us to define a (in general degenerate) Riemannian metric $g_{\underline{i},\mathrm{tor}}$ on $\overline{G}$ by stipulating locally that $g_{\underline{i},\mathrm{tor}}|_{U_j} = (\psi_{\underline{i}}^{(j)})^\ast g$. % By construction, the volume form $\vol(g_{\underline{i},\mathrm{tor}}|_{X_{\underline{i}}\setminus E_{\underline{i}}})$ coincides with $\omega_{\underline{i}}$.
	
Furthermore, let $g^{\dprime}$ be the non-degenerate Riemann metric associated with the positive definite $(1,1)$-form $c_1(\overline{N}_\nu)$ on $A_{\IC_\nu}^{\mathrm{an}}$. By \cite[Lemma 3.8]{Voisin2007}, we know that 
\begin{equation}
\label{equation::estimate2}
\int_U c_1(\overline{N}_\nu)^{\wedge d^\prime} = d^\prime! \int_U \vol(g^{\dprime}|_{\pio(X)}) = \int_U \vol((d^\prime!)^{2/d^\prime} g^{\prime \prime})
\end{equation}
for every open $U \subseteq \pio(X)^{\mathrm{sm}}$. We set $g_{\mathrm{ab}} = (d^\prime!)^{2/d^\prime} (d\pio)^\ast g^{\prime \prime}$.

We claim that $g_{\underline{i}}=g_{\underline{i},\mathrm{tor}} + g_{\mathrm{ab}}$ is the seeked (possibly degenerate) Riemann metric on $\overline{G}$. Indeed, for each $x \in X_{\underline{i}} \setminus E_{\underline{i}}$ the tangent space $T_x X_{\underline{i}}$ decomposes as $\ker(d\psi_{\underline{i}}^{(j)}|_{X_{\underline{i}}})_x \oplus \ker(d\pio|_{X_{\underline{i}}})_x$ by Lemma \ref{lemma::structuralinformation} (b). By construction, $g_{\underline{i},\mathrm{tor}}$ (resp.\ $g_{\mathrm{ab}}$) is zero on $\ker(d\psi_{\underline{i}}^{(j)})_x$ (resp.\ $\ker(d\pio)_x$). We deduce
\begin{align*}
\vol(g_{\underline{i}}|_{T_x X_{\underline{i}}})
&= \vol(g_{\underline{i},\mathrm{tor}}|_{T_x X_{\underline{i}} \cap \ker(d\pio)}) \wedge 
\vol(g_{\mathrm{ab}}|_{T_x X_{\underline{i}} \cap \ker(d\psi_{\underline{i}}^{(j)})}) \\
&= d^\prime! \cdot (d\psi_{i}^{(j)}|_{T_xX_{\underline{i}}})^\ast \vol(g|_{(S^{1})^{t^\prime}}) \wedge  (d\pio_{T_xX_{\underline{i}}})^\ast \vol(g^\dprime|_{\pio(X)}) \\
&= \omega_{\underline{i},x} \wedge (c_1(\pio^\ast \No_\nu)|_{T_xX_{\underline{i}}})^{\wedge d^\prime},
\end{align*}
where we used \eqref{equation::estimate1} and \eqref{equation::estimate2} in the third equality. The estimate \eqref{equation::boundmeasure} follows and is in fact an equality.

For the last assertion, we can just take the sum $g = \sum_{\underline{i} \in I_X} g_{\underline{i}}$ of the already constructed Riemannian metrics $g_{\underline{i}}$, $\underline{i} \in I_X$. Inequality \eqref{equation::boundmeasure} with $g_{\underline{i}}$ replaced by $g$ is then a straightforward consequence of Minkowski's determinant inequality (see e.g.\ \cite[Corollary II.3.21]{Bhatia1997}). 
\end{proof}

\section{The Bogomolov Conjecture}
\label{section::bogomolov}

Our argument for deducing (BC) from the archimedean case of Proposition \ref{proposition1} follows Zhang's argument \cite{Zhang1998}, which itself is a generalization of an argument due to Ullmo \cite{Ullmo1998}. The new difficulty is that the measures $\mu_\nu$, $\nu \in \Sigma_\infty(K)$, from Proposition \ref{proposition1} are not described by smooth differential forms, which is why we need Lemma \ref{lemma::measure_realanalytic}. %For an algebraic torus $G=\Gm^t$ or a split semiabelian variety $G=\Gm^t \times A$, our proof of $\mathrm{(BC)}$ from $\mathrm{(EC)}$ differs from the one given, respectively, by Zhang \cite[Theorem 6.2]{Zhang1995}, which uses the Manin-Mumford Conjecture and not $\mathrm{(EC)}$, or Chambert-Loir \cite[Section 7]{Chambert-Loir1999}. In fact, these proofs do not seem to extend well to general semiabelian varieties. 

\begin{proposition} \label{proposition2} $\mathrm{(BC)}$ is true for every semiabelian variety $G$ over $\IQbar$.
\end{proposition}

Before proving the proposition, we start with a lemma.

\begin{lemma} \label{lemma::genericembedding} Let $X$ be a geometrically irreducible subvariety of $G$. Assume that the stabilizer group
\begin{equation*}
\mathrm{Stab}_{G_{\IQbar}}(X_{\IQbar}) = \{ g \in G_{\IQbar} \ | \ g + X_{\IQbar} = X_{\IQbar} \}
\end{equation*}
is trivial (i.e., equal to $\{ e_{G_{\IQbar}} \}$). For all integers $m \gg_X 1$, the algebraic map
\begin{equation*}
\alpha_m: X^m \longrightarrow G^{m-1}, \ (x_1,x_2,\dots,x_m) \longmapsto (x_1-x_2,x_2-x_3,\dots, x_{m-1}-x_m),
\end{equation*}
is then generically finite of degree $1$.
\end{lemma}
\begin{proof}
This can be proven in the same way as \cite[Lemma 3.1]{Zhang1998}.
\end{proof}

With this lemma, we can start the main proof of this section.

\begin{proof}[Proof of Proposition \ref{proposition2}] 
We may and do assume that $X$ is of positive dimension. Let $\Gm^t$ be the toric part of $G$. We first reduce to the case $\mathrm{Stab}_{G_{\IQbar}}(X_{\IQbar}) = \{ e_{G_{\IQbar}} \}$. By enlarging $K$, we can assume that $\mathrm{Stab}_{G_{\IQbar}}(X_{\IQbar})$ is $H_{\IQbar}$ for some algebraic subgroup $H \subseteq G$. Consider the quotient $\varphi: G \twoheadrightarrow G/H =: G^\prime$. It is well-known that $G^\prime$ is a semiabelian variety (\cite[Corollary 5.4.6]{Brion2017}). Denote its toric part by $\Gm^{t^\prime}$ and its abelian quotient by $\pi^\prime: G^\prime \rightarrow A^\prime$. The map $\varphi$ is a homomorphism (\cite[Theorem 2]{Iitaka1976}), and the image $X^\prime = \varphi(X)$ is an irreducible subvariety of $G^\prime$ satisfying $\mathrm{Stab}_{G^\prime_{\IQbar}}(X^\prime_{\IQbar})= \{ e_{G^\prime_{\IQbar}} \}$. Evidently, $X^\prime$ is not the translate by a torsion point of a connected subgroup of $G^\prime$ unless $X$ is so.

We can then reduce $(\mathrm{BC})$ for $X$ to $(\mathrm{BC})$ for $X^\prime$. Let $\overline{G} = G \times^{\Gm^t} (\IP^1)^t$ and $\overline{G}^\prime = G^\prime \times^{\Gm^{t^\prime}}  (\IP^1)^{t^\prime}$ be the standard compactifications from Subsection \ref{section::semiabelian_compact} and let $\overline{\pi}:\overline{G} \rightarrow A$ and $\overline{\pi}^\prime: \overline{G}^\prime \rightarrow A^\prime$ be the associated projections. We also use the nef line bundles $M_{\overline{G}}$ and $M_{\overline{G}^\prime}$ as defined there, and we fix an ample symmetric line bundle $N$ (resp.\ $N^\prime$) on $A$ (resp.\ $A^\prime$). Set $L = M_{\overline{G}} \otimes N$ (resp.\ $L^\prime = M_{\overline{G}^\prime} \otimes N^\prime$) and endow all line bundles with the adelic metrics from Subsection \ref{section::semiabelianvarietyheights}. By Lemma \ref{lemma::heightcomparison} applied to $\varphi$, we have
\begin{equation}
\label{equation::heights}
h_{\widetilde{L}^\prime}(\varphi(x)) \ll_{N,N^\prime,\varphi} h_{\widetilde{L}}(x)
\end{equation}
for every closed point $x \in G$. If there exists some $\varepsilon > 0$ such that
\begin{equation*}
X_\varepsilon^\prime = \{ \text{closed point }x^\prime \in X^\prime \ | \ h_{\widetilde{L}^\prime}(x^\prime) \leq \varepsilon \}
\end{equation*}
is not Zariski-dense in $X^\prime$, then its preimage $\varphi^{-1}(X_\varepsilon^\prime)$ is likewise not Zariski-dense in $X$. By (\ref{equation::heights}), $\varphi^{-1}(X_\varepsilon^\prime)$ contains 
\begin{equation*}
X_{\varepsilon^\prime} = \{ \text{closed point } x \in X\ | \ h_{\widetilde{L}}(x) \leq \varepsilon^\prime \}
\end{equation*}
for some sufficiently small $\varepsilon^\prime > 0$. We may hence assume that $\mathrm{Stab}_{G_{\IQbar}}(X_{\IQbar}) = \{ e_{G_{\IQbar}}\}$ in the following.

We argue by contradiction and assume that $(\mathrm{BC})$ is wrong for $X$. This means that there exists a Zariski-dense sequence $(x_i) \in X^\IN$ of small points. By Lemma \ref{lemma::genericembedding}, we can fix some integer $m$ such that $\alpha_m$ is generically finite of degree $1$. Pick a bijection 
\begin{equation*}
\IN \longrightarrow \IN^m, \ i \longmapsto (\phi_1(i),\dots, \phi_m(i)),
\end{equation*}
and define the new sequence
\begin{equation*}
y_i = (x_{\phi_1(i)}, \dots, x_{\phi_m(i)}), \ i \in \IN,
\end{equation*} 
which is clearly Zariski-dense in $X^m$. Using \cite[Lemma 4.1]{Zhang1998}, we can even assume that $(y_i)$ is $X^m$-generic by passing to a subsequence. From their construction, both $(y_i) \in X^m(\IQbar)$ and $(\alpha_m(y_i)) \in G^{m-1}(\IQbar)$ are sequences of small points in $G^m$ and $G^{m-1}$, respectively. The sequence $(\alpha_m(y_i))$ is also $\alpha_m(X^m)$-generic. Let $U \subseteq X^m$ be a dense open subset such that $\alpha_m|_U: U \rightarrow \alpha_m(U)$ is an isomorphism. %We can also assume that $U$ is disjoint from the ($m$-fold) diagonal $\Delta(X)\subset X^m$. 
For sufficiently large $i$, we have $y_i \in U$ and $\alpha_m(y_i) \in \alpha_m(U)$. 

For the sequel, fix an arbitrary archimedean place $\nu \in \Sigma_\infty(K)$. Proposition \ref{proposition1} yields Borel measures $\mu_1$ and $\mu_2$ on $(\overline{X}^m)_{\IC_\nu}^{\mathrm{an}}$ and $\alpha_m(\overline{X}^m)_{\IC_\nu}^{\mathrm{an}}$, respectively, such that the following two assertions are true:
\begin{enumerate}
\item[(a)] For every $f \in \mathscr{C}^0_c((X^m)_{\IC_\nu}^{\mathrm{an}})$, we have
\begin{equation*}
\label{equation::limit2}
\frac{1}{\# \mathbf{O}_\nu(y_i)}\sum_{y \in \mathbf{O}_\nu(y_i)} f(y) \longrightarrow \int_{(X^m)_{\IC_\nu}^{\mathrm{an}}} f \mu_1 \quad (i \rightarrow \infty).
\end{equation*}
\item[(b)] For every $f_0 \in \mathscr{C}^0_c(\alpha_m(X^m)_{\IC_\nu}^{\mathrm{an}})$, we have
\begin{equation*}
\label{equation::limit1}
\frac{1}{\# \mathbf{O}_\nu(\alpha_m(y_i))}\sum_{y \in \mathbf{O}_\nu(\alpha_m(y_i))} f_0(y) \longrightarrow \int_{\alpha_m(X^m)_{\IC_\nu}^\mathrm{an}}f_0 \mu_2 \quad (i \rightarrow \infty).
\end{equation*}
\end{enumerate}
Setting $f=f_0 \circ \alpha_m$, we infer that
\begin{equation}
\label{equation::bogomolov_integrals}
\int_{(X^m)_{\IC_\nu}^{\mathrm{an}}} (f_0 \circ \alpha_m) \mu_1 =  \int_{\alpha_m(X^m)_{\IC_\nu}^\mathrm{an}}f_0 \mu_2
\end{equation}
for every $f_0 \in \mathscr{C}^0_c(\alpha_m(X^m)_{\IC_\nu}^{\mathrm{an}})$. We derive a contradiction from this equality through a closer look at the measures $\mu_1$ and $\mu_2$.

Applying Lemma \ref{lemma::measure_realanalytic} with $X \subseteq G$ replaced by $\alpha_m(X^m) \subseteq G^{m-1}$, we obtain that there exist finitely many (embedded) real-analytic submanifolds $\{\mathcal{M}_k\}_{1 \leq k \leq K}$ of $\alpha_m(X^m)^{\mathrm{an}}_{\IC_\nu}$, each endowed with a positive $\mathscr{C}^\infty$-volume form $\Omega_k$, such that
\begin{equation}
\label{equation::integral4}
\int_{\alpha_m(X^m)^{\mathrm{an}}_{\IC_\nu}} f_0 \mu_2 = \sum_{k=1}^K \int_{\mathcal{M}_k} f_0 \Omega_k.
\end{equation}
In addition, Lemma \ref{lemma::riemannian_metric} provides us with a Riemannian metric $g$ on $\overline{G}^{m-1}$ such that
\begin{equation}
\label{equation::integral_by_volume}
\int_{U} \Omega_k  \leq \int_{U} \vol(g|_{\mathcal{M}_k})
\end{equation}
for each $k \in \{ 1, \dots, K\}$ and every open subset $U \subseteq \mathcal{M}_k$. As $(\alpha_m(X^m) \setminus \alpha_m(U))^{\mathrm{an}}_{\IC_\nu}$ is a locally pluripolar subset of $\alpha_m(X^m)^{\mathrm{an}}_{\IC_\nu}$, the measure $\mu_2$ does not attach any mass to it (see Subsection \ref{section::borelmeasure}). We can hence assume that 
\begin{equation*}
\mathcal{M}_k \cap \left( \alpha_m(X^{m}) \setminus \alpha_m(U) \right)^{\mathrm{an}}_{\IC_\nu} = \emptyset
\end{equation*}
for all $k\in \{ 1, \dots, K \}$. Each set $\alpha_m^{-1}(\mathcal{M}_k)$ is an (embedded) real-analytic manifold contained in $U$. %from $\Delta(X)$. 
By the substitution formula, the right-hand side of \eqref{equation::integral4} equals
\begin{equation*}
\sum_{k=1}^{K}
\int_{\alpha_m^{-1}(\mathcal{M}_k)} (f_0 \circ \alpha_m) \alpha_m^\ast \Omega_k.
\end{equation*}

A closer inspection of Proposition \ref{proposition1} shows that the measure $\mu_1$ on $X^m$ equals the $m$-fold product measure
\begin{equation*}
\left(c_1(\Mo_{G,\nu}|_X)^{d-d^\prime} \wedge c_1(\pio^\ast \No_\nu|_X)^{d^\prime}\right)^{\times m}.
\end{equation*}
By Lemma \ref{lemma::measure_realanalytic} applied to $X \subseteq G$, there exist (embedded) real-analytic submanifolds $\{\mathcal{M}_k^\prime\}_{1 \leq k \leq K^\prime}$ of $X^{\mathrm{an}}_{\IC_\nu}$ and a positive volume form $\Omega_k^\prime$ on each $\mathcal{M}_k^\prime$ such that
\begin{equation*}
\int_{(X^m)_{\IC_\nu}^{\mathrm{an}}} f \mu_1 = \sum_{1 \leq k_1,\dots,k_{m} \leq K^\prime} \int_{\mathcal{M}_{k_1}^\prime \times \cdots \times \mathcal{M}_{k_m}^\prime} f (\Omega_{k_1}^\prime \wedge \cdots \wedge \Omega_{k_{m}}^\prime)
\end{equation*}
for all $f \in \mathscr{C}^0_c((X^m)_{\IC_\nu}^{\mathrm{an}})$.

Combining this with \eqref{equation::bogomolov_integrals}, we obtain the identity
\begin{equation}
\label{equation::integral5}
\sum_{1 \leq k_1,\dots,k_{m} \leq K^\prime} \left( \int_{\mathcal{M}_{k_1}^\prime \times \cdots \times \mathcal{M}_{k_m}^\prime} f (\Omega_{k_1}^\prime \wedge \cdots \wedge \Omega_{k_{m}}^\prime) \right)
=
\sum_{k=1}^{K}
\int_{\alpha_m^{-1}(\mathcal{M}_k)} f \alpha_m^\ast \Omega_k
\end{equation}
for all $f \in \mathscr{C}^0_c(U^{\mathrm{an}}_{\IC_\nu})$. For every $\varepsilon > 0$, there exists an open set $V \supset (X^m \setminus U)$ such that both sides of \eqref{equation::integral5} are less than $\varepsilon$ for the indicator function $f = \mathbf{1}_{V}$. By approximation, this implies that \eqref{equation::integral5} is also valid for general functions $f \in \mathscr{C}^0_c((X^m)^{\mathrm{an}}_{\IC_\nu})$. Comparing supports, we obtain that 
\begin{equation*}
\bigcup_{1 \leq k_1,\dots,k_{m} \leq K^\prime} \overline{(\mathcal{M}_{k_1}^\prime \times \cdots \times \mathcal{M}_{k_m}^\prime)} = \bigcup_{k=1}^K \overline{\alpha_m^{-1}(\mathcal{M}_k)}.
\end{equation*}
Since $\mu_1$ associates a zero measure to any closed real analytic submanifold of dimension $< m(d+d^\prime)$, the dimension of each $\alpha_m^{-1}(\mathcal{M}_k)$ is $m(d+d^\prime)$. In other words, each $\alpha_m^{-1}(\mathcal{M}_k)$ is an open subset of 
\begin{equation*}
\bigcup_{1 \leq k_1,\dots,k_{m} \leq K^\prime} \overline{(\mathcal{M}_{k_1}^\prime \times \cdots \times \mathcal{M}_{k_m}^\prime)}.
\end{equation*}
Pick some $x \in \mathcal{M}_k^\prime$, $k \in \{1,\dots, K^\prime \}$, such that there exists an open neighborhood $x \in V \subset X$ disjoint from $\mathcal{M}_1^\prime \cup \cdots \cup \mathcal{M}_{k-1}^\prime \cup \mathcal{M}_{k+1}^\prime \cup \cdots \cup \mathcal{M}_{K^\prime}^\prime$. Shrinking $V$ if necessary, there exists a real-analytic isomorphism $\psi: (-1,1)^{d+d^\prime} \rightarrow \mathcal{M}_k^\prime \cap V$ with $\psi(0,\dots,0)=x$. For convenience, we write 
\begin{equation*}
\psi_m=\psi \times \cdots \times \psi: (-1,1)^{(d+d^\prime)m} \rightarrow G^m
\end{equation*}
for its $m$-fold product and $\Delta_m \subset (-1,1)^{(d+d^\prime)m}$ for the diagonally embedded copy of $(-1,1)^{d+d^\prime}$.

Let $f\in \mathscr{C}^0_c(V^m)$ be non-negative with $f(x,\dots,x) > 0$. Defining $B_\varepsilon = ((-\varepsilon,\varepsilon)^{(d+d^\prime)})^m \subset \IR^{(d+d^\prime)m}$ and writing $\vol(B_\varepsilon)$ for its volume with respect to the standard Euclidean metric, we note that
\begin{align*}
\label{equation::integral_estimate1}
\int_{\psi_m(B_\varepsilon) \cap \alpha_m^{-1}(\mathcal{M}_k)} f\alpha_m^\ast \Omega_k 
= 
\int_{\alpha_m \circ \psi_m(B_\varepsilon) \cap \mathcal{M}_k} f \Omega_k
\leq 
|f|_{\mathrm{sup}} \int_{\alpha_m \circ \psi_m(B_\varepsilon) \cap \mathcal{M}_k} \vol(g|_{\mathcal{M}_k})
\end{align*}
for each $k \in \{1,\dots,K\}$ by \eqref{equation::integral_by_volume}. Since the differential $d(\alpha_m \circ \psi_m)_{(0,\dots,0)}$ annihilates the $(d+d^\prime)$-dimensional $\IR$-subspace $T_{(0,\dots,0)}\Delta_m$, we furthermore obtain
\begin{align*}
\int_{\alpha_m \circ \psi_m(B_\varepsilon) \cap \mathcal{M}_k} \vol(g|_{\mathcal{M}_k}) = \int_{B_\varepsilon} \vol((\alpha_m\circ\psi_m)^\ast g)
\ll_{\psi_m, \alpha_m, g, \Omega_1,\dots, \Omega_K} \varepsilon^{d+d^\prime} \vol(B_\varepsilon) \nonumber
\end{align*}
for sufficiently small $\varepsilon>0$. Using \eqref{equation::integral5}, we obtain however
\begin{equation*}
\sum_{i=1}^K \int_{\psi_m(B_\varepsilon) \cap \alpha_m^{-1}(\mathcal{M}_k)} f\alpha_m^\ast \Omega_k =\int_{\mathcal{M}_{k}^\prime \times \cdots \times \mathcal{M}_{k}^\prime} f (\Omega_{k}^\prime \wedge \cdots \wedge \Omega_{k}^\prime)
\gg_{f, \psi_m, \Omega_k^\prime} \vol(B_\varepsilon)
\end{equation*}
since $\Omega_k^\prime$ is (strictly) positive at $x$ by construction. We obtain a contradiction by combining this with the other estimates above and letting $\varepsilon \rightarrow 0$.
\end{proof}

\section{The Strong Equidistribution Conjecture}
\label{section::mainthm}

For completeness, we give the well-known argument for $\mathrm{(EC)} \wedge \mathrm{(BC)} \Rightarrow \mathrm{(SEC)}$ (see \cite[p.\ 165]{Zhang1998}).

\begin{proof}[Proof of Theorem \ref{theorem::main}] 
Let $(x_i) \in G^\IN$ be a strict sequence of small height. If $(x_i)$ were not $G$-generic, there would exist a proper algebraic subvariety $X$ and a Zariski-dense subsequence $(x_{n_i}) \in X^{\IN}$ of small points. Proposition \ref{proposition2} implies that $X$ is contained in a finite union of proper algebraic subgroups of $G$. This contradicts the strictness of $(x_i)$ and hence $(x_i)$ must be $G$-generic. This allows us to apply Proposition \ref{proposition1}, concluding the proof of $(\mathrm{SEC})$.
\end{proof}

\begin{appendix}
\section{Global regularization and archimedean local heights} 

\label{appendixA}

In this appendix, we indicate how to extend the archimedean local heights defined by Gubler \cite{Gubler2003} to semipositive $\mathscr{C}^0$-metrics through global regularization. This allows us to use the facts on archimedean local heights that are provided in \cite{Gubler2003} also in this more general setting.
%This should be well-known, but a comfortable reference is unfortunately lacking. 

We work with a projective $K$-variety $X$ and an archimedean place $\nu \in \Sigma_\infty(K)$. Let $\Lo_i = (L_i, \Vert \cdot \Vert_i)$, $1 \leq i \leq d^\prime +1$, be semipositive $\nu$-metrized line bundles on $X$ and let $\widehat{D}_i=(\Lo_i, Y_i, \mathbf{s}_i)$, $1 \leq i \leq d^\prime +1$, denote $\nu$-metrized pseudo-divisors on $X$. If each $\nu$-metric $\Vert \cdot \Vert_i$ ($1 \leq i \leq d^\prime +1$) is $\mathscr{C}^\infty$, Gubler \cite[Definition 3.3]{Gubler2003} defines a local height $\lambda_{\widehat{D}_1,\widehat{D}_2,\dots,\widehat{D}_{d^\prime+1}}(\mathfrak{Z})$ for each $d^\prime$-cycle $\mathfrak{Z}$ on $X$ that satisfies the condition
\begin{equation}
\label{equation::standarddisjointnesscondition}
Y_1 \cap Y_2 \cap \cdots \cap Y_{d^\prime+1} \cap |\mathfrak{Z}| = \emptyset.
\end{equation}
To reduce to this case, we choose an ample line bundle $M$ on $X$ and a smooth $\nu$-metric $\Vert \cdot \Vert_0$ on $M$ such that $\Mo=(M, \Vert \cdot \Vert_0)$ is a strictly positive $\nu$-metrized line bundle. For each cycle $\mathfrak{Z}$ satisfying \eqref{equation::standarddisjointnesscondition}, we can choose global sections $\mathbf{t}_i: X \rightarrow M$ ($1 \leq i \leq d^\prime +1$) such that
\begin{equation}
\label{equation::standarddisjointnesscondition2}
(Y_1 \cup |\Div(\mathbf{t}_1)|) \cap (Y_2 \cup |\Div(\mathbf{t}_2)|) \cap \cdots \cap (Y_{d^\prime+1} \cup |\Div(\mathbf{t}_{d^\prime+1})|) \cap |\mathfrak{Z}| = \emptyset.
\end{equation}

As the $\nu$-metrized line bundle $\Lo_i^{\otimes n} \otimes \Mo$ is strictly positive, we can approximate its continuous metric by $\mathscr{C}^\infty$-metrics. 

\begin{lemma} 
\label{lemma::smoothening}
	For each $i \in \{ 1, \dots, d^\prime + 1\}$ and each integer $k\geq 1$, there exists a smooth $\nu$-metric $\Vert \cdot \Vert_i^{(n,k)}$ on $L_i^{\otimes n} \otimes M$ such that, writing $\Lo_i^{(n,k)}=(L_i^{\otimes n} \otimes M, \Vert \cdot \Vert_i^{(n,k)})$, we have
\begin{enumerate}
	\item[(a)] $c_1(\Lo_i^{(n,k)}) \geq (1-1/k) (n \cdot c_1(\Lo_i) + c_1(\Mo)) \geq 0$, and 
	\item[(b)] %\Vert \cdot \Vert^{(n,k)}_i/(\Vert \cdot \Vert^{\otimes n}_i \otimes \Vert \cdot \Vert_0) \longrightarrow 1 \ \ (k\rightarrow \infty)
	$e^{-1/2k} \leq \Vert \cdot \Vert^{(n,k)}_i/(\Vert \cdot \Vert^{\otimes n}_i \otimes \Vert \cdot \Vert_0) \leq 1$ everywhere on $X^{\mathrm{an}}_{\IC_\nu}$.
\end{enumerate}
\end{lemma}

\begin{proof} This follows from (the proof of) \cite[Theorem 4.6.1]{Maillot2000} (with $\lambda=1/k$) for the strictly positive $\nu$-line bundle $\Lo_i^{(n,k)}$.
\end{proof}

We define the $\nu$-metrized pseudo-divisor $\widehat{D}^{(n,k)}_i = (\Lo_i^{(n,k)}, Y_i \cup |\Div(\mathbf{t}_i)|, \mathbf{s}_i^{\otimes n} \otimes \mathbf{t}_i)$ for each $i \in \{ 1,\dots, d^\prime +1 \}$. Using \cite[Definition 3.3]{Gubler2003}, we obtain a local height $$\lambda_{\widehat{D}_1^{(n,k)}, \widehat{D}_2^{(n,k)},\dots, \widehat{D}_{d^\prime+1}^{(n,k)}}(\mathfrak{Z})$$ for all integers $n,k \geq 1$.

Our main goal in this appendix is to establish the following lemma.

\begin{lemma} For each $d^\prime$-cycle $\mathfrak{Z}$ on $X$ satisfying (\ref{equation::standarddisjointnesscondition}), the double limit
\begin{equation}
\label{equation::newdefinition}
\lambda_{\widehat{D}_1,\widehat{D}_2,\cdots,\widehat{D}_{d^\prime+1}}(\mathfrak{Z}) :=
\lim_{n \rightarrow \infty} n^{-(d^\prime+1)} \left( \lim_{k \rightarrow \infty} \lambda_{\widehat{D}_1^{(n,k)}, \widehat{D}_2^{(n,k)},\dots, \widehat{D}_{d^\prime+1}^{(n,k)}}(\mathfrak{Z})\right)
\end{equation}
exists and depends only on $\widehat{D}_i$ and $X$ (in particular, not on $\overline{M}$, the sections $\mathbf{t}_i$ nor on the $\nu$-metrics $\Vert \cdot \Vert^{(n,k)}_i$). Setting \eqref{equation::newdefinition} extends the definition of \cite[Definition 3.3]{Gubler2003} such that the induction formula
\begin{multline}
\label{equation::induction_formula_2}
\lambda_{\widehat{D}_1,\widehat{D}_2,\cdots,\widehat{D}_{d^\prime+1}}([Z]) = \\
\lambda_{\widehat{D}_1,\widehat{D}_2,\cdots,\widehat{D}_{d^\prime}}( [\Div(\mathbf{s}_{d^\prime+1}|_{Z})])
- \int_{Z^{\mathrm{an}}_{\IC_\nu}} 
\log \Vert \mathbf{s}_{d^\prime+1} \Vert_{d^\prime+1} c_1(\Lo_{1}) \wedge \cdots \wedge c_1(\Lo_{d^\prime})
\end{multline}
holds for each irreducible subvariety $Z \subseteq X$ of dimension $d$. (If $d= 0$, this means
\begin{equation*}
\label{equation::point_height}
\lambda_{\widehat{D}_1}([Z]) = - \sum_{x \in Z_{\IC_\nu}^{\mathrm{an}}}\log\Vert \mathbf{s}_{1}(x)\Vert_{1}.)
\end{equation*} 
\end{lemma}

\begin{proof} By linearity, we can restrict to the case where $\mathfrak{Z} = [Z]$ for an irreducible subvariety $Z \subseteq X$ of dimension $d^\prime$. We use an induction on the dimension $d^\prime$ of $X$. The case $d^\prime = 0$ is straightforward as
\begin{equation*}
n^{-1}\lambda_{\widehat{D}_1^{(n,k)}}([Z]) \longrightarrow - \sum_{x \in Z_{\IC_\nu}^{\mathrm{an}}} \left( \log \Vert \mathbf{s}_1(x) \Vert_{1} + n^{-1} \log \Vert \mathbf{t}_1(x) \Vert_{0} \right)
\ \ (k \rightarrow \infty)
\end{equation*}
is a direct consequence of Lemma \ref{lemma::smoothening} (b). Let now $d^\prime \geq 1$ and assume that the lemma is already proven for all dimensions $<d^\prime$.

As the local height does not depend on the ambient variety $X$ (apply \cite[Proposition 3.6]{Gubler2003} to the inclusion $Z \hookrightarrow X$), we can assume $X=Z$ and hence $d^\prime = d$. By Hironaka's resolution theorem \cite{Hironaka1964} (see also \cite{Kollar2007}), there always exists a smooth variety $\widetilde{X}$ and a birational, projective morphism $f: \widetilde{X} \rightarrow X$. Again by \cite[Proposition 3.6]{Gubler2003}, we have
\begin{equation*}
\lambda_{\widehat{D}_1^{(n,k)}, \widehat{D}_2^{(n,k)},\dots, \widehat{D}_{d+1}^{(n,k)}}([X]) = \lambda_{f^\ast \widehat{D}_1^{(n,k)}, f^\ast \widehat{D}_2^{(n,k)},\dots, f^\ast\widehat{D}_{d+1}^{(n,k)}}([\widetilde{X}])
\end{equation*}
where $f^\ast \widehat{D}_i^{(n,k)} = (f^\ast \Lo_i^{(n,k)}, f^{-1}(Y_i) \cup |\Div(f^\ast \mathbf{t}_i)|, (f^\ast \mathbf{s}_i)^{\otimes n} \otimes f^\ast \mathbf{t}_i)$. This and the compatibility of Monge-Ampère measures with pullback (see footnote \ref{footnote} on page \pageref{footnote} above) allow us to reduce the proof of the lemma to the case where $X=Z$ is smooth. 

By the induction formula \cite[Proposition 3.5]{Gubler2003}, the local height $\lambda_{\widehat{D}_1^{(n,k)}, \widehat{D}_2^{(n,k)},\dots, \widehat{D}_{d+1}^{(n,k)}}([X])$ equals
\begin{multline*}
\lambda_{\widehat{D}_1^{(n,k)}, \widehat{D}_2^{(n,k)},\dots, \widehat{D}_{d}^{(n,k)}}(\Div(\mathbf{s}_{d + 1}^{\otimes n} \otimes \mathbf{t}_{d + 1})) 
- \int_{X^{\mathrm{an}}_{\IC_\nu}} \log \Vert \mathbf{s}^{\otimes n}_{d+1} \otimes \mathbf{t}_{d + 1} \Vert_{d+1}^{(n,k)} \left(\bigwedge_{i=1}^d c_1(\Lo_{i}^{(n,k)})\right).
\end{multline*}
(Note that the integral is finite by \cite[Théorème 4.1]{Chambert-Loir2009}.)
By our inductive assumption, the double limit
\begin{equation}
\lim_{n \rightarrow \infty} n^{-d} \left(  \lim_{k \rightarrow \infty}
\lambda_{\widehat{D}_1^{(n,k)}, \widehat{D}_2^{(n,k)},\dots, \widehat{D}_{d}^{(n,k)}}(\Div(\mathbf{s}_{d + 1})) \right)
\end{equation}
exists, and
\begin{equation}
\lim_{n \rightarrow \infty} n^{-d+1} \left(  \lim_{k \rightarrow \infty}
\lambda_{\widehat{D}_1^{(n,k)}, \widehat{D}_2^{(n,k)},\dots, \widehat{D}_{d}^{(n,k)}}(\Div(\mathbf{s}^\prime)) \right) = 0.
\end{equation}
To show convergence of \eqref{equation::newdefinition}, it hence suffices to prove that the double limit
\begin{equation}
\label{equation::horrible_doublelimit}
\lim_{n \rightarrow \infty} n^{-(d+1)} \left( \lim_{k \rightarrow \infty} \int_{X^{\mathrm{an}}_{\IC_\nu}} \log \Vert \mathbf{s}^{\otimes n}_{d+1} \otimes \mathbf{t}_{d + 1} \Vert_{d+1}^{(n,k)} \left(\bigwedge_{i=1}^d c_1(\Lo_{i}^{(n,k)})\right) \right)
\end{equation}
exists. For ease of notation, let us write
\begin{align*}
V^{(n,k)} = \ \log \Vert \mathbf{s}_{d+1}^{\otimes n} \otimes \mathbf{t}_{d+1} \Vert_{d+1}^{(n,k)} \ \text{ and } \
V^{(n)} = \ n \cdot \log \Vert \mathbf{s}_{d+1} \Vert_{d+1} + \log \Vert \mathbf{t}_{d+1} \Vert_0.
\end{align*}
Furthermore, we set
\begin{align*}
\mu_l^{(n,k)} = \ &\left(\bigwedge_{i=1}^{d-l} c_1(\Lo_i^{(n,k)}) \right) \wedge \left( \bigwedge_{i=d-l+1}^d (n \cdot c_1(\Lo_i) + c_1(\Mo)) \right)
\end{align*}
for each $l \in \{ 0, \dots, d \}$. To indicate its independence of $k$, we also write $\mu_d^{(n)}$ instead of $\mu_d^{(n,k)}$. We claim that
\begin{equation}
\label{equation::horrible_claim}
\int_{X^{\mathrm{an}}_{\IC_\nu}} V^{(n,k)} \mu^{(n,k)}_{0} - \int_{X^{\mathrm{an}}_{\IC_\nu}} V^{(n)} \mu_d^{(n)} \longrightarrow 0 \ \  (k \rightarrow \infty).
\end{equation}
In fact, this difference equals
\begin{multline*}
\sum_{l=0}^{d-1} \left( \left(\frac{k-1}{k}\right)^{l} \int_{X^{\mathrm{an}}_{\IC_\nu}} V^{(n,k)} \mu_l^{(n,k)} - \left(\frac{k-1}{k}\right)^{l+1}\int_{X^\mathrm{an}_{\IC_\nu}} V^{(n,k)} \mu_{l+1}^{(n,k)} \right) 
\\
+ \left( \left(\frac{k-1}{k}\right)^{d} \int_{X^{\mathrm{an}}_{\IC_\nu}} V^{(n,k)} \mu_d^{(n)} - \int_{X^{\mathrm{an}}_{\IC_\nu}} V^{(n)} \mu_d^{(n)} \right).
\end{multline*}
By Lemma \ref{lemma::smoothening}, we have $|V^{(n,k)} - V^{(n)}| \rightarrow 0$ ($k \rightarrow \infty$) uniformly on $X_{\IC_\nu}^{\mathrm{an}}$. Using Lemma \ref{lemma::chernforms} (d), we deduce that
\begin{equation*}
\left\vert \int_{X^{\mathrm{an}}_{\IC_\nu}} V^{(n,k)} \mu_d^{(n)} - \int_{X^{\mathrm{an}}_{\IC_\nu}} V^{(n)} \mu_d^{(n)} \right\vert \leq
\int_{X^{\mathrm{an}}_{\IC_\nu}} |V^{(n,k)} - V^{(n)}| \mu_d^{(n)} \longrightarrow 0 \ \ (k \rightarrow \infty).
\end{equation*}
This implies that
\begin{equation*}
\left(\frac{k-1}{k}\right)^{d} \int_{X^{\mathrm{an}}_{\IC_\nu}} V^{(n,k)} \mu_d^{(n)} \longrightarrow \int_{X^{\mathrm{an}}_{\IC_\nu}} V^{(n)} \mu_d^{(n)} \ \ (k \rightarrow \infty).
\end{equation*}
For each $l \in \{ 0, \dots, d-1 \}$, the difference
\begin{equation*}
\int_{X^{\mathrm{an}}_{\IC_\nu}} V^{(n,k)} \mu_l^{(n,k)} - \left(\frac{k-1}{k}\right) \int_{X^\mathrm{an}_{\IC_\nu}} V^{(n,k)} \mu_{l+1}^{(n,k)}
\end{equation*}
equals
\begin{equation}
\label{equation::mongeampere_withT}
\int_{X^{\mathrm{an}}_{\IC_\nu}} V \left(\bigwedge_{i=1}^{d-l-1} c_1(\Lo_i^{(n,k)}) \right) \wedge T_i \wedge \left( \bigwedge_{i=d-l+1}^d (n \cdot c_1(\Lo_i) + c_1(\Mo)) \right)
\end{equation}
where 
\begin{equation*}
T_l = c_1(\Lo_{d-l}^{(n,k)}) -  \frac{k-1}{k} (n \cdot c_1(\Lo_{d-l})+c_1(\Mo)).
\end{equation*}
Note that $T_l$ is a closed positive $(1,1)$-current by Lemma \ref{lemma::smoothening} (a). Hence the wedge product of currents in \eqref{equation::mongeampere_withT} is well-defined in the sense of Monge-Ampère measures (see \cite[Subsection 3.1.1]{Guedj2017}), yielding a positive measure on $X^{\mathrm{an}}_{\IC_\nu}$. 

To bound \eqref{equation::mongeampere_withT}, we choose a finite covering $\{ U_{\alpha} \}_{\alpha \in A}$ of $X^{\mathrm{an}}_{\IC_\nu}$ by open sets $U_\alpha$ with the property that, for each $\alpha \in A$, there exist open sets $V_\alpha, W_\alpha$ such that
\begin{enumerate}
	\item[(a)] the topological closure $\overline{U}_\alpha \subset X_{\IC_\nu}^{\mathrm{an}}$ (resp.\ $\overline{V}_\alpha \subset X_{\IC_\nu}^{\mathrm{an}}$) is a compact subset of $V_\alpha$ (resp.\ $W_\alpha$), and
	\item[(b)] there exist non-vanishing sections $\mathbf{s}_{i,\alpha}: W_\alpha \rightarrow (L_i)^{\mathrm{an}}_{\IC_\nu}$, $i \in \{1,\dots,d+1\}$, and $\mathbf{t}_{\alpha}: W_\alpha \rightarrow M^{\mathrm{an}}_{\IC_\nu}$.
\end{enumerate}
On the open subset $V_\alpha \subset W_\alpha$, the functions
\begin{align*}
u_{i,\alpha} = \ &\log \Vert \mathbf{s}_{i,\alpha}^{\otimes n} \otimes \mathbf{t}_{\alpha} \Vert^{(n,k)}_i, \\
v_{i,\alpha} = \ &\log \Vert \mathbf{s}_{i,\alpha}^{\otimes n} \otimes \mathbf{t}_{\alpha} \Vert^{(n,k)}_i - \frac{k-1}{k} (n \log \Vert \mathbf{s}_{i,\alpha} \Vert_i + \log\Vert \mathbf{t}_\alpha \Vert_0), \text{ and} \\
w_{i,\alpha} = \ &n \log \Vert \mathbf{s}_{i,\alpha} \Vert_i + \log\Vert \mathbf{t}_\alpha \Vert_0
\end{align*}
are locally bounded potentials for $c_1(\Lo^{(n,k)}_i)$, $T_i$, and $nc_1(\Lo_i) + c_1(\Mo)$, respectively. They are hence all locally bounded plurisubharmonic functions on each $U_\alpha$. We can bound \eqref{equation::mongeampere_withT} by
\begin{equation}
\label{equation::horrible_integral}
\sum_{\alpha \in A} \int_{\overline{U}_\alpha} |V| dd^c u_{1,\alpha} \wedge \cdots \wedge dd^c u_{d-l-1,\alpha} \wedge dd^c v_{d-l,\alpha} \wedge dd^c w_{d-l+1,\alpha} \wedge \cdots dd^c w_{d,\alpha}.
\end{equation}
Applying the Chern-Levine-Nirenberg inequalities for each $U_\alpha \subset W_\alpha$ (namely, \cite[Theorem 3.14]{Guedj2017}), each integral in the above sum is
\begin{align*}
\ll_{U_\alpha,V_\alpha} \Vert V \Vert_{L^1(V_\alpha)} \left( \prod_{i=1}^{d-l-1} \Vert u_{i,\alpha} \Vert_{L^{\infty}(V_\alpha)} \right) \Vert v_{d-l,\alpha} \Vert_{L^\infty(V_\alpha)} \Vert \left( \prod_{i=d-l+1}^d \Vert w_{i,\alpha}\Vert_{L^\infty(V_\alpha)} \right)
\end{align*}
where the norms $\Vert \cdot \Vert_{L^1(V_\alpha)}$ and $\Vert \cdot \Vert_{L^\infty(V_\alpha)}$ are with respect to $c_1(\Mo)^{\wedge d}$ (or, equivalently, the measure induced by any other Kähler form on $X_{\IC_\nu}^{\mathrm{an}}$). (Note that $\Vert V \Vert_{L^1(V_\alpha)}$ is finite as it is locally the difference of two plurisubharmonic functions, which are locally integrable by \cite[Proposition 1.34]{Guedj2017}.)

Using Lemma \ref{lemma::smoothening} (b), we obtain an upper bound
\begin{equation*}
\Vert u_{i,\alpha} \Vert_{L^\infty(V_\alpha)} \leq \Vert w_{i,\alpha} \Vert_{L^\infty(V_\alpha)} \ll_{\mathbf{s}_{i,\alpha},\mathbf{t}_{\alpha},U_\alpha,V_\alpha,n} 1 
\end{equation*}
for each $i \in \{ 1, \dots, d\}$ and each $\alpha \in A$. In addition, we have 
\begin{equation*}
|v_{i,\alpha}| = \log \left( \frac{\Vert \mathbf{s}_{i,\alpha}^{\otimes n} \otimes \mathbf{t}_\alpha \Vert_{i}^{(n,k)}} {\Vert \mathbf{s}_{i,\alpha}\Vert_i^n \Vert \mathbf{t}_\alpha \Vert_0} \right) + \frac{w_{i,\alpha}}{k} \longrightarrow 0 \ \ (k \rightarrow \infty)
\end{equation*}
uniformly on $V_\alpha$ because of Lemma \ref{lemma::smoothening} (b) and the boundedness of $w_{i,\alpha}$ on $\overline{V}_\alpha \subset W_\alpha$. Combining all these estimates, we obtain that \eqref{equation::horrible_integral} converges to $0$ as $k \rightarrow \infty$. This confirms our claim \eqref{equation::horrible_claim}. Unraveling notations, we conclude that the inner limit ($k\rightarrow \infty$) in \eqref{equation::horrible_doublelimit} exists and equals
\begin{equation*}
\int_{X^{\mathrm{an}}_{\IC_\nu}} V \mu_d=\int_{X^{\mathrm{an}}_{\IC_\nu}} (n \log \Vert \mathbf{s}_{d + 1} \Vert_{d + 1} + \log \Vert \mathbf{t}_{d + 1} \Vert_{0}) \left(\bigwedge_{i=1}^d n \cdot c_1(\Lo_{i}) + c_1(\Mo)\right).
\end{equation*}
Rearranging, we see that this integral equals
\begin{equation*}
n^{d+1}\int_{X_{\IC_\nu}^{\mathrm{an}}} \log \Vert \mathbf{s}_{d + 1} \Vert_{d + 1} c_1(\Lo_1) \wedge c_1(\Lo_2) \wedge \cdots \wedge c_1(\Lo_d) + O(n^{d})
\end{equation*}
and the convergence of the double limit \eqref{equation::newdefinition} follows immediately. The induction formula \eqref{equation::induction_formula_2} is also a consequence of our proof. Uniqueness (i.e., independence of $\overline{M}$, its sections $\mathbf{t}_i$ and the $\nu$-metrics $\Vert \cdot \Vert^{(n,k)}_i$) follows inductively.
\end{proof}

Finally, let us note that our definition of $\lambda_{\widehat{D}_1,\widehat{D}_2,\cdots,\widehat{D}_{d^\prime+1}}(\mathfrak{Z})$ by a limit process as in \eqref{equation::newdefinition} also allows us to observe that \cite[Propositions 3.4, 3.5, 3.6, 3.7, 3.8 and Theorem 10.6]{Gubler2003} remain true in our more general setting. In particular, we realize retrospectively that the induction formula of \cite[Proposition 3.5]{Gubler2003} can be also used as a straightforward definition of the local heights $\lambda_{\widehat{D}_1,\widehat{D}_2,\cdots,\widehat{D}_{d^\prime+1}}(\mathfrak{Z})$ for general semipositive $\nu$-metrized line bundles on $X$. 

\section{An elementary computation}
\label{appendixB}

For convenience, we give the details of the elementary computation used in the proof of Lemma \ref{lemma::measure_realanalytic}.

\begin{lemma} For every $h \in \mathscr{C}^0_c(\IC^\times)$, we have
\begin{equation*}
\int_{\IC^\times} h(z) dd^c \left\vert \log |z| \right\vert = \int_{[0,2\pi]} h(e^{i\phi})\frac{d\phi}{\pi}.
\end{equation*}
\end{lemma}

Note that the measure $d\phi/\pi$ on the right-hand side assigns volume $2$ (and not $1$) to $[0,2\pi]$.

\begin{proof}
By the definition of the Bedford-Taylor measure $dd^c \left\vert \log |z| \right\vert$, it suffices to prove that
\begin{equation}
\label{equation::integralappB}
\int_{\IC^\times} \left\vert \log |z| \right\vert dd^c h = \int_{[0,2\pi]} h(e^{i\phi})\frac{d\phi}{\pi}
\end{equation}
for every test function $h \in \mathscr{C}^\infty_c(\IC^\times)$. Using polar coordinates $z = r e^{i\phi}$, we have
\begin{equation*}
dd^c h = \frac{1}{2\pi}\left( r \frac{\del^2 h}{\del r^2} + \frac{\del h}{\del r} + \frac{1}{r} \frac{\del^2 h}{\del \phi^2} \right) dr \wedge d\phi,
\end{equation*}
and the left-hand side of \eqref{equation::integralappB} equals
\begin{equation}
\label{equation::integralappB2}
\frac{1}{2\pi} \int_0^{2\pi} \left( \int_{1}^\infty + \int_{1}^0 \right) \left( \log(r)r \frac{\del^2 h}{\del r^2} + \log(r) \frac{\del h}{\del r}+ \frac{\log(r)}{r}\frac{\del^2 h}{\del \phi^2}\right) dr \wedge d\phi.
\end{equation}
Integration by parts yields
\begin{equation*}
\left( \int_{1}^\infty + \int_{1}^0 \right) \log(r)r \frac{\del^2 g}{\del r^2}dr = 2 g(1) + \left( \int_{1}^\infty + \int_{1}^0 \right) \frac{gdr}{r}
\end{equation*}
and
\begin{equation*}
\left( \int_{1}^\infty + \int_{1}^0 \right) \log(r) \frac{\del g}{\del r}dr = - \left( \int_{1}^\infty + \int_{1}^0 \right) \frac{gdr}{r}
\end{equation*}
for any smooth function $g \in \mathscr{C}^\infty_c(\IR^{>0})$.  As $\int_{0}^{2\pi} \del^2 h/ \del \phi^2 = [\del h/\del \phi]^{2\pi}_0 = 0$, it follows that \eqref{equation::integralappB2} equals $\int_{0}^{2\pi}h(e^{i\phi})d\phi/\pi$.
\end{proof}
\end{appendix}

\textbf{Acknowledgements:} The author thanks Yuri Bilu, Antoine Chambert-Loir, Walter Gubler, Philipp Habegger, Klaus Künnemann, Umberto Zannier and Shou-Wu Zhang for discussions and encouragement. In addition, he thanks the participants of the Oberseminar on Arakelov Theory at the University of Regensburg for pointing out countless inaccuracies and typos in a previous version of this text. This research project started while the author was enjoying the hospitality of the Max Planck Institute for Mathematics and continued at the Fields Institute during its ``Thematic Program on Unlikely Intersections, Heights, and Efficient Congruencing''. The author thanks both institutions for their support. Finally, he thanks the referees for their ample suggestions, which helped to improve the expository quality of this article.

\bibliographystyle{plain}
\bibliography{../Bibliography/references}

\end{document}